\documentclass[12pt, twoside, a4paper]{article}

\usepackage{amsthm}
\usepackage{amssymb}
\usepackage{amsmath}
\usepackage{mathtools}
\usepackage{stmaryrd}

\usepackage[english]{babel}
\usepackage[utf8]{inputenc}

\usepackage[T1]{fontenc}

\usepackage{enumitem}
\setlist{nosep} 

\usepackage{tikz-cd}

\usepackage{xcolor,colortbl}

\usepackage{adjustbox}

\usepackage{caption} 

\usepackage{titletoc}
\dottedcontents{section}[1.8em]{}{1.7em}{1pc}

\usepackage{hyperref}
\hypersetup{pdfborder=0 0 0}

\usepackage[all]{xy}

\usepackage{datetime}
\newdateformat{monthyeardate}{
  \monthname[\THEMONTH] \THEYEAR}

\newtheorem{theorem}{Theorem}[section]
\newtheorem{definition}[theorem]{Definition}
\newtheorem{proposition}[theorem]{Proposition}
\newtheorem{remark}[theorem]{Remark}
\newtheorem{lemma}[theorem]{Lemma}

\newtheorem{corollary}[theorem]{Corollary}

\newcommand{\Z}{\mathbb{Z}}
\newcommand{\C}{\mathbb{C}}

\newcommand{\R}{\mathbb{R}}
\newcommand{\F}{\mathbb{F}}

\DeclareMathOperator{\ez}{EZ}
\DeclareMathOperator{\id}{Id}
\DeclareMathOperator{\ev}{ev}
\DeclareMathOperator{\Set}{Set}
\DeclareMathOperator*{\colim}{colim}
\DeclareMathOperator{\Fun}{Fun}
\DeclareMathOperator{\Hom}{\mathrm{Hom}}

\DeclareMathOperator{\tot}{\mathrm{Tot}}
\DeclareMathOperator{\sing}{\mathrm{Sing}}

\renewcommand{\top}{\mathcal{T}\!\!op}
\newcommand{\un}[1]{\underline{#1}}
\newcommand{\ov}[1]{\overline{#1}}
\newcommand{\Ch}{\mathcal{C}\hspace{-0.04cm}h}

\newcommand{\lef}[1]{\prescript{#1}{} \!\!\> }
\newcommand{\llef}[1]{\mathstrut_{#1} }
\newcommand{\bu}{\bullet}
\newcommand{\lb}{\lef \bullet}

\newcommand{\CF}[2]{\lb \mathcal{F}(#1,#2)^\bu}
\renewcommand{\AA}{\mathcal{A}}
\newcommand{\BB}{\mathcal{B}}
\newcommand{\CC}{\mathcal{C}}
\newcommand{\DD}{\mathcal{D}}
\newcommand{\PP}{\mathcal{P}}
\newcommand{\GG}{\mathcal{G}}
\newcommand{\pGG}{\mathcal{G}^\textrm{pre}}
\newcommand{\FF}{\mathcal{F}}
\newcommand{\bM}{\mathcal{M}}
\newcommand{\MM}{\raisebox{0pt}[0pt][0pt]{$\overline{\mathcal{M}}$}}

\newlength{\hnor}
\newlength{\dnor}
\newlength{\currenttextsize}
\makeatletter
\newcommand{\fixhd}[1]{
  \setlength{\currenttextsize}{\f@size pt}
  \raisebox{0pt}[0pt][0pt]{\fontsize{\currenttextsize}{1cm}\selectfont $#1$}
}
\makeatother
\newcommand{\sfixhd}[1]{
  \raisebox{0pt}[0pt][0pt]{\scriptsize $#1$}
}

\newcommand{\tl}[2]{{\lef{#1} \fixhd{#2}}}
\newcommand{\bl}[2]{{\llef{#1} \fixhd{#2}}}
\newcommand{\tr}[2]{{\fixhd{#1}^{#2}}}
\newcommand{\br}[2]{{\fixhd{#1}_{#2}}}
\newcommand{\tltr}[3]{{\lef{#1} \fixhd{#2}^{#3}}}
\newcommand{\bltr}[3]{{\llef{#1} \fixhd{#2}^{#3}}}
\newcommand{\tlbr}[3]{{\lef{#1} \fixhd{#2}_{#3}}}
\newcommand{\blbr}[3]{{\llef{#1} \fixhd{#2}_{#3}}}
\newcommand{\tlar}[4]{{\lef{#1} \fixhd{#2}^{#3}_{#4}}}

\newcommand{\tltltr}[4]{\tltr{\substack{{#1}\vspace{-0.7mm}\\{#2}\vspace{0.3mm}}}{{#3}}{{#4}}}
\newcommand{\tltlar}[5]{\tlar{\substack{{#1}\vspace{-0.7mm}\\{#2}\vspace{0.3mm}}}{{#3}}{{#4}}{{#5}}}

\newcommand{\stl}[2]{{\lef{#1} \sfixhd{#2}}}
\newcommand{\sbl}[2]{{\llef{#1} \sfixhd{#2}}}
\newcommand{\str}[2]{{\sfixhd{#1}^{#2}}}

\newcommand{\stltr}[3]{{\lef{#1} \sfixhd{#2}^{#3}}}
\newcommand{\sbltr}[3]{{\llef{#1} \sfixhd{#2}^{#3}}}

\newcommand{\cst}{\mathrm{cst}}

\newcommand{\fwcs}{\otimes_{\textrm{f}}^{\textrm{cs}}}
\newcommand{\fwlp}{\otimes_{\textrm{f}}^{\textrm{lp}}}
\newcommand{\fwrp}{\otimes_{\textrm{f}}^{\textrm{rp}}}
\newcommand{\fwbp}{\otimes_{\textrm{f}}^{\textrm{bp}}}
\newcommand{\fwq}{\otimes_{\textrm{f}}^{\textrm{?}}}

\newcommand{\bpb}[1]{\tltr{\bu}{P #1}{\bu}}
\newcommand{\sbpb}[1]{\stltr{\bu}{P #1}{\bu}}

\newcommand{\und}[1]{\underset{\makebox[0pt]{\scriptsize $#1$}}}
\newcommand{\undL}[2]{\overset{\makebox[0pt]{\scriptsize L}}{\underset{\makebox[0pt]{\scriptsize $#1$}} #2}}

\newcommand{\norm}[1]{\lVert #1 \rVert}

\newcommand{\ug}{\textrm{ug}}

\DeclarePairedDelimiter\pare{(}{)}
\DeclarePairedDelimiter\inner{\langle}{\rangle}

\renewcommand{\epsilon}{\varepsilon}

\begin{document}
\title{Local systems and vanishing Maslov class}
\author{Axel Husin, Thomas Kragh}
\date{\monthyeardate\today}
\maketitle

\begin{abstract}
  It is well known that closed exact Lagrangians in cotangent bundles of closed manifolds have vanishing Maslov class and are homotopy equivalent to the zero section. In this paper we greatly simplify the proof of vanishing Maslov class and generalize the proof to a slightly larger family of Weinstein domains.
\end{abstract}

\tableofcontents

\section{Introduction}
In this paper we give a short proof of the fact that a closed exact Lagrangian $L$ in the cotangent bundle $T^*X$ of a closed and connected manifold have vanishing Maslov class and is homology equivalent to $X$. This was originally proved by the second author of this paper and Abouzaid in \cite{MySympfib} and \cite{Abou1}. An alternate proof using micro local sheaf theory was given in \cite{Guillermou1}. The proof in this paper combine a few new ideas with ideas from those papers and ideas from \cite{MR2373371}, \cite{Kragh2016} and \cite{FSS}. Our proof even generalizes the statement to slightly more general Weinstein domains $W_X \supset D^*X$ as in Theorem~\ref{thm:article:1} below.

Let $X$ be a closed connected smooth manifold of dimension $n$. Let $W_X$ be a Weinstein domain obtained by inductively attaching sub-critical Weinstein handles to $D^*X$ over a small contractible patch in $X$. That is we fix a small $D^n \cong A \subset X$ and define $W_X$ by inductively attaching the Weinstein handles to $D^*X$ with attaching spheres over $S^*A$ (in the inductive process we allow later handles to attach onto the previously attached handles and the part of $S^*A$ that remains as boundary). The assumption that $A$ is contractible implies $W_X \simeq X \vee B$ where $B$ is the underlying CW complex defined by the attaching (after collapsing $D^*A$ to a point). In particular, we have a retraction
\begin{align*}
  r : W_X \to X
\end{align*}
of the inclusion $X \subset W_X$. We also assume that the standard quadratic complex volume form on $D^*X$ extends over $W$. In order to avoid certain technicalities we also assume that $W_X=D^*X$ when $\dim X \leq 5$ (but we don't believe this to be essential).

\begin{theorem} \label{thm:article:1}
  For any closed exact Lagrangian $L \subset W_X$ the composition $L \subset W_X \xrightarrow{r} X$ is a homology equivalence. Furthermore, the Maslov class of $L$ vanishes for any extension of the quadratic complex volume form.
\end{theorem}

In a follow up paper the first author of this paper will construct counter examples to the above statement where the assumption that the handles are sub-critical is removed.

\subsection*{Acknowledgments}

The second author would like to thank Mohammed Abouzaid for insightful discussions surrounding the material.

\section{Notation and conventions}\label{sec:notation-conventions}

We will assume that $X$ is a connected smooth closed manifold of dimension at least 6.  The non-connected cases and statements are easy to infer from the connected case. To prove the theorem in the lower dimension (where we assumed $W_X=D^*X$) we may simply replace $X$ by $X\times S^6$ and $L\subset T^*X$ by $L \times S^6 \subset T^*X \times T^* S^6$.

We will let $(W, \omega=d\lambda)$ denote \emph{any} compact Liouville domain (with contact boundary). We let $W_X$ denote the Weinstein domain obtained from $X$ as described in the introduction, on $D^*X$ we use the tautological one-form $\lambda=-pdq$. We will give more details on the construction of $W_X$ in Section~\ref{sec:constr-w_x-skel}.

We will generally use homological conventions, which means that
\begin{itemize}
\item Differentials decrease degrees and action.
\item All filtrations are increasing (they are sequential colimits).
\item The total complex of a bi-graded complex is given by direct sums.
\item Our Floer complexes are defined using tensors of local systems and not homomorphism of such. 
\item Our version of $\mu_2$ is a co-product not a product. It also differs from simply being the dual of the standard $\mu_2$ as it is a generalization to more general local systems and thus takes on a slightly different form.
\end{itemize}

Let $\F$ be any field. We denote the category of chain complexes over $\F$ by $\Ch_\F$. When $\F$ has characteristic 2 we denote the category of ``ungraded chain complexes'' by $\Ch_{\F}^{\ug}$. This has objects given by vector spaces $V$ over $\F$ together with a differential $d:V \to V$ such that $d^2=0$. The morphisms in this category are also chain maps - i.e. maps $V \to V$ commuting with $d$.

In this paper we generally let $\Ch$ denote either of these two categories. We will refer to the two cases as the \emph{graded} case or the \emph{ungraded} case. There is a forgetful functor $\Ch_{\F} \to \Ch_{\F}^{\ug}$, which takes the direct sum of all degrees. Even though the ``g'' in dga means graded, we will even in the ungraded case use the terminology dga and dg-module for differential \emph{un}graded algebras and differential \emph{un}graded modules.

The tensor product in the graded case is the usual graded tensor product and the tensor product in the ungraded case is the usual tensor product of vector spaces with the differential $\partial(x\otimes y)=\partial_1x\otimes y+x\otimes \partial_2y$ (note that signs are irrelevant as we assumed characteristic 2 in the ungraded case).

In the graded case let $\Sigma^n\F$ denote the chain complex supported in degree $n$ and with the $n$th vector space given by $\F$. We use the suspension functor $\Sigma^n(-)=(\Sigma^n\F)\otimes (-)$ which raises degree by $n$ and changes $\partial$ to $(-1)^n\partial$. For any $\F$ vector space $M$ we let $M$ also denote the chain complex with zero differential, in the graded case we put $M$ in degree 0. When suspensions and tensor factors switch places we have the Koszul sign rule. For instance we will often need the isomorphism $\Sigma^{k} A\otimes \Sigma^{l} B\cong \Sigma^{k+l} A\otimes B$ sending $a\otimes b$ to $(-1)^{|a|l}a\otimes b$ where $|a|$ is the grading of $a\in A$, the sign comes from moving the suspension $\Sigma^l\F$ past $A$. We also have natural isomorphisms $\Sigma^{k}\Sigma^{l}A\to \Sigma^{k+l}A$ sending $1\otimes 1\otimes a$ to $1\otimes a$ and $\Sigma^k\Sigma^lA\to \Sigma^l\Sigma^k A$ sending $1\otimes 1\otimes a$ to $(-1)^{kl}1\otimes 1\otimes a$. However, be warned that the composition $\Sigma^k\Sigma^lA\to \Sigma^l\Sigma^kA\to \Sigma^{k+l}A$ is \emph{not} the same as the map $\Sigma^k\Sigma^l A \to \Sigma^{k+l}A$ when $kl$ is odd. In the ungraded case the suspension functor is defined to be the identity functor.

We will often discuss spectral sequences and in the ungraded case we do not want to consider ungraded spectral sequence. So to avoid this we identify the $\Ch_{\F}^{\ug}$ with the subcategory of $\Ch_\F$ given by 1-periodic graded chain complexes
\begin{align*}
  \cdots \leftarrow V \xleftarrow{d} V \xleftarrow{d} V \leftarrow \cdots
\end{align*}
and 1-periodic maps. This way both input and output of the spectral sequences can be treated as if they are graded. As all spectral sequences considered in the article are half plane with exiting differentials with Hausdorff and exhaustive filtrations, it follows from Theorem 6.1 in \cite{MR1718076} that these are strongly convergent.

Be warned, however, that considering $\Ch_{\F}^{\ug}$ as periodic chain complexes is not compatible with the tensor product.

We will often work with categories $\CC$ that has a notion of morphisms being weak equivalences. In $\Ch$ this is given by quasi isomorphisms. This will always be defined by a sub category $\CC_{\simeq} \subset \CC$ containing all objects and isomorphisms. This induces an equivalence relations on the objects of $\CC$, we say that two objects are \emph{weakly equivalent} if there exists a zig-zag of weak equivalences between them. We will write $A\simeq B$ if $A$ and $B$ are weakly equivalent.

We will call a functor $F : \CC \to \DD$ between such categories \emph{homotopical} if it maps weak equivalences to weak equivalences. In this case we call it
\begin{itemize}
\item \emph{$\pi_0$-injective} if it is injective on weak equivalence classes.
\item \emph{$\pi_0$-surjective} if it is surjective on weak equivalence classes.
\item \emph{$\pi_0$-equivalence} if it is a bijection on weak equivalence classes.
\end{itemize}

We will not need the full theory of model categories although many aspects of it is implicitly used. E.g. we will be using the notion of semi-free (recalled below), which could be a definition of cofibrant objects.

\subsection{Modules over a (possibly ungraded) dga}\label{sec:modules-over-dga}

If $\AA$ is a dga in $\Ch$ ($\Ch_\F$ or $\Ch_{\F}^{\ug}$) we may consider the category of left $\AA$-modules where the objects are given by $M$ in $\Ch$ together with an associative action $\AA \otimes M \to M$. A weak equivalence is a map of modules $M \to N$ which is also a quasi isomorphism. The bar resolution
\begin{align*}
  0 \leftarrow M \leftarrow  \AA \otimes M \leftarrow \cdots \leftarrow  \AA^{\otimes p} \otimes M \leftarrow \cdots
\end{align*}
is an acyclic complex of modules, and thus removing the last $M$ and taking total complex provides a left module, which we denote $\FF M$, weakly equivalent to $M$. It has a filtration $F^p\FF M$ satisfying $F^{-1} \FF M = 0$ and $F^\infty \FF M = N$ such that the successive quotients
\begin{align*}
  F^p \FF M / F^{p+1} \FF M= \AA^{\otimes p} \otimes M
\end{align*}
are free modules. We call any module $N$ admitting such a filtration with free filtration quotients satisfying $F^{-1} N = 0$ and $F^\infty  N = N$ a \emph{semi-free} module.

This means that one may in the usual way prove that any quasi isomorphism $f:\AA \to \BB$ of dgas induces a $\pi_0$-equivalence of left modules categories over the dgas, the equivalence is given by transport of coefficients (pullback) $f^*:\BB\text{-Mod}\to \AA\text{-Mod}$ and derived tensor $\BB \otimes_{\AA}^L(-):\AA\text{-Mod}\to \BB\text{-Mod}$. The proof uses the spectral sequences from the filtrations considered above (where ungraded complexes are treated as periodic). Note that it is very important that the filtration of the bar complex starts at $0$ to get a half plane spectral sequence.

We may even consider both left and right modules over $\AA$, and generally these categories can be very different. However in the cases we care about where $\AA=C_*\Omega X$ is chains on a loop space, we have an antipode which reverses the direction of the loop making the categories of left and right modules equivalent. We also have a diagonal map of dgas
\begin{align*}
  \Delta : C_*\Omega X \to C_*\Omega X \otimes C_*\Omega X
\end{align*}
given by the diagonal and the Alexander-Whitney map. This makes it possible to define a left $C_*\Omega X$ module structure on the tensor product over $\F$ of two left modules. As we will do this more generally in the case we care about in the next section we omit any details here.

It will, however, in this construction be important for this and other constructions that the Eilenberg-Zilber and Alexander-Whitney maps satisfy the commutativity relation given by the following commuting diagram
\begin{align} \label{eq:article:7}
  \begin{split}
  \xymatrix{
    C_*(X) \otimes C_*(Y) \ar[d]^{\ez} \ar[rr]^-{\Delta\otimes \Delta} && C_*(X)^{\otimes 2} \otimes C_*(Y)^{\otimes 2} \ar[d]^{(\ez \otimes\ez) \circ \tau}\\
    C_*(X \times Y) \ar[rr]^{\Delta} && C_*(X \times Y)^{\otimes 2}
  }    
  \end{split}
\end{align}
where $\tau$ is the signed interchange of the two middle factors.

\section{Path local systems} \label{sec:path-local-systems}

As in \cite{Husin_2023} and similar to what was considered in \cite{Abou2} and \cite{MR2373371} we will consider so-called path local systems on our manifolds as one of our models for generalized (sometimes called derived) local systems. As $X$ is connected these path local systems will be equivalent to modules over the dga of chains on the based loop space $C_*\Omega X$, and can be considered a base-point free version of this category. The equivalence is given by restriction to the fibers over the base-point, this is proved in Lemma~\ref{lem:article:15}. Everything in this section works for any connected space $X$.

Let $\tltr{x}{\PP X}{y}$ denote the space of Moore paths in $X$ from $x$ to $y$. That is the space of pairs $(\gamma,r)$ with $\gamma:[0,\infty) \to X$ continuous and $\gamma(0)=x$ and $\gamma(t)=y$ for $t\geq r$. We topologize this as a subset of a mapping space times $\R$. We let $\tltr{x}{PX}{y}=C_*(\tltr{x}{\PP X}{y})$ denote the graded or ungraded (two different cases) chain complex of singular simplicial chains with coefficients in $\F$. We have strictly associative maps
\begin{align*}
  \tltr{x}{PX}{y} \otimes \tltr{y}{PX}{z} \to \tltr{x}{PX}{z}
\end{align*}
given by the Eilenberg-Zilber map and concatenation. We use standard product notation $\alpha\beta$ for this map. We let $\cst_x \in \tltr{x}{PX}{x}$ denote the constant path with 0 parametrization (the unit for concatenation) at $x$.

\begin{definition}
  A left path local system $\tl{\bu}{D}$ on a connected manifold $X$ assigns to each $x\in X$ a fiber $\tl{x}{D}\in \Ch$ and to each pair $x,y\in X$ a chain map (called parallel transport) $\tltr{x}{PX}{y} \otimes \tl{y}{D}\to \tl{x}{D}$, again we use standard action notation $\alpha d$ for this map, and these should satisfy $\cst_xd=d$ and $\beta(\alpha d)=(\beta \alpha)d$. The category of left path local systems on $X$ is denoted by $\tl{X}{\GG}$. Morphisms are fiber-wise chain maps commuting with the path action.
\end{definition}

We could of course in the same way define left path local systems on non connected manifolds, however to have an equivalence to $C_*\Omega X$ modules and to avoid having to care about several components we assume that $X$ is connected. Also since $X$ is connected it is easy to verify that all fibers are homotopy equivalent.

Note that we think of the above as parallel transport \emph{backwards} along the path. The category of right path local systems $\tr{D}{\bu}$ on $X$ is defined similarly, but using the more natural forward parallel transport along each path, and is denoted by $\tr{\GG}{X}$. We will even consider the category $\tltr{X}{\GG}{Y}$ of bi-path local systems $\tltr{\bu}{D}{\bu}$ which assigns fibers $\tltr{x}{D}{y}$ in $\Ch$ for $x\in X$, $y\in Y$ with \emph{commuting} left and right actions by paths in $X$ and $Y$ respectively.

A natural example of such an object is $\bpb{X}\in \tltr{X}{\GG}{X}$, which we will see corresponds to $C_*\Omega X$ as a bi-module over itself. Another important example is $\tl{\bu}{\F}$ defined by $\tl{x\>}{\F}=\F$ and with the augmentation action using the canonical argumentation $\tltr{x}{PX}{y} \to \F$ sending every zero simplex to $1$. In general the classical local systems of (possibly graded) vector spaces embed into $\tl{X}{\GG}$ by using the maps
\begin{align*}
  \tltr{x}{PX}{y} \to \Z[\pi_0(\tltr{x}{\PP X}{y})]
\end{align*}
where the latter has 0 differential and acts on the fibers of the classical local system.

We have an isomorphism of categories $\tau : \tl{X}{\GG} \cong \tr{\GG}{X}$ by simply reversing the paths before acting, i.e. $\tr{\tau(D)}{x}=\tl{x}{D}$ and the parallel transport in $\tau(D)$ is defined as the composition
\begin{align*}
  \tr{\tau(D)}{x}\otimes \tltr{x}{PX}{y}\to \tltr{y}{PX}{x} \otimes \tl{x}{D} \to \tl{y}{D} = \tr{\tau(D)}{y}
\end{align*}
where the first map swaps the tensor factors and reverses the paths (recall the sign when swapping the factors), and the second map is the parallel transport in $D$. We will simply write this as $\tau(\tl{\bu}{D}) = \tr{\overline{D}}{\bu}$. Similarly, we use the notation for $\tl{\bu}{\overline{D}} \in \tl{X}{\GG}$ when $\tr{D}{\bu} \in \tr{\GG}{X}$.

A \emph{weak equivalence} in any of these three categories is a morphism which induces a quasi isomorphism on all fibers. Note that as $X$ is connected it is easy to use the parallel transport to prove that it is enough to be a quasi isomorphism on a single fiber in order to be a weak equivalence.

There is an internal fiber-wise tensor product $\fwq$ on each of these categories, here $?$ is either $\mathrm{lp}$ (left path), $\mathrm{rp}$ (right path) or $\mathrm{bp}$ (bi path). E.g. for objects $\tl{\bu} D,\tl{\bu} E \in \tl{X}{\GG}$ we write this as $\tl{\bu}{D} \fwlp \tl{\bu}{E}$ or sometimes  $\tl{\bu}{(D\fwlp E)}$, its fibers are given by
\begin{align*}
  \tl{x}{(D \fwlp E)} = \tl{x}{D}\otimes \tl{x}{E}.
\end{align*}
The path action is defined using the Alexander-Whitney diagonal map $\Delta$, which in the case of left path local systems means:
\begin{align*}
  \tltr{x}{PX}{y} \otimes (\tl{y}D \otimes \tl{y}E) & \xrightarrow{\Delta\otimes \id} \tltr{x}{PX}{y} \otimes \tltr{x}{PX}{y} \otimes (\tl{y}D \otimes \tl{y}E) \cong \\
  & \cong   (\tltr{x}{PX}{y} \otimes \tl{y}D) \otimes (\tltr{x}{PX}{y} \otimes \tl{y}E) \to \tl{x}D \otimes \tl{x}E.
\end{align*}
Similarly we define $\tr{(D \fwrp E)}{\bu}$ for $\tr{D}{\bu},\tr{E}{\bu} \in \tr{\GG}{X}$, and even $\tltr{\bu}{(D\fwbp E)}{\bu}$ for $\tltr{\bu}{D}{\bu}, \tltr{\bu}{E}{\bu} \in \tltr{X}{\GG}{Y}$. The associativity for these tensor local systems follows by the EZ and AW compatibility described in Equation~(\ref{eq:article:7}). It is also clear that the fiber-wise tensor is homotopical in the sense that if $\tl \bu D\simeq \tl \bu D'$ and $\tl \bu E\simeq \tl \bu E'$ are weakly equivalent then $\tl{\bu}{D} \fwlp \tl{\bu}{E}\simeq\tl{\bu}{D'} \fwlp \tl{\bu}{E'}$.

There is also a \emph{global} tensor product defined as follows. If $\tr{D}{\bu}\in \tr{\GG}{X}$ and $\tl{\bu} E \in \tl{X}{\GG}$ then we can define the tensor product over $\tltr{\bu}{PX}{\bu}$ as the chain complex
\begin{align*}
  \tr{D}{\bu} \und{PX}{\otimes} \tl{\bu} E=\left(\bigoplus_{x\in X} \tr{D}{x} \otimes \tl{x}{E}\right)\Big /\left\langle d\alpha\otimes e-d\otimes \alpha e \right\rangle
\end{align*}
with $d \in D^x, e\in \tl{y}{E}$ and $\alpha \in \tltr{x}{PX}{y}$. Tensoring with $\tl{\bu} PX^\bu$ is the identity
\begin{align*}
  \tl{\bu} PX^\bu \und{ PX}\otimes \tl{\bu} D \cong \tl{\bu} D
\end{align*}
where the isomorphism sends $\alpha\otimes d$ to $\alpha d$. The global tensor is very useful in terms of Morita theory, indeed, given a $\tltr{\bu}{D}{\bu} \in \tltr{X}{\GG}{Y}$ we can define the functor
\begin{align*}
  \tltr{\bu}{D}{\bu} \und{PY} \otimes (-) : \tl{Y}{\GG} \to \tl{X}{\GG}.
\end{align*}
Be warned that the global tensor is not homotopical, however we will take care of this later when defining the derived global tensor.

\begin{lemma} \label{lem:article:15}
  The functor sending a left path local system on $X$ to the fiber over a chosen base point $x\in X$ is homotopical and defines a $\pi_0$-equivalence between $\tl{X}{\GG}$ and left modules over $C_*\Omega X$. Similarly, $\tr{\GG}{X}$ and $\tltr{X}{\GG}{X}$ are $\pi_0$-equivalent to right- and bi-modules respectively.

  Furthermore, these functors send fiber-wise tensor to the tensor over $\F$ of modules.
\end{lemma}

\begin{proof}
  We only consider the case $\tl{X}{\GG}$ as the other cases are similar. Clearly the functor restricting to a fiber is homotopical. The functor back is defined by sending a left module $M$ to the left path local system
  \begin{align*}
    \tltr{\bu}{PX}{x} \und{C_*\Omega X}\otimes M.
  \end{align*}
  For $M\in C_*\Omega X\text{-Mod}$, we have a natural isomorphism $\tltr{x}{PX}{x} \und{C_*\Omega X}\otimes M\cong M$ and for $\tl{\bu}{D} \in \tl{X}{\GG}$ we have a natural map defined by parallel transport from the composition of the functors to the identity
  \begin{align*}
    \tltr{\bu}{PX}{x} \und{C_*\Omega X}\otimes \tl{x}{D} \to \tl{\bu}{D}.
  \end{align*}
  This is an isomorphism over the fiber $x$ and hence a weak equivalence of left path local systems. It is clear that restricting path local systems to the fiber $x$ commutes with the tensor products.
\end{proof}

\begin{corollary} \label{cor:article:6}
  In the graded case, if a path local system has homology support in a single degree, then it is weakly equivalent to the classical local system defined by its homology.
\end{corollary}

\begin{proof}
  By the lemma above we may prove this for a module $M$ over $C_*\Omega X$ and by shifting we assume that the homology is supported in degree $0$. First we replace $M$ with the weakly equivalent submodule
  \begin{align*}
    \cdots \rightarrow M_i \rightarrow \cdots\rightarrow  M_1 \rightarrow  \ker d_0 \rightarrow 0.
  \end{align*}
  This is a well defined sub-module as the zero simplices in $C_0\Omega X$ are all closed and hence preserve $\ker d_0$. This submodule has a module quotient map to the homology local system, which is a weak equivalence.
\end{proof}

\begin{corollary} \label{cor:article:5}
  The pullback functor for any homotopy equivalence $f : X \to Y$ induces a $\pi_0$-equivalence $f^* : \tl{Y}{\GG} \to \tl{X}{\GG}$ (and similar for the other versions). 
\end{corollary}

\begin{proof}Choose a basepoint $x\in X$ and let $f(x)\in Y$ be the basepoint in $Y$. We have a commutative diagram of functors 
  \begin{center}
    \begin{tikzcd}
      \tl Y\GG\ar{r}{f^*}\ar[swap]{d}{\text{res}}&\tl X\GG\ar{d}{\text{res}}\\
      C_*\Omega Y\text{-Mod}\ar{r}{f^*}&C_*\Omega X\text{-Mod}
    \end{tikzcd}
  \end{center}
  The vertical arrows restricting a path local system to a single fiber are $\pi_0$-equivalences by Lemma~\ref{lem:article:15}, and the lower horizontal arrow is a $\pi_0$-equivalence by the discussion in Section~\ref{sec:modules-over-dga}. It follows that the top horizontal arrow $f^*:\tl Y\GG\to \tl X\GG$ is also a $\pi_0$-equivalence.
\end{proof}

We say that a left path local system $\tl{\bu} D$ is free if it is isomorphic to a (possibly infinite) direct sum with summands of the form $\Sigma^n(\tl{\bu} PX^{x})$ (for varying $x$) and we say that $\tl{\bu} D$ is semi-free if there exists a filtration
\begin{align*}
  0=\tl{\bu}D_{-1}\subset \tl{\bu} D_0\subset \tl{\bu} D_1\subset \tl{\bu}D_2\subset  \dots\subset \tl{\bu} D
\end{align*}
such that $\bigcup \tl{\bu} D_n=\tl{\bu} D$ and each $\tl{\bu} (D_n/D_{n-1})$ is free.

Similarly we define free and semi-free for right path local systems. For bi-path local systems we define these notions point wise in the sense that $\tltr{\bu}{D}{\bu}\in \tltr{X}{\GG}{Y}$ is called semi-free if all $\tltr{x}{D}{\bu}$ and $\tltr{\bu}{D}{y}$ are semi-free. Note that this is not analogous to being semi-free as a bi-module. Indeed, with our definition $\tltr{\bu}{PX}{\bu}$ is semi-free.

\begin{lemma} \label{lem:article:2}
  All left path local systems has a weakly equivalent semi-free left path local system. 
\end{lemma}

\begin{proof}
  Let $\tl\bullet D\in \tl X\GG$, restrict it to $\tl xD\in C_*\Omega X\text{-Mod}$ and take a semi-free replacement $\FF\, \tl xD$ (as discussion in Section~\ref{sec:modules-over-dga}), then the composition 
  \begin{align*}
    \tl\bu{\FF D}:=\tltr{\bu}{PX}{x} \und{C_*\Omega X}\otimes \FF\, \tl xD\to \tltr{\bu}{PX}{x} \und{C_*\Omega X}\otimes \tl xD \to  \tl \bullet D
  \end{align*}
  is a weak equivalence since the first map is a quasi isomorphisms over the fiber $x$ and the second map is an isomorphism over the fiber $x$. The left hand side is semi-free using the filtration on $\FF \, \tl xD$.
\end{proof}

Motivated by this and the fact that the global tensor is not homotopical, we define the derived global tensor
\begin{align*}
  \tr{D}{\bu} \undL{PX}{\otimes} \tl{\bu} E = \tr{D}{\bu} \und{PX}{\otimes} \tl{\bu}{\FF E} \simeq \tr{\FF D}{\bu} \und{PX}{\otimes} \tl{\bu}{\FF E} \simeq \tr{\FF D}{\bu} \und{PX}{\otimes} \tl{\bu}{E}.
\end{align*}
The following lemma proves these equivalences and that the derived global tensor is homotopical.

\begin{lemma}\label{lem:tensor_lemma}
  If $D^\bu$ is semi-free or both $\tl{\bu}E$ and $\tl{\bu}E'$ are semi-free, and $f:\tl{\bu} E\to \tl{\bu} E'$ a weak equivalence then $\id\otimes f:D^\bullet \und{PX}{\otimes} \tl{\bu} E\to D^\bullet \und{PX}{\otimes} \tl{\bu} E'$  is a quasi isomorphism.
\end{lemma}

\begin{proof}
  When $D^\bu$ is free i.e. $D^\bu=\bigoplus_i \tltr{x_i}{PX}{\bu}$ then $\id\otimes f:D^\bullet \und{PX}{\otimes} \tl{\bu} E\to D^\bullet \und{PX}{\otimes} \tl{\bu} E'$ is isomorphic to the quasi isomorphism $\bigoplus_i f_{x_i}:\bigoplus_i \tl{x_i}{E}\to \bigoplus_i \tl{x_i}{E'}$. For the case when $\tr D{\bu}$ is semi-free we consider the induced filtration on $D^\bullet \und{PX}{\otimes} \tl{\bu} E$, $D^\bullet \und{PX}{\otimes} \tl{\bu} E'$ and the map $\id\otimes f$. The map on page 0 between the spectral sequences is by the free case a quasi isomorphism. Since the spectral sequences are supported in the right half plane with exiting differentials they are convergent and hence $\id\otimes f$ is a quasi isomorphism.

  If both $\tl{\bu}E$ and $\tl{\bu}E'$ are semi-free, the lemma is proved by replacing $\tr D\bu$ with $\tr{\FF D}\bu$ from Lemma~\ref{lem:article:2}, $D^\bullet \und{PX}{\otimes} \tl{\bu} E\simeq \FF D^\bullet \und{PX}{\otimes} \tl{\bu} E\simeq \FF D^\bullet \und{PX}{\otimes} \tl{\bu} E'\simeq D^\bullet \und{PX}{\otimes} \tl{\bu} E'$.
\end{proof}

\begin{corollary}
  If $\tltr{\bu}{D}{\bu}\in \tltr{X}{\GG}{Y}$ is semi-free then the induced Morita functors $\tltr{\bu}{D}{\bu}\und{PY}\otimes (-)$ and $(-)\und{PX}\otimes \tltr{\bu}{D}{\bu}$ are homotopical.
\end{corollary}

\section{Floer theory with path local systems} \label{sec:floer-theory-with}

As in \cite{Husin_2023} we will define Floer \emph{homology} with path local systems, but in this article the Floer complex is defined as a direct sum of tensor products instead of internal homs.

We assume that we are given a pair $(K,L)$ of exact Lagrangians $K,L \subset W$ where $W$ is a compact Liouville domain. The primitives to the Liouville form are denoted by $f_K:K \to \R$ and $f_L:L \to \R$ and satisfies $df_K=\lambda_{|K}$ and $df_L=\lambda_{|L}$. For an intersection point $x\in K\cap L$ we define its action as $\mathcal A(x)=f_K(x)-f_L(x)$.

When $c_1(W)$ is 2-torsion we may pick a global quadratic complex volume form. Using this we can when the Maslov classes of both $K$ and $L$ vanish choose a grading on both $K$ and $L$ to get graded Floer complexes. In addition, if we have relative pin structures on $K$ and $L$ we can orient the moduli spaces and work even in the case when $\F$ is not of characteristic 2 (see e.g. \cite{MR2441780}).

By a strip from $x\in K\cap L$ to $y\in K\cap L$ we mean a pseudo holomorphic map $u:D^2\to W$ with two punctures at $\pm 1$ (see Figure~\ref{Fig:Stripdefn}) such that the boundaries between the punctures are mapped to $K$ respectively $L$ according to Figure~\ref{Fig:Stripdefn}.
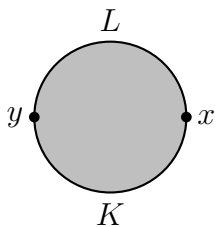
\begin{figure}[ht]
  \begin{center}
    \begin{tikzpicture}
      \fill[lightgray] (0,0) circle (1cm);
      \draw[thick] (0,0) circle (1cm);
      \fill (-1,0) circle (2pt) node[left] {$y$};
      \fill (1,0) circle (2pt)  node[right] {$x$};;
      \draw (0,1) node[above] {$L$};
      \draw (0,-1) node[below] {$K$};
    \end{tikzpicture}
  \end{center}
  \caption{Strip definition}\label{Fig:Stripdefn}
\end{figure}
Note that this depends on the order of the pair $(K,L)$. By Stokes theorem $\mathcal A(y)=\mathcal A(x)-\int_{D^2}u^*\omega$ and hence $\mathcal A(y)\leq \mathcal A(x)$. In either case by Proposition 2.1 in \cite{MR2904032} after a small perturbation of one of the Lagrangians and $J$ we may assume that:
\begin{itemize}
\item $K$ and $L$ intersect transversely.
\item For each pair of intersection points $x,y\in K\cap L$ the compactified moduli space $\tltr{y}{\MM}{x}$ of unparameterized strips as in Figure~\ref{Fig:Stripdefn} is a compact topological manifold with boundary such that
  \begin{align*}
    \partial \tltr{y}{\MM}{x} = \bigcup_{z \in K\cap L} \tltr{y}{\MM}{z} \times \tltr{z}{\MM}{x}
  \end{align*}
  where we use the convention that $\tltr{x}{\MM}{x}=\varnothing$.
\item The different compactified pieces of the boundary meet at codimension 1 submanifolds (in the boundary) given inductively as
  \begin{align*}
    \bigcup_{z_1,z_2 \in K\cap L} \tltr{y}{\MM}{z_2} \times \tltr{z_2}{\MM}{z_1} \times \tltr{z_1}{\MM}{x}.
  \end{align*}
\end{itemize}
The moduli spaces are orientable over $\F$ (vacuous statement depending on case) and the orientations can be chosen in a coherent way, see Section~12b in \cite{MR2441780}. The next Lemma states that we can choose fundamental chains on the moduli spaces in a coherent way, this is also explained after Lemma~2.3 in \cite{MR2904032}.
\begin{lemma}\label{lem:article:fundamental chains}
  We may inductively, in the \emph{local} dimension $k$, pick fundamental chains $\tlar{y}{c}{x}k \in C_k(\tltr{y}{\MM}{x})$ such that
  \begin{align}\label{equation:dc_sign}
    \partial \tlar{y}{c}{x}k = \sum_{\substack{z\in K\cap L\\k_1+k_2=k-1}} (-1)^{k_1+1}\ \tlar{y}{c}{z}{k_1}\times\tlar{z}{c}{x}{k_2} \in C_{k-1}(\partial \tltr{y}{\MM}{x}).
  \end{align}
\end{lemma}
Notes:
\begin{itemize}
\item $\alpha \times \beta$ denotes Eilenberg-Zilber on $\alpha \otimes \beta$.
\item The sign $(-1)^{k_1+1}$ is due to the fact that we are really orienting the spaces before taking the quotient by $\R$.
\item For each component in $\tltr{y}{\MM}{x}$ of dimension $k$, breaking into two pieces can only happen when $k_1+k_2=k-1$.
\item When the Maslov index does not vanish, several different pairs $(k_1,k_2)$ could potentially contribute for the same $z$.
\item In the following we denote the sum of these classes $\tltr ycx=\sum_k \tlar ycxk$, which only has one term in the graded case. 
\end{itemize}
\vspace{5mm}

Let $\ev_K : \tltr{y}{\MM}{x} \to \tltr{x}{\mathcal PK}{y}$ denote the evaluation on the side of $D^2$ mapping to $K$ (parametrized clockwise). We parametrize this path by arc length (continuous on the moduli space by Lemma~\ref{lem:article:12}). Similarly we define $\ev_L : \tltr{y}{\MM}{x} \to \tltr{y}{\mathcal PL}{x}$ as the evaluation at the side mapping to $L$ (still clockwise and parameterized by arc length).

Given $D^\bu \in \GG^K$ and $\tl{\bu} E \in \tl{L}{\GG}$ we define
\begin{align}\label{equation:floer_complex}
  CF_*(K^{D^\bu},L^{\stl{\bu}{E}})=\bigoplus_{x\in K\cap L}\Sigma^{|x|} D^x\otimes \tl{x}{E}.
\end{align}
Note that when we do not have grading, $\Sigma^{|x|}$ is the identity functor and all signs in what follows may be ignored, and when we do have grading, $|x|$ is well defined. The reason we define the Floer complex using both left and right path local systems is that it will be useful later when doing Morita theory.

Before defining the differential on $CF_*(K^{D^\bu},L^{\stl{\bu}{E}})$ we note that we have an action
\begin{align*}
  \tltr yTx:&\Sigma^{|y|-|x|}C_*(\tltr{y}{\MM}{x})\otimes \Sigma^{|x|}D^x\otimes \tl{x}{E}\\
            &\to\Sigma^{|x|}\Sigma^{|y|-|x|}\tr{D}x\otimes C_*(\tltr{y}{\MM}{x}) \otimes \tl yE\\
  &\to\Sigma^{|y|}\tr{D}{y}\otimes \tl{y}E
\end{align*}which consists of reshuffling the tensor factors and suspensions, then taking the push forward by $\ev_K\times \ev_L$ followed by the Alexander-Whitney map and then parallel transporting in $D$ and $E$. For $\alpha\in C_*(\tltr{y}{\MM}{x})$ we use the notation $\tltr yTx(\alpha\otimes d\otimes e)=\alpha(d\otimes e)$. It follows that $\tltr yTx$ is a chain map and we have associativity in the sense that
\begin{align}\label{equation:floer_associativity}
  \alpha(\beta(d\otimes e))=(-1)^{(|y|-|z|+|\alpha|)(|z|-|x|)}(\alpha\times \beta)(d\otimes e)
\end{align}
for $\alpha\in C_*(\tltr{y}{\MM}{z})$ and $\beta \in C_*(\tltr{z}{\MM}{x})$, the sign is due to $\Sigma^{|z|-|x|}$ moving past $\Sigma^{|y|-|z|}$ and $\alpha$.

The differential $\mu_1:CF_*(K^{D^\bu},L^{\stl{\bu}{E}})\to CF_*(K^{D^\bu},L^{\stl{\bu}{E}})$ is defined on each summand by using both their \emph{internal} differential $\partial_x=\partial_{\Sigma^{|x|}\str{D}{x}\otimes \stl{x}{E}}$ and \emph{external} differentials $\tlar{y}{\mu}{x}{1}:\Sigma^{|x|}\tr{D}{x}\otimes \tl{x}{E}\to \Sigma^{|y|}\tr{D}{y}\otimes \tl{y}{E}$ defined by $\tlar{y}{\mu}{x}{1}(d\otimes e)=\tltr{y}{c}{x}(d\otimes e)$. The external differentials are not chain maps, however they are of degree minus one and satisfy
\begin{align}\label{equation:dmu1}
  \notag\partial_{\mathcal Hom}(\tlar{y}{\mu}{x}1):&=\partial_y\tlar{y}{\mu}{x}{1} +\tlar{y}{\mu}{x}{1}\partial_x\\
  \notag&=(-1)^{|y|-|x|}(\partial \tltr{y}{c}{x})(-)\\
  \notag &=(-1)^{|y|-|x|}\sum_{z}(-1)^{|z|-|y|}(\tltr ycz\times \tltr zcx)(-)\\
  \notag &=(-1)^{|y|-|x|}\sum_z(-1)^{|z|-|y|}(-1)^{|x|-|z|}\, \tlar{y}{\mu}{z}{1}\tlar{z}{\mu}{x}{1}\\
                                               &=\sum_z\tlar{y}{\mu}{z}{1}\tlar{z}{\mu}{x}{1}
\end{align}
where we in the second equality use that $\tltr yTx$ is a chain map and Leibniz rule and in the third equality use Equation~\ref{equation:dc_sign} and in the forth equality use Equation~\ref{equation:floer_associativity}.
We now define
\begin{align*}
  \mu_1=\sum_x\partial_x-\sum_{x,y}\tlar{y}{\mu}{x}{1},
\end{align*}
where we use the convention that $\partial_x$ and $\tlar{y}{\mu}x1$ is zero on all summands except $\Sigma^{|x|}\tr{D}x\otimes \tl{x}E$. We get $\mu_1^2=0$ since
\begin{align*}
  \blbr{y}{(\mu_1^2)}{x}&=-\partial_y\tlar{y}{\mu}{x}{1}-\tlar{y}{\mu}{x}{1}\partial_x+\sum_{z} \tlar{y}{\mu}{z}{1} \tlar{z}{\mu}{x}{1}=0
\end{align*}
 where $\blbr{y}{(\mu_1^2)}{x}$ denotes $\mu_1^2$ restricted to $\Sigma^{|x|}\tr{D}x\otimes \tl{x}E$ and image projected to $\Sigma^{|y|}\tr{D}y\otimes \tl{y}E$. All the terms cancel by Equation~\ref{equation:dmu1}.

\begin{lemma} \label{lem:article:8}
  Any weak equivalence $\tl{\bu}{E} \to \tl{\bu}{E'}$ induces a quasi isomorphism $CF_*(K^{D^\bu},L^{\stl{\bu}{E}}) \to  CF_*(K^{D^\bu},L^{\stl{\bu}{E'}})$.
\end{lemma}

\begin{proof}
  This follows easily using the action filtration and the fact that fibers are quasi isomorphic. 
\end{proof}

Given two Lagrangians $K,L\subset W$ we define (after perturbing one of the Lagrangians and $J$) the important bi path local system
\begin{align}
  \CF{K}{L} =CF_*(K^{\sbpb{K}},L^{\sbpb{L}})=\bigoplus_{x\in K\cap L}\Sigma^{|x|}\tltr{\bu}{PK}{x} \otimes \tltr{x}{PL}{\bu}
\end{align} \vspace{-0.6cm}

\noindent
in $\tltr{K}{\GG}{L}$. The parallel transport is defined by $\gamma(\sigma\otimes \tau)=(-1)^{|\gamma||x|}\gamma\sigma \otimes \tau$ and $(\sigma\otimes \tau)\gamma=\sigma \otimes \tau\gamma$, where the sign for the left parallel transport comes from moving $\gamma$ past the suspension $\Sigma^{|x|}$.

This bi path local system is important as it is somewhat universal. Indeed, we have
\begin{align} \label{eq:article:14}
  \tr{D}{\bu} \und{PK} \otimes  \CF{K}{L} \und{PL} \otimes \tl{\bu}{E} \cong CF_*(K^{D^\bu},L^{\stl{\bu}{E}}).
\end{align}

\begin{lemma}\label{lemma:article:floer-bimodule-is-semifree}
  The bi path local system $\CF{K}{L}$ is semi-free.
\end{lemma}

\begin{proof}
  We will prove that $\tltr{\bu}{\FF(K,L)}{y}$ in $\tl{K}{\GG}$ is semi-free. The other case is similar.

  Let $m$ be an upper bound on the dimension of all moduli spaces involved in defining the complex (by Gromov compactness this exists). Order the finitely many intersection points $x\in K\cap L$ in an order of non decreasing action $x_0,\dots, x_k$ and define $i(x_j)=j$. Consider the filtration
  \begin{align*}
    F^p(\tltr{\bu}{\FF(K,L)}{y}) = \bigoplus_{x\in K\cap L}\Sigma^{|x|}\tltr{\bu}{PK}{x} \otimes G^{p-i(x)(m+1)}(\tltr{x}{PL}{y})
  \end{align*}
  where $G^q$ is the simplex degree filtration on the singular simplicial complex. Any differential has to lower the action and can increase the simplicial degree by at most $m$. It follows that the graded quotient of this filtration is a direct sum
  \begin{align*}
    \bigoplus_{x\in K\cap L}\Sigma^{|x|}\tltr{\bu}{PK}{x} \otimes (G^{p-i(x)(m+1)}(\tltr{x}{PL}{y})/G^{p-i(x)(m+1)-1}(\tltr{x}{PL}{y}))
  \end{align*}
  where now the direct sum is both as chain complexes and left path local systems (unlike those above). We also note that the differential on the last tensor factor is zero - hence this sum is free.
\end{proof}

As we are using homological convention we now define a co-product
\begin{align*}
  \mu_2: \CF{K}{M} \to \CF{K}{L} \und{ PL}\otimes \CF{L}{M}
\end{align*}
which we note maps to the global tensor and not the regular tensor. To ease notation, we go through the definition without having the left and right actions along for the ride, and equivalently define it as
\begin{align*}
  \mu_2:CF_*(K^{D^\bu},M^{\stl{\bu}{E}})&\to CF_*(K^{D^\bu},L^{\stltr{\bu}{PL}{\bu}})\und{ PL}\otimes CF_*(L^{\stltr{\bu}{PL}{\bu}},M^{\stl{\bu}{E}})\\
  &\cong \bigoplus_{\substack{y\in K\cap L\\z\in L\cap M}}\Sigma^{|y|}\Sigma^{|z|}\tr{D}{y}\otimes \tltr{y}{PL}{z}\otimes \tl{z}{E}
\end{align*}
where the isomorphism sends $d\otimes \gamma_1\otimes \gamma_2\otimes e$ to $(-1)^{(|d|+|\gamma_1|)|z|}d\otimes \gamma_1\gamma_2\otimes e$, the sign is due to moving $\Sigma^{|z|}$ past $D^y\otimes \tltr{y}{PL}{\bullet}$. This isomorphism induces a differential on $\bigoplus\Sigma^{|y|}\Sigma^{|z|}\tr{D}{y}\otimes \tltr{y}{PL}{z}\otimes \tl{z}{E}$ in the following way. Similar to above we define the chain maps
\begin{align*}
  \tltr{y'}{\widetilde T}{y}:&\Sigma^{|y'|-|y|} C_*(\tltr{y'}{\,\MM}{y}) \otimes \Sigma^{|y|}\Sigma^{|z|}\tr D{y}\otimes \tltr {y}{PL}z\otimes \tl zE\\
                            &\to\Sigma^{|y|}\Sigma^{|y'|-|y|}\Sigma^{|z|}\tr D{y}\otimes  C_*(\tltr{y'}{\,\MM}{y})\otimes \tltr {y}{PL}z\otimes \tl zE\\
  &\to \Sigma^{|y'|}\Sigma^{|z|}\tr D{y'}\otimes \tltr {y'}{PL}z\otimes \tl zE
\end{align*}
and
\begin{align*}\tltr{z'}{\,\widehat T}{z}:&\Sigma^{|z'|-|z|}C_*(\tltr{z'}{\,\MM}{z})\otimes  \Sigma^{|y|}\Sigma^{|z|}\tr Dy\otimes \tltr y{PL}{z}\otimes \tl {z}E\\
  &\to \Sigma^{|y|}\Sigma^{|z|}\Sigma^{|z'|-|z|}\tr Dy\otimes \tltr y{PL}{z}\otimes C_*(\tltr{z'}{\,\MM}{z})\otimes \tl {z}E\\
  &\to \Sigma^{|y|}\Sigma^{|z'|}\tr D{y}\otimes \tltr {y}{PL}{z'}\otimes \tl {z'}E.
\end{align*}
Denote $\tltr{y'}{\widetilde{T}}{y}(\alpha\otimes d\otimes \gamma\otimes e)=\alpha(d\otimes \gamma\otimes e)$, $\tltr{z'}{\widehat{T}}{z}(\alpha\otimes d\otimes \gamma\otimes e)=\alpha(d\otimes \gamma\otimes e)$ and $\tlar{y'}{\widetilde \mu}{y}{1}(d\otimes \gamma\otimes e)=\tltr{y'}{c}{y}(d\otimes \gamma\otimes e)$, $\tlar{z'}{\widehat \mu}{z}{1}(d\otimes \gamma\otimes e)=\tltr{z'}{c}{z}(d\otimes \gamma\otimes e)$. The induced differential $\mu_1'$ on $\bigoplus\Sigma^{|y|}\Sigma^{|z|}\tr{D}y\otimes \tltr{z}{PL}{z}\otimes \tl{z}E$ then becomes
\begin{align*}
  \mu_1'&=\sum_{y,z}\partial_{y,z}- \sum_{y,y',z}\tlar{y'}{\widetilde \mu}{y}{1} - \sum_{y,z,z'}\tlar{z'}{\widehat \mu}{z}{1}
\end{align*}
where $\partial_{y,z}=\partial_{\Sigma^{|y|}\Sigma^{|z|}\str{D}y\otimes \stltr{y}{PL}{z}\otimes \stl{z}{E}}$.

Note that our $\mu_2$ lands in a tensor product over $PL$ which is somewhat different than the standard $\mu_2$ in the Fukaya category. However, in the next section we discuss another variant of $\mu_2$ which is slightly closer to the usual one.
\begin{figure}[ht]
  \begin{center}
   \begin{tikzpicture}
      \fill[lightgray] (0,0) circle (1cm);
      \draw[thick] (0,0) circle (1cm);
      \fill (1,0) circle (2pt) node[right] {$x$};
      \fill (-120:1) circle (2pt)  node[below left] {$y$};;
      \fill (120:1) circle (2pt) node[above left] {$z$};;
      \draw (-60:1) node[below] {$K$};
      \draw (-1,0) node[left] {$L$};
      \draw (60:1) node[above] {$M$};
    \end{tikzpicture}
  \end{center}
  \caption{Floer co-product} \label{Fig:FLoercoprod}
\end{figure}
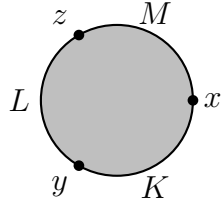
We define our $\mu_2$ by considering the compactified moduli spaces $\tltltr{z}{y}{\,\MM}{x}$ of pseudo holomorphic discs with three marked points as in Figure~\ref{Fig:FLoercoprod}. The codimension one boundary of this consists of strips bubbling of at one of the three points and the orientations of these moduli spaces can as before be chosen coherently, see Section~12b in \cite{MR2441780}. Assume that we have already inductively picked fundamental chains $\tlar{y}{c}{x}{k} \in \tltr{y}{\MM}{x}$ for the three involved Floer moduli spaces. We again inductively in dimension pick fundamental chains $\tltlar{z}{y}{\,c}{x}{k}\in C_k(\tltltr{z}{y}{\,\MM}{x})$ for these spaces extending the choices on their boundaries so that they satisfy
\begin{align}\label{article:mu2_signs}
  \notag \partial \tltlar{z}{y}{\,c}{x}{k}&=\sum_{\substack{x'\in K\cap M\\k_1+k_2=k-1}}(-1)^{k_1+1}\, \tltlar{z}{y}{\,c}{x'}{k_1}\times \tlar{x'}{c}{x}{k_2}\\
                   \notag &\quad+ \sum_{\substack{y'\in K\cap L\\k_1+k_2=k-1}}(-1)^{|z|(k_1+1)}\,\tlar{y}{c}{y'}{k_1} \times \tltlar{z}{y'}{c}{x}{k_2}\\
                                                                               &\quad + \sum_{\substack{z'\in L\cap M\\k_1+k_2=k-1}}\tlar{z}{c}{z'}{k_1} \times \tltlar{z'}{y}{c}{x}{k_2}.
\end{align}
We can now define chain maps
\begin{align*}
  \tltltr{z}{y}{\,T}{x}:&\Sigma^{|y|+|z|-|x|}C_*(\tltltr{z}{y}{\,\MM}{x})\otimes \Sigma^{|x|}\tr{D}x\otimes \tl{x}E\\
  &\to \Sigma^{|x|}\Sigma^{|y|+|z|-|x|}\tr{D}x\otimes C_*(\tltltr{z}{y}{\,\MM}{x})\otimes  \tl{x}E\\
  &\to \Sigma^{|y|}\Sigma^{|z|}\tr Dy\otimes \tltr y{PL}z\otimes \tl zE
\end{align*}by rearranging the tensor factors and suspensions, taking the push forward of $\ev : \tltltr{z}{y}{\,\MM}{x} \to \tltr{x}{\mathcal PK}{y} \times \tltr{y}{\mathcal PL}{z} \times \tltr{z}{\mathcal PM}{x}$, applying the Alexander-Whitney map, and parallel transporting in $D$ and $E$. We denote $\tltltr zyTx(\alpha\otimes d\otimes e)=\alpha(d\otimes e)$ and have the following associativity
\begin{align*}
  \alpha(\beta(d\otimes e))=(-1)^{(|y|+|z|-|x'|+|\alpha|)(|x'|-|x|)}(\alpha\times \beta)(d\otimes e)
\end{align*}
for $\alpha\in C_*(\tltltr{z}{y}{\,\MM}{x'})$, $\beta\in C_*(\tltr{x'}{\,\MM}{x})$, the sign is due to $\Sigma^{|x'|-|x|}$ moving past $\Sigma^{|y|+|z|-|x'|}$ and $\alpha$,
\begin{align*}
  \alpha(\beta(d\otimes e))=(-1)^{(|y|-|y'|+|\alpha|)(|y'|-|x|)+|\alpha||z|}(\alpha\times \beta)(d\otimes e)
\end{align*}
for $\alpha\in C_*(\tltr{y}{\,\MM}{y'})$, $\beta\in C_*(\tltltr{z}{y'}{\,\MM}{x})$, the sign is due to $\Sigma^{|y'|-|x|}$ moving past $\Sigma^{|y|-|y'|}$ and $\alpha$, and $\Sigma^{|z|}$ moves past $\alpha$,
\begin{align*}
  \alpha(\beta(d\otimes e))=(-1)^{(|z|-|z'|+|\alpha|)(|y|+|z'|-|x|)}(\alpha\times \beta)(d\otimes e)
\end{align*}
for $\alpha\in C_*(\tltr{z}{\,\MM}{z'})$, $\beta\in C_*(\tltltr{z'}{y}{\,\MM}{x})$, the sign is due to $\Sigma^{|y|+|z'|-|x|}$ moving past $\Sigma^{|z|-|z'|}$ and $\alpha$.

We then let $\tltlar{z}{y}{\mu}{x}{2}=\tltltr{z}y{\,c}{x}(d\otimes e)$.
As for $\mu_1$ we use that $\tltltr{z}{y}{\,T}{x}$ is a chain map, Leibniz rule, Equation~\ref{article:mu2_signs} and the associativity above to get the formula
\begin{align}\label{equation:dmu2}
  \notag\partial_{\mathcal Hom}\tltlar{z}y{\mu}{x}{2}:&=\partial_{y,z}\tltlar{z}y{\mu}{x}{2}-\tltlar{z}y{\mu}{x}{2}\partial_x \\
  \notag&=(-1)^{|y|+|z|-|x|}(\partial \tltltr{z}{y}{c}{x})(-)\\
  \notag&=\sum_{x'}(-1)^{|x'|-|x|+1}\, \tltlar{z}y{\mu}{x'}{2}\tlar{x'}{\mu}{x}{1}\\
                   \notag&\quad+ \sum_{y'}(-1)^{|y|-|y'|}\, \tlar{y}{\widetilde \mu}{y'}{1}\tltlar{z}{y'}{\mu}{x}{2}\\
                                                                                                 &\quad + \sum_{z'}(-1)^{|z|-|z'|}\,\tlar{z}{\widehat \mu}{z'}{1} \tltlar{z'}{y}{\mu}{x}{2}.
\end{align}
Now define
\begin{align*}
  \mu_2=\sum_{x,y,z}(-1)^{|x|-|y|-|z|}\,\tltlar{z}{y}{\mu}{x}{2},
\end{align*}
where $\tltlar{z}{y}{\mu}x2$ is zero on all summands except $\Sigma^{|x|}\tr{D}x\otimes \tl{x}E$.

From Equation~\ref{equation:dmu2} it follows that $\mu_2$ is a chain map
\begin{align*}
  \blbr{y,z}{(\mu_1\mu_2)}{x}&=(-1)^{|x|-|y|-|z|}\partial_{y,z} \tltlar{z}y{\mu}{x}{2}\\
  &\quad- \sum_{y'}(-1)^{|x|-|y'|-|z|}\,\tlar{y}{\widetilde \mu}{y'}{1}\tltlar{z}{y'}{\mu}{x}{2}-\sum_{z'}(-1)^{|x|-|y|-|z'|}\,\tlar{z}{\widehat \mu}{z'}{1}\tltlar{z'}{y}{\mu}{x}{2}\\
                                 &=(-1)^{|x|-|y|-|z|}\,\tltlar{z}y{\mu}{x}{2}\partial_x-\sum_{x'}(-1)^{|x'|-|y|-|z|}\,\tltlar{z}y{\mu}{x'}{2}\tlar{x'}{\mu}{x}{1}=\blbr{y,z}{(\mu_2\mu_1)}{x}
\end{align*}
where $\blbr{y,z}{(\mu_1\mu_2)}{x}$ is $\mu_1\mu_2$ with input in $\Sigma^{|x|}\tr{D}x\otimes \tl{x}E$ and output in $\Sigma^{|y|}\Sigma^{|z|}\tr{D}y\otimes \tltr{y}{PL}{z}\otimes \tl{z}E$ and similarly for $\blbr{y,z}{(\mu_2\mu_1)}{x}$.

\subsection{\texorpdfstring{A $\mu_2$ variant}{A mu\_2 variant}}\label{section:mu_variant}
The above defined $\mu_2$ involves paths in the middle Lagrangian. This is not exactly the usual $\mu_2$, however augmenting away these paths, we get something more recognizable. We call the following composition $\mu_2^{\mathbb F}$
\begin{align*}
  \mu_2^{\mathbb F}:CF_*(K^{\str{D}\bullet}, M^{\stl\bullet E})&\xrightarrow{\mu_2} CF_*(K^{D^\bu},L^{\stltr{\bu}{PL}{\bu}})\und{ PL}\otimes CF_*(L^{\stltr{\bu}{PL}{\bu}},M^{\stl{\bu}{E}})\\
                                             &\to CF_*(K^{D^\bu},L^{\stl{\bu}{\F}})\otimes CF_*(L^{\str{\F}{\bu}},M^{\stl{\bu}{E}})\\
  &\cong \bigoplus_{\substack{y\in K\cap L\\z\in L\cap M}}\Sigma^{|y|}\Sigma^{|z|}\tr{D}{y}\otimes \tl{z}{E}
\end{align*}
where the second map simply sends $d\otimes \gamma_1\otimes \gamma_2\otimes e$ to $\epsilon(\gamma_1\gamma_2)d\otimes e$ and the last isomorphism sends $d\otimes e$ to $(-1)^{|d||z|}d\otimes e$ where the sign is due to the suspension $\Sigma^{|z|}$ moving past $\tr Dy$.

\begin{remark} \label{rem:article:1}
  We note in particular that when $D^\bu = \F^\bu$ and $\tl{\bu}{E}=\tl{\bu}{\F}$ this recovers usual Floer complexes and $\mu_2$ as a co-product.
\end{remark}

\subsection{Continuation maps}\label{section:continuation}
As our definition of $\mu_2$ is a slightly different than usual, it is a bit difficult to describe its co-unit. It is therefore at this point more convenient to define continuation maps using continuity properties.

Let $K\hookrightarrow W$ be an exact Lagrangian embedding with primitive $f_K$ such that $df_K=\lambda_{|K}$ and $L_t\hookrightarrow W$, $t\in \R$ a one parameter family of exact Lagrangian embeddings with a primitive $f_{L_0}$ such that $df_{L_0}=\lambda_{|L_0}$ on $L_0$. We assume that $L_t=L_0$ is constant for $t<\epsilon$ and $L_t=L_1$ is constant for $t>1-\epsilon$, for some small $\epsilon>0$. We also assume that both $L_0$ and $L_1$ are transverse to $K$. It is known that such a family $L_t$ can be realized as the flow of a Hamiltonian vector field $X_{H_t}$ defined by $\omega(X_{H_t},-)=-dH_t$ for a time dependent Hamiltonian $H_t:W\to \R$, furthermore we can assume that $H_t=0$ when $t\notin (\epsilon,1-\epsilon)$. We can now define primitives $f_{L_t}$ such that $df_{L_t}=\lambda_{|L_t}$ on $L_t$ by the equation
\begin{align*}
  f_{L_t}(\varphi^t x)=f_{L_0}(x)+\int_0^t\left(\lambda_{\varphi^\tau x}(X_\tau(\varphi^\tau x))-H_\tau(\varphi^\tau x)\right)d\tau, \quad x\in L_0
\end{align*}
where $\varphi^t$ denotes the time $t$ flow of $X_{H_t}$. Let $D^\bu\in \tr \GG{K}$, $\tl \bu E_0\in \tl {L_0}\GG$ and by the pullback of the map $(\varphi^t)^{-1}:L_t\to L_0$ we define $\tl \bu E_t\in \tl {L_t}\GG$. We will then define a chain map called the continuation map $C_{K,L_t}:CF_*(K^{D^\bu},L_1^{\stl \bu E_1})\to CF_*(K^{D^\bu},L_0^{\stl \bu E_0})$ which we in Proposition~\ref{prop:article:invariant_under_isotopy} will show is a quasi isomorphism.

The construction is as follows. From the data of $K$ and $L_t$ we will define exact Lagrangians $\widetilde K,\widetilde L\subset (W\times \R_t\times \R_y,\lambda\oplus \frac 12(tdy-ydt), J\oplus J_0)$ in a bigger space 
\begin{align*}
  &\widetilde{K}=K\times \{(t,R(t^2-t))\}\subset W\times \R^2,\\
  &\widetilde L=\{(x,t,H_t(x)),\ x\in L_t\}\subset W\times \R^2
\end{align*}
for a constant $R\gg 0$ so large that the only intersection points $(x,t,y)\in \widetilde K\cap \widetilde L$ occur when $t\in \{0,1\}$. We define a primitive on $\widetilde K$ by $f_{\widetilde K}(x,t,y)=f_K(x)+\frac R6 t^3$ and a primitive on $\widetilde L$ by $f_{\widetilde L}(x,t,y)=f_{L_t}(x)+\frac 12 ty$. In case of vanishing Maslov class and a choice of gradings on $K$ and $L_t$ varying continuously in $t$, we can induce gradings on $\widetilde K$ and $\widetilde L$ such that $\deg_{\widetilde K, \widetilde L}(z,0,0)=\deg_{K,L_0}(z)$ for $z\in K\cap L_0$ and $\deg_{\widetilde K,\widetilde L}(z,1,0)=\deg_{K,L_1}(z)+1$ for $z\in K\cap L_1$. In case we have relative pin structures on $K$ and $L_t$ varying continuously in $t$ we can induce relative pin structures on $\widetilde K$ and $\widetilde L$, this way we can orient moduli spaces and work with signs.

By the pullback of $\tr D\bu$ through the projection $\widetilde K\to K$ we define $\tr {\widetilde D}\bu\in \tr \GG {\widetilde K}$ and by the pullback of $\tl\bu{E_0}$ through the composition $\widetilde L\to \bigcup L_t\to L_0$ we define $\tl \bu {\widetilde E}\in \tl {\widetilde L}\GG$.

Now consider the Floer complex $CF_*(\widetilde K^{\str {\widetilde D}\bu},\widetilde L^{\stl \bu {\widetilde E}})$ as defined in Equation~\ref{equation:floer_complex}, we choose fundamental chains on moduli spaces of strips extending the choices of fundamental chains used to define the differential in $CF_*(K^{D^\bu},L_0^{\stl \bu E_0})$ and $CF_*(K^{D^\bu},L_1^{\stl \bu E_1})$. More precisely there is a bijection of moduli spaces $\tlar{(y,t,0)}\MM{(x,t,0)}{\widetilde K,\widetilde L}\cong \tlar{y}\MM{x}{K,L_t}$ for $t\in\{0,1\}$ and we require $\tlar{(y,0,0)}c{(x,0,0)}{\widetilde K,\widetilde L}=\tlar y c x {K,L_0}$ and $\tlar{(y,1,0)}c{(x,1,0)}{\widetilde K,\widetilde L}=-\tlar y c x {K,L_1}$, the minus is because the moduli spaces $\tlar{(y,1,0)}\MM{(x,1,0)}{\widetilde K,\widetilde L}$ and $\tlar{y}\MM{x}{K,L_1}$ have different orientations, see Section~12b in \cite{MR2441780}. The rest of the fundamental chains on the moduli spaces are as usual required to satisfy Equation~\ref{equation:dc_sign} and can again be achieved by picking the fundamental chains inductively on the dimension of moduli spaces. By projection of curves in $W\times \R^2$ to $\R^2$ and by monotonicity the differential on $CF_*(\widetilde K^{\str {\widetilde D}\bu},\widetilde L^{\stl \bu {\widetilde E}})$ squares to zero and takes the form
\begin{align*}
  \begin{pmatrix}
    (\mu_{1})_{K,L_0}&C_{K,L_t}\\
    0&-(\mu_{1})_{K,L_1}
  \end{pmatrix}
\end{align*}
in a basis where the generators in $W\times \{0\}\times\{0\}$ come before the generators in $W\times \{1\}\times\{0\}$, for more sophisticated versions of monotonicity see Appendix~\ref{sec:monotonicity}. The map $C_{K,L_t}:CF_*(K^{D^\bu},L_1^{\stl \bu {E_1}})\to CF_*(K^{D^\bu},L_0^{\stl \bu {E_0}})$ is called a continuation map and is a chain map since the matrix describing the differential on $CF_*(\widetilde K^{\str {\widetilde D}\bu},\widetilde L^{\stl \bu {\widetilde E}})$ squares to zero.
\begin{remark}\label{rem:article:sign_continuation}
  We could have defined the continuation map $C_{K,L_t}$ with the opposite sign and all results in this section would still hold.
\end{remark}

\begin{lemma}\label{lem:article:constant_isotopy}
  If $L_t=L$ is the constant isotopy, then $C_{K,L_t}:CF_*(K^{D^\bu},L^{\stl \bu E})\to CF_*(K^{D^\bu},L^{\stl \bu E})$ is a quasi isomorphism.
\end{lemma}

\begin{proof}
  In the constant case we can pick $H_t=0$ and $f_{L_t}=f_{L_0}$. As the moduli spaces are on product form (before modding out by $\R$) it follows using action filtration on the original complex, that the map $C_{K,L_t}$ is upper triangular with plus or minus the identity on each generator on page one of the spectral sequence. Indeed, the map is simply given by parallel transporting along a constant path coming from the projection to $W$ of the boundary of the discs $\{x\}\times \{R(t^2-t)\leq y\leq 0\}\subset W\times \R^2$ where $x\in K\cap L$. Hence $C_{K,L_t}$ is a quasi isomorphism.
\end{proof}

\begin{lemma}\label{lem:article:general_homotopic_isotopeis}Let $L_{s,t}\hookrightarrow W$, $t,s\in \R$ be a homotopy (in $s$) of isotopies of exact Lagrangian embeddings so that $L_{s,t}=L_{s,0}$ for $t<\epsilon$, $L_{s,t}=L_{s,1}$ for $t>1-\epsilon$, $L_{s,t}=L_{0,t}$ for $s<\epsilon$ and $L_{s,t}=L_{1,t}$ for $s>1-\epsilon$. We assume that $K$ is transverse to $L_{0,0}$, $L_{0,1}$, $L_{1,0}$ and $L_{1,1}$. Let $\tr {D}\bu\in \tr{\GG}K$ and $\tl \bu{E_{0,0}}\in \tl{L_{0,0}}\GG$ and as before by pullback define $\tl \bu{E_{s,t}}\in \tl{L_{s,t}}\GG$. Then the compositions of continuation maps $C_{K,L_{0,t}}C_{K,L_{s,1}}\simeq C_{K,L_{s,0}}C_{K,L_{1,t}}$ are chain homotopic.
\end{lemma}

\begin{proof} By the construction above we for each fixed $s$ get exact Lagrangians $\widetilde K,\widetilde L_{s}\subset W\times \R^2$. Repeating this process we get exact Lagrangians $\widehat K,\widehat L\subset W\times \R^4$ and a chain complex $CF_*(\widehat K^{\str {\widehat D}\bu},\widehat L^{\stl \bu {\widehat E}})$ where the differential by monotonicity squares to zero and takes the form
  \begin{align*}
    \begin{pmatrix}
      (\mu_1)_{K,L_{0,0}}&C_{K,L_{0,t}}&C_{K,L_{s,0}}&h\\
      0&-(\mu_1)_{K,L_{0,1}}&0&C_{K,L_{s,1}}\\
      0&0&-(\mu_1)_{K,L_{1,0}}&-C_{K,L_{1,t}}\\
      0&0&0&(\mu_1)_{K,L_{1,1}}
    \end{pmatrix}.
  \end{align*}
  The element in the top right corner of the square of the matrix should be zero so
  \begin{align*}
    (\mu_1)_{K,L_{0,0}}h+C_{K,L_{0,t}}C_{K,L_{s,1}}-C_{K,L_{s,0}}C_{K,L_{1,t}}+h(\mu_1)_{K,L_{1,1}}=0
  \end{align*}
  and hence $h:CF_*(K^{D^\bu},L_{1,1}^{\stl \bu E_{1,1}})\to \Sigma CF_{*}(K^{D^\bu},L_{0,0}^{\stl \bu E_{0,0}})$ is a chain homotopy between $C_{K,L_{0,t}}C_{K,L_{s,1}}$ and  $C_{K,L_{s,0}}C_{K,L_{1,t}}$.
\end{proof}

\begin{proposition}\label{prop:article:invariant_under_isotopy}
  The continuation map
  \begin{align*}
    C_{K,L_{t}}:CF_*(K^{D^\bu},L_{0}^{\stl \bu E_{0}})\to CF_*(K^{D^\bu},L_{1}^{\stl \bu E_{1}})
  \end{align*}
  is a quasi isomorphism and hence the homology of the above defined Floer complexes are invariant under isotopies of exact Lagrangians.
\end{proposition}
\begin{proof}
  Let $L'_t=L_{1-t}$ be the reverse isotopy. Create a two parameter family $L_{s,t}\hookrightarrow W$ of exact Lagrangian embeddings so $L_{s,t}=L_0$ for $t<\epsilon$, $L_{s,t}=L'_s$ for $t>1-\epsilon$, $L_{s,t}=L_t$ for $s<\epsilon$ and $L_{s,t}=L_0$ for $s>1-\epsilon$. By Lemma~\ref{lem:article:general_homotopic_isotopeis} $C_{K,L_{t}}C_{K,L'_t}\simeq C_{K,L_{0}}C_{K,L_{0}}$ are chain homotopic, $C_{K,L_{0}}$ is the continuation map defined from the constant isotopy and hence a quasi isomorphism by Lemma~\ref{lem:article:constant_isotopy}, and therefore $C_{K,L_{t}}C_{K,L'_t}$ is a quasi isomorphism. Similarly $C_{K,L'_{t}}C_{K,L_t}$ is a quasi isomorphism and the proposition follows.
\end{proof}

\section{Cosimplicial local systems} \label{sec:cosimplicial-local}

Another model for generalized (aka derived) local systems that we will employ is the notion of cosimplicial local systems. We call them this as they are reminiscent of co-sheaves modeled on semi-simplicial sets.

Let $S$ be a semi-simplicial set (no simplicial degeneracies). Recall that this is the same as a contravariant functor $S : \Delta \to \Set$ where $\Delta$ has one object $\un n=\{0,1,\dots,n\}$ for each natural number $n$ and morphisms given by \emph{injective} increasing maps $\sigma : \un n \to \un m$. The elements in the set $S(\un p)$ are called the $p$-simplices or $p$-cells in $S$.

If $S(\sigma)(\alpha)=\beta$ we write $\beta \subset \alpha$ and call it a sub-face, but note that usually (but not with the assumptions below) the same $\beta$ can be a sub-face in $\alpha$ in multiple ways. Maps of semi-simplicial sets are the same as natural transformations of such functors. The standard semi-simplicial $p$ simplex $s\Delta^p$ can be defined as the functor $\Hom_\Delta(- ,\un p)$. For each $p$-simplex $\alpha \in S(\un p)$ we thus have a canonical semi-simplicial map
\begin{align*}
  i_\alpha : s\Delta^p \to S
\end{align*}
given by the maps $\Hom_\Delta(\un q, \un p) \to S(\un q)$ natural in $\un q$ sending $\sigma$ to $S(\sigma)(\alpha)$. So, e.g. the identity on $\un p$ is mapped to $\alpha$. In fact there is a natural bijection between $p$-simplices in $S$ and semi-simplicial maps $s\Delta^p \to S$. We will also consider the boundary $\partial \alpha$, by which we mean the image of $i_\alpha$ on $s\Delta^p$ with the top simplex $\id_{\un p}$ removed.

From now on we will assume that semi-simplicial sets, denoted by $S$ (or sometimes $T$) satisfy
\begin{itemize}[topsep=7pt,leftmargin=40pt]
\item[$\qquad$\textbf{S1:}] The semi-simplicial set is connected, injective and have contractible stars $s(\alpha)$. 
\end{itemize}
Here injective means that all $p$-simplices have $p+1$ distinct corners and for all $p+1$-tuples $(\alpha_0,\dots,\alpha_p)\in S(\un 0)$ of $0$-cells, there should exist at most one $p$-simplex with corners $(\alpha_0,\dots,\alpha_p)$. The star $s(\alpha)$ for each $\alpha \in S$ is defined to be the closure of the union of simplices $\beta$ which intersects $\alpha$ non-trivially. Here closure means taking the closure under face maps (which is equivalent to the topological closure after geometric realization). The property S1 implies for instance that $|i_\alpha|:|\alpha| \to |S|$ is an embedding.

Let $\CC(S)$ be the category with an object for each simplex $\alpha$ in $S$ and a morphism $\alpha \to \beta$ if $\alpha\subset \beta$. A map of semi-simplicial sets $S\to T$ induces a functor $\CC(S) \to \CC(T)$. The geometric realization $|S|$ of $S$ is defined as the colimit in spaces
\begin{align*}
  |S|=\colim_{\CC(S)} \Delta^{-}
\end{align*}
where $\Delta^{-} : \CC(S) \to \top$ sends a $q$ simplex to the topological $q$ simplex $\Delta^q$ and each map to the associated maps on the simplices. Note that any Cw complex $X$ is homotopy equivalent to the geometric realization of a semi-simplicial set $S$ satisfying S1 above. Indeed, taking the singular simplices in $X$ and forgetting degeneracies we may take the thrice iterated barycentric sub division to obtain such (we will not need nor prove this but in Section~\ref{sec:simpl-morse-smale} we recall the barycentric subdivision).

The category $\GG_S$ of right cosimplicial local systems on $S$ will be defined to be a full subcategory of functors and natural transformations
\begin{align*}
  \GG_S \subset \Fun(\CC(S) , \Ch).
\end{align*}
The maps in the image of such a functor are called \emph{co-face} maps. We write such a functor $\br A\bu \in \Fun(\CC(S),\Ch)$ evaluated on a simplex $\alpha \in S$ as $A_\alpha$ and we will call this the \emph{fiber} of $A$ at $\alpha$. For any semi-simplicial map $f:T \to S$ we define the pullback $f^*A$ by composing with the induced functor, which yields $(f^*A)_{\alpha} = A_{f(\alpha)}$. We also define the total complex of $A$ (analogous to geometric realization) as the colimit
\begin{align*}
  \tot(\br A\bu) = \colim_{\CC(S)} \br A\bu
\end{align*}
This is not so well-behaved (homotopically) for general such functors $A$. E.g. if $A_\alpha=\F$ with all maps identities then the total complex is just $\F$. However if we e.g. put
\begin{align*}
  A_\alpha=C_*^\Delta(s\Delta^{\deg(\alpha)})
\end{align*}
(which is quasi isomorphic to $\F$) with the natural inclusions used as the co-face maps then $\tot(A) \cong C_*^\Delta(S)$.

We could work with homotopy colimits to fix this issue. However, this would obscure and make certain arguments later much harder. So, instead we restrict to the following more homotopically well-behaved sub-categories.
\begin{definition}
  A right \emph{pre-cosimplicial} local system $A$ is a functor $\CC(S) \to \Ch$ such that for any sub-semi-simplicial set $S_1\subset S$ the map $\tot(A_{\mid S_1})\to \tot(A)$ is injective. We denote the category of these and natural transformations by $\pGG_S$.

  A right \emph{cosimplicial} local system $A$ is a pre-cosimplicial local system such that each map $A_{\alpha_0}\to A_\beta$ for $\alpha_0$ any zero cell in $\beta$ is a quasi isomorphism. The category of these are denoted $\GG_S$.

  A weak equivalence $A \to B$ in $\GG_S$ or $\pGG_S$ is a map (natural transformation) such that each map $A_\alpha \to B_\alpha$ is a quasi isomorphism.
\end{definition}

The following lemma is immediate as $S$ is assumed to be connected.

\begin{lemma} \label{lem:lemmas:single_fiber}
  Let $A,B\in \GG_S$. If $f:A\to B$ is a quasi isomorphism on a single fiber $\alpha\in S$ then $f$ is a weak equivalence.
\end{lemma}

We will later work mostly with cosimplicial local systems, but as we will need pre-cosimplicial local systems at a few points we prove several lemmas also for these. The difference between $A$ being a cosimplicial local system as opposed to a pre-cosimplicial local system is essentially that the homology is locally constant. In Section~\ref{sec:equiv-glps-gcss} we describe a way of turning a pre-cosimplicial local system into an actual cosimplicial local system (a sort of fibrant replacement).

The badly behaved example above is not even a pre-cosimplicial local system. Indeed, for $\alpha \in S$ a 1 simplex the map $\tot(A_{\mid \partial \alpha}) =\F\oplus\F \to \F=\tot(A_{\mid\alpha})$ is not injective.

For any $A\in\pGG_S$ we define the \emph{top} of $A$ at $\alpha$ to be
\begin{align*}
  \ov A_\alpha = \Sigma^{-\deg(\alpha)}(A_\alpha / \tot(A_{\mid\partial \alpha})).
\end{align*}
We may filter $\tot(A)$ by simplex degree. I.e. $F^p\tot(A)=\tot(A_{\mid S[p]})$ where $S[p]\subset S$ denotes the $p$-skeleton. The following is immediate from these definitions and the part in Section~\ref{sec:notation-conventions} about convergence of spectral sequences.

\begin{lemma} \label{lem:article:14}
  For any $A \in \pGG_S$ the simplex filtration induces a spectral sequence converging to $H_*(\tot(A))$ whose page 0 is given by the direct sum of chain complexes
  \begin{align*}
    \bigoplus_{\alpha} \ov A_\alpha.
  \end{align*}
\end{lemma}

It follows that page 1 as an $\F$ vector space is given in degree $(p,q)$ by
\begin{align*}
  \bigoplus_{|\alpha|=p} H_q(\ov A_\alpha)
\end{align*}
and the general structure means that the differential on each summand is a sum of maps
\begin{align} \label{eq:article:3}
  d_i : H_*(\ov A_\alpha) \to H_*(\ov A_{\partial_i\alpha})
\end{align}
where $\partial_i\alpha$ denotes the $i$th codimension 1 face. For $A \in \GG_S$ we will see below that the maps $d_i$ are all isomorphisms and hence we can define $H_*(A)$ as a (classical) local system of (possibly graded) homology groups on $S$.

\begin{remark} \label{rem:article:2}
  A good example to keep in mind is if $E \to |S|$ is a Serre fibration then we may define
  \begin{align*}
    A_\alpha = C_*(|i_\alpha|^*E)
  \end{align*}
  where $|i_\alpha| : \Delta^{\deg(\alpha)} \to |S|$ is the inclusion of the standard simplex associated to $\alpha$. In this example the following lemma recovers the Serre spectral sequence.  
\end{remark}

\begin{lemma}\label{lem:lemmas:spectral_sequence_for_total_space}
  For $A\in \GG_S$ the maps $d_i$ defined above are isomorphisms and the spectral sequence in Lemma~\ref{lem:article:14} has page 1 isomorphic to $C_*^\Delta(S;H_*(A))$.
\end{lemma}

\begin{proof}
  The proof is by induction on the $p$-skeleton of $S$, over the $0$-skeleton the statement is trivial. Assume that the statement is true over the $p-1$ skeleton and consider $\alpha\in S(\un p)$. From the inclusions $\tot(A_{\mid h_i\alpha})\subset \tot(A_{\mid \partial \alpha})\subset \tot(A_{\mid \alpha})$ where $h_i\alpha$ is the $i$th horn of $\alpha$ we get a long exact sequence
  \begin{align*}
    H_{p+q}\frac{\tot(A_{\mid \alpha})}{\tot(A_{\mid h_i\alpha})}\to \underbrace{H_{p+q}\frac{\tot(A_{\mid \alpha})}{\tot(A_{\mid \partial \alpha})}}_{H_{q}(\ov A_\alpha)}\to \underbrace{H_{p+q-1}\frac{\tot(A_{\mid \partial \alpha})}{\tot(A_{\mid h_i\alpha})}}_{H_q(\ov A_{\partial_i\alpha})}\to H_{p+q-1} \frac{\tot(A_{\mid \alpha})}{\tot (A_{\mid h_i\alpha})}.
  \end{align*}
  By induction and the defining condition on $A \in \GG_S$ the homology of $\tot(A_{\mid h_i\alpha})$ is the same as the homology over a single fiber and hence the first and last term vanishes, therefore the map $d_i:H_q(\ov A_\alpha)\to H_q(\ov A_{\partial_i\alpha})$ is an isomorphism.
  
  The isomorphism from page one of $\tot(A_{\mid S[p]})$ to $C_p^\Delta(S; H_*(A))$,
  \begin{align*}
    E_{p,q}^1=\bigoplus_{|\alpha|=p} H_q(\ov A_\alpha)\to \bigoplus_{|\alpha|=p} H_q(A_\alpha)=C_p^\Delta(S; H_q(A)),
  \end{align*}
  is given by maps $H_q(\ov A_\alpha)\to H_q(A_\alpha)$ defined through the factorization
  \begin{align*}
    H_q(\ov A_\alpha)\xrightarrow{(-1)^{i_1+\dots +i_q}d_{i_q}\circ \dots\circ  d_{i_1}} H_q(\ov A_{\alpha_0})=H_q(A_{\alpha_0}) \to H_q(A_\alpha)
  \end{align*}
  where $\alpha_0$ is any zero simplex in $\alpha$ and $i_1,\dots,i_q$ is any sequence of integers so that the maps $d_{i_1},\dots,d_{i_q}$ are defined. Since the co-face maps satisfy the cosimplicial relations, the isomorphism is well defined and commutes with the differentials.
\end{proof}

\begin{corollary} \label{cor:lem:lemmas:1}
  Any weak equivalence $A \to B$ in $\pGG_S$ induces a quasi isomorphism on the total complexes $\tot A \to \tot B$.
\end{corollary}

\begin{proof}
  The assumptions mean that the mapping cones $M_\alpha$ of each map $A_\alpha \to B_\alpha$ defines a fiber wise acyclic object $M \in \GG_S$ (not just $\pGG_S$). Its total complex is the cone of the total complexes hence also acyclic by the lemma above.
\end{proof}

\begin{corollary}\label{cor_tot_is_homotopical}
  The functors $\tot : \pGG_S \to \Ch$ and $\tot : \GG_S \to \Ch$ are homotopical (i.e. preserves weak equivalences).
\end{corollary}

\subsection{Tensors of cosimplicial local systems} \label{sec:tensors}

There is again both a fiber-wise tensor product $\fwcs$ and a global tensor product $\und{S}\otimes$ on objects in $\GG_S$ (and $\pGG_S$) both of which we define in this section.

The fiber-wise tensor we denote $\fwcs$ and is defined simplex wise by
\begin{align*}
  (A\fwcs B)_\alpha = A_\alpha \otimes B_\alpha.
\end{align*}

\begin{lemma}
  Both $\pGG_S$ and $\GG_S$ are closed under the fiber wise tensor.
\end{lemma}

\begin{proof}
  Any $A$ is in $\pGG_S$ (or $\GG_S$) if and only if $i_\alpha^*A$ is in $\pGG_{s\Delta^{\deg(\alpha)}}$ (or $\GG_{s\Delta^{\deg(\alpha)}}$) for all simplices $\alpha$ in $S$. The tensor commutes with the restriction so it is enough to prove the lemma on a standard simplex. So we consider for each $n$ the case $S=s\Delta^n$. We let $\alpha$ denote the unique $n$ simplex and note that $\tot D = D_\alpha$ for any $D \in \Fun(\CC(S),\Ch)$. For $n=0$ the requirements to be in the categories are vacuous.

  Let $A,B \in \pGG_S$. As $A_\beta \otimes B_\beta \subset A_\beta \otimes B_\alpha$ for each $\beta$ we get the factorization
  \begin{align*}
    \tot((A\fwcs B)_{\mid S_1}) \subset \tot(A_{\mid S_1}) \otimes B_\alpha \subset A_\alpha\otimes B_\alpha
  \end{align*}
  for each $S_1 \subset s\Delta^n$, which proves that $A\otimes B$ is in $\pGG_S$.

  If further $A,B\in \GG_S$ then it follows that the inclusion
  \begin{align} \label{eq:article:1}
    A_{\alpha_0} \otimes B_{\alpha_0} \to A_\alpha \otimes B_\alpha = \tot(A\fwcs B),
  \end{align}
  where $\alpha_0$ denotes any zero simplex in $S$, is a quasi isomorphism as it is a tensor of such. Hence $A\otimes B$ is in $\GG_S$.
\end{proof}

The global tensor, we denote by $\und{S}\otimes$ and it is defined as
\begin{align*}
  \br A\bu \und{S}\otimes \br B\bu = \tot(A\fwcs B).
\end{align*}

Since we later want to do Morita theory we will denote objects in $\GG_S$ by $\br{A}{\bu}$ and call $\GG_S$ the category of right cosimplicial local systems on $S$. We then also introduce the category $\bl{S}{\GG}=\br{\GG}{S}$ of left cosimplicial local systems on $S$, objects in $\bl{S}{\GG}$ are denoted by $\bl{\bu} A$. We make this distinction even though the categories $\GG_S$ and $\bl{S}{\GG}$ are exactly the same. Usually we will take the global tensor of a right and a left (pre-)cosimplicial system.

Since taking total complex is homotopical we get the following.

\begin{corollary}\label{lem:cor:tensor_lemma_cosimp}
  If $A_\bu\in \GG_S$ and $\bl{\bu} B,\bl{\bu} B'\in \bl{S}{\pGG}$ and if $f:\bl{\bu} B\to \bl{\bu} B'$ is a weak equivalence then $\id\otimes f:A_\bu\und{S}\otimes \bl{\bu} B\to A_{\bu}\und{S}\otimes \bl{\bu} B'$ is a quasi isomorphism.
\end{corollary}
\begin{corollary}
  Both the fiberwise and global tensor of cosimplicial local systems are homotopical.
\end{corollary}

\section{Combined categories} \label{sec:combined-categories}

In the previous section we let $\br{\GG}{S}$ consists of functors that take values in $\Ch$, but of course we could consider other categories, e.g. the path local system categories from Section~\ref{sec:path-local-systems} or $\bl{T}{\GG}$ for another semi-simplicial set $T$ satisfying S1. For such $T$ and manifolds $K$ and $L$ we let $\CC$ denote either $\bl{T}{\GG} (=\br{\GG}{T}$), $\tl{K}{\GG}$, $\GG^L$ or $\tltr{K}{\GG}{L}$.

The category of cosimplicial (or pre-cosimplicial) local system in $\CC$ is a subcategory of functors $A_\bullet : \CC(S) \to \CC$ such that taking the fibers over points in $K$, $L$ or $K\times L$ or a simplex $\alpha \in T$ we get objects in $\br{\GG}{S}$ (or $\pGG_S$). Note that in this generalization of $\GG_S$, the condition that each $A_{\alpha_0} \to A_\alpha$ is a quasi isomorphism is replaced by it being a weak equivalence in $\CC$.

The categories of cosimplicial local systems over $S$ with values in $\bl{T}\GG$, $\tl{K}{\GG}$, $\GG^L$ or $\tltr{K}{\GG}{L}$ are denoted by $\blbr{T}{\GG}{S}$, $\tlbr{K}{\GG}{S}$, $\tlar{}{\GG}{L}{S}$ and $\tlar{K}{\GG}{L}{S}$ respectively. We have put a subscript bullet in the notation to signify that it is a local system where simplices can be put at the bullet. Generally we will use the convention exemplified by:
\begin{align*}
  \tlar{\bu}{A}{\bu}{\bu} \in \tlar{K}{\GG}{L}{S} \qquad \textrm{and} \qquad  \blbr{\bu}{A}{\bu} \in \blbr{T}{\GG}{S}.
\end{align*}
In the first, the two upper bullets signify fibers over $K$ and $L$ as previously used for bi path local systems, and the lower bullet is used for inserting simplices from $S$.

In the second example we similarly use a lower left bullet for the target category $\bl{T}{\GG}$. But note that there is no real difference between left and right here as both $\blbr{T}{\GG}{S}$ and $\blbr{S}{\GG}{T}$ are isomorphic to categories of functors $\CC(T) \times \CC(S) \to \Ch$ with symmetric conditions. For the same reason we will sometimes e.g. write $\tl{\bu} A_\bu \in  \tlbr{K}{\GG}{S}$ to emphasize that we would like to think of $S$ more as a right local system (in the next section when describing Morita theory this means we think of $S$ as a source).

Maps are again natural transformations of the functors, which means that over each simplex they consist of maps in the category $\CC$. The total complex $\tot(A_\bullet)$ is naturally an object in $\CC$. Weak equivalences are defined to be fiber-wise weak equivalences over $S$.

\begin{lemma} \label{cor:lem:lemmas:1b}
  The functor $\tot$ from (pre-)cosimplicial local systems in $\CC$ to $\CC$ is homotopical.
\end{lemma}

\begin{proof}
  This follows directly from Corollary~\ref{cor:lem:lemmas:1} and the fact that weak equivalences in $\CC$ are defined as maps which are quasi isomorphisms in each fiber over $K$, $L$, $K\times L$ or $T$.
\end{proof}

We also call such a combined object \emph{semi-free} if taking fibers in all but one ``path fiber'' is semi-free. E.g. $\tlar{\bu}{A}{\bu}{\bu}$ is semi-free if $\tlar{x}{A}{\bu}{\alpha}$ and $\tlar{\bu}{A}{y}{\alpha}$ is semi-free for all $x,y,\alpha$.

\section{\texorpdfstring{Equivalence of $\protect\tl{|S|}{\GG}$ and $\protect\bl{S}{\GG}$}{Equivalence of path and cosimplicial local systems}} \label{sec:equiv-glps-gcss}

To turn a cosimplicial local system into a left or right path local system on the geometric realization $X=|S|$ we essentially use the standard idea of path space replacement (fibrant replacement) when constructing Serre fibrations. However, the best way we found to phrase this for later use is in terms of Morita theory and Morita equivalence. That is, we prove that the functors
\begin{align} \label{eq:article:2}
  \tl{\bu} PX_\bullet\und{S}\otimes - : \bl{S}{\GG}\to \tl{X}{\GG}\\
  \bl{\bu} PX^\bu \und{PX}\otimes - : \tl{X}{\GG} \to \bl{S}{\GG } \label{eq:article:2b}
\end{align}
are inverse $\pi_0$-equivalences of categories. Here $\tl{\bu} PX_\bullet\in \tlbr{X}{\GG}{S}$ have the fiber $\tlbr{x}{PX}{\alpha}$ consisting of chains of the space of Moore paths $\tltr{x}{\PP X}{|\alpha|}$ beginning at $x$ and ending anywhere in $|\alpha|$, the definition of $\bl{\bu} PX^\bu \in \bltr{S}{\GG}{X}$ is similar. That $\bltr{\bu}{PX}{\bu}\und{PX}\otimes \tl{\bu}{D}$ is pre-cosimplicial follows by definition and it is furthermore cosimplicial as we have a homotopy inverse to $\bltr{\alpha_0}{PX}{\bu}\und{PX}\otimes \tl{\bu}{D}\to \bltr{\beta}{PX}{\bu}\und{PX}\otimes \tl{\bu}{D}$ induced from a homotopy inverse (in $\tr{\GG}{X}$) to $\bltr{\alpha_0}{PX}{\bu}\to \bltr{\beta}{PX}{\bu}$.

The following lemma proves that the functors are homotopical (which is why we do not need the tensor in Equation~\eqref{eq:article:2b} to be derived) and the two lemmas after that say that the notion of fibers is preserved.

\begin{lemma}\label{lem:PX_functors_homotopical}
  The functors in Equation~(\ref{eq:article:2}) and (\ref{eq:article:2b}) are homotopical.
\end{lemma}

\begin{proof}
  If $f:\tl{\bu}{D}\to \tl{\bu}{E}$ is a weak equivalence then $\id\otimes f:\bltr{\bu}{PX}\bu\und{PX}\otimes \tl{\bu}D\to \bltr{\bu}{PX}\bu\und{PX}\otimes \tl{\bu}E$ is over any zero-cell a quasi isomorphism (as $\bltr{\alpha_0}{PX}\bu \cong \tltr{|\alpha_0|}{PX}\bu$ and $\tltr{|\alpha_0|}{PX}\bu \und{PX} \otimes \tl{\bu}{F} \cong \tl{|\alpha_0|}{F}$) and hence a weak equivalence. If $f:\bl{\bu}A\to \bl{\bu}B$ is a weak equivalence then by Corollary~\ref{lem:cor:tensor_lemma_cosimp} $\id\otimes f:\tlbr{x}{PX}{\bu}\und{S}\otimes \bl{\bu}A\to \tlbr{x}{PX}{\bu}\und{S}\otimes \bl{\bu}B$ is a quasi isomorphism for each $x$ and hence a weak equivalence.
\end{proof}

\begin{lemma}\label{lem:article:5}
  For $\bl{\bu} A \in \bl{S}{\GG}$ the inclusion $\bl{\alpha_0}{A} \to \tlbr{|\alpha_0|}{PX}{\bu} \und{S}\otimes \bl{\bu} A$ given by $a \mapsto \cst_{|\alpha_0|} \otimes a$ is a quasi isomorphism when $\alpha_0\in S$ is a zero-cell.
\end{lemma}

\begin{proof}
  Consider the spectral sequence we get by filtering the right hand side by the chain degree in $\tlbr{|\alpha_0|}{PX}{\beta}$. That map on page 1 reads
  \begin{align} \label{eq:article:5}
    H(\bl{\alpha_0}{A})\to \left[\bigoplus_{\beta\in S}\tlbr{|\alpha_0|}{PX}{\beta} \otimes H(A)\right]/\sim
  \end{align}
  where $H(A)$ is the local system defined on $X$ by $\bl{\bu} A$. The right hand side is a sub complex of $C_*(\tl{|\alpha_0|}{\PP X},H(A)) \simeq H(\bl{\alpha_0}{A})$ where now $H(A)$ is considered a local system on this path space by pullback from $X$. This sub complex is defined by those simplices in $\tl{|\alpha_0|}{\PP X}$ where the evaluation at the end is contained in a single simplex $|\beta|$ with $\beta \in S$. As the inclusion $S \subset \sing X$ is a weak equivalence (homotopy equivalence after geometric realization) it follows that the inclusion of this sub complex is a quasi isomorphism, which proves the inclusion in Equation~(\ref{eq:article:5}) is a quasi isomorphism.
\end{proof}

\begin{lemma}\label{lem:article:4}
  For $\tl{\bu} D \in \tl{X}{\GG}$ and $\alpha_0 \in S$ a zero simplex the inclusion $\tl{|\alpha_0|}{D} \to \bltr{\alpha_0}{PX}{\bu} \und{PX}\otimes \tl{\bu} D$ given by $d \mapsto \cst_{|\alpha_0|} \otimes d$ is a quasi isomorphism.
\end{lemma}

\begin{proof}
  By definition of the tensor over $PX$ this is in fact an isomorphism with inverse $\alpha\otimes d\mapsto \alpha d$.
\end{proof}

\begin{lemma}\label{lem:article:zig-zag}
  There is a weak equivalence $\tl{\bu} PX_\bullet \und{S}\otimes \bl{\bu} \,PX^\bu  \to \tl{\bu} PX^\bu $.
\end{lemma}

\begin{proof}
  For each $\alpha \in S$ and $x,y\in X$ consider the map
  \begin{align} \label{eq:article:8}
    \tltr{x}{\PP X}{|\alpha|}  \times \tltr{|\alpha|}{\PP X}{y} \to \tltr{x}{\PP X}y
  \end{align}
  which sends $((\gamma_1,r_1),(\gamma_2,r_2))$ to the path $(\gamma,r_1+r_2+1)$ where
  \begin{align*}
    \gamma(t) =
    \begin{cases}
      \gamma_1(t) & t\leq r_1 \\
      \tilde \gamma(t-r_1) & r_1\leq t\leq r_1+1 \\
      \gamma_2(t-r_1-1) & t\geq r_1+1
    \end{cases}
  \end{align*}
  where $\tilde \gamma:[0,1] \to X$ is the unique affine path in $\alpha$ (using the simplicial structure) from $\gamma_1(r_1)$ to $\gamma_2(0)$.

  Using this we define the map
  \begin{align*}
    \tlbr{\bu}{PX}\alpha \otimes \bltr{\alpha}{PX}\bu \xrightarrow{\ez} C_*(\tltr{\bu}{\PP X}{|\alpha|} \times \tltr{|\alpha|}{\PP X}{\bu}) \to \tltr{\bu}{PX}\bu.
  \end{align*}
  These are compatible with path actions from the left and right, and with inclusions $\alpha \subset \beta$ so we get a map of bi-path-local systems
  \begin{align*}
    \tlbr{\bu}{PX}\bullet \und{S}\otimes \bltr{\bu}{PX}\bu  \to \tltr{\bu}{PX}\bu.
  \end{align*}
  The Eilenberg-Zilber part is a weak equivalence so we prove the lemma by proving that
  \begin{align*}
    \tot (\alpha \mapsto C_*(\tltr{x}{\PP X}{|\alpha|} \times \tltr{|\alpha|}{\PP X}{y})) \to \tltr{x}{PX}y.
  \end{align*}
  is a quasi isomorphism. To see this we first consider the Serre fibration
  \begin{align*}
    E=\tl x{\PP X} \times_X \tr{\PP X}y \to X
  \end{align*}
  whose total space over a simplex $|\alpha| \subset X$ is a homotopy equivalent subspace in $\tltr x{\PP X}{|\alpha|} \times \tltr{|\alpha|}{\PP X}{y}$. This means that the cosimplicial local system given over $\alpha$ by $C_*(|i_\alpha|^* E)$ (as in Remark~\ref{rem:article:2}) is equivalent to the cosimplicial local system $\alpha \mapsto C_*(\tltr{x}{\PP X}{|\alpha|} \times \tltr{|\alpha|}{\PP X}{y})$. The total complexes are thus equivalent, and the lemma follows as the map from the total space of $E$ to $\tl{y} PX^x$ composing the two paths and inserting a constant path of length 1 between them is a homotopy equivalence.
\end{proof}

\begin{lemma} \label{lem:article:3}
  There is a weak equivalence $\bl{\bu} PX^\bu  \und{PX}\otimes\tl{\bu} PX_\bullet\to  \bl{\bu} PX_\bu$ that sends $\beta\otimes \alpha$ to $\beta\alpha$.
\end{lemma}

\begin{proof}
  We prove that  $\bltr{\beta}{PX}{\bu}  \und{PX}\otimes\tl{\bu} PX_\alpha \to  \blbr{\beta}{PX}{\alpha}$ is a quasi isomorphism for fixed $\alpha$ and $\beta$. Let $\alpha_0\subset \alpha$ be a zero cell and consider the commutative diagram
  \begin{center}
    \begin{tikzcd}
      \bltr{\beta}{PX}{\bu} \und{PX}\otimes\tl{\bu} PX_{\alpha_0}\ar{r}{\cong}\ar[swap]{d}{\simeq}&\blbr{\beta}{PX}{\alpha_0}\ar{d}{\simeq}\\
      \bltr{\beta}{PX}{\bu}  \und{PX}\otimes\tl{\bu} PX_\alpha\ar{r}&\blbr{\beta}{PX}{\alpha}
    \end{tikzcd}
  \end{center}
  The top row is an isomorphism since there is an inverse sending $\gamma$ to $\gamma\otimes \mathrm{cst}_{\alpha_0}$. The right map is a homotopy equivalence as it is so on the level of spaces. The map $\tlbr{\bu}{PX}{\alpha_0} \to \tlbr{\bu}{PX}{\alpha}$ is similarly a homotopy equivalence of left path local systems. Indeed, we may deformation retract the latter onto the former by continuously adding part of the straight line in $|\alpha|$ to $|\alpha_0|$ (e.g. parametrized by arc length). Tensoring this homotopy equivalence with the identity shows that the left map is a homotopy equivalence. The commutativity of the diagram now implies that the lower map is a homotopy equivalence and hence a quasi isomorphism.
\end{proof}

Note that the global tensor products are associative in the sense that
\begin{align*}
    &(A_\bullet\und{S}\otimes \bl{\bu} B^\bu)\und{PX}\otimes \tl{\bu} C \cong A_\bullet \und{S}\otimes(\bl{\bu} B^\bu\und{PX}\otimes  \tl{\bu} C ),\\
    &(A^\bullet\und{PX}\otimes \tl{\bu} B_\bu)\und{S}\otimes \bl{\bu} C \cong A^\bullet \und{PX}\otimes(\tl{\bu} B_\bu\und{S}\otimes  \bl{\bu} C ),
\end{align*}
hence we do not need to write parenthesis that specify in what order we take global tensor products.

\begin{corollary} \label{cor:article:1}
  When $\tl{\bu} D\in \tl X\GG$ then $\tl{\bu} D\simeq \tlbr{\bu}{PX}\bu \und{S}\otimes \bltr{\bu}{PX}{\bu}\und{PX}\otimes \tl{\bu} D$.
\end{corollary}

\begin{proof}
  When $\tl\bu D$ is semi-free this follows by Lemma~\ref{lem:article:zig-zag} and Lemma \ref{lem:tensor_lemma}. When $\tl \bu D$ is not semi-free we choose a semi-free replacement $\tl \bu{\FF D}$ which exists by Lemma~\ref{lem:article:2} and get
  \begin{align*}
    \tl{\bu}D \simeq \tl{\bu}{\FF D} \simeq \tlbr{\bu}{PX}{\bu} \und{S}\otimes \bltr{\bu}{PX}{\bu}\und{PX}\otimes \tl{\bu}{\FF D}\simeq \tlbr\bu{PX}\bu \und{S}\otimes \bl{\bu} \,PX^\bu\und{PX}\otimes \tl{\bu}{D}
  \end{align*}
  where the last equivalence follows by Lemma~\ref{lem:PX_functors_homotopical}.
\end{proof}

\begin{corollary} \label{cor:article:2}
  When $\bl{\bu} A\in \bl{S}{\GG}$ then $\bl{\bu} A\simeq \bl{\bu} PX^\bullet \und{PX}\otimes \tl{\bu} PX_\bu\und{S}\otimes \bl{\bu} A$.
\end{corollary}

\begin{proof}
  By Lemma~\ref{lem:article:3} and Corollary~\ref{lem:cor:tensor_lemma_cosimp} we have to prove that
  \begin{align*}
    \bl{\bu} A\simeq \bl{\bu} PX_\bu \und{S}\otimes \bl{\bu} A.
  \end{align*}
  We do this by finding a zig-zag. Define $\bl{\bu} C\in \bl{S}{\GG}$ by $\bl{\alpha}{C}=C_*^\Delta(\alpha)$, then there is a weak equivalence $\epsilon\otimes \id:\bl{\bu} C\fwcs \bl{\bu} A\to \bl{\bu} A$ sending $\beta\otimes a$ to $\epsilon(\beta)a$ where $\epsilon$ is the augmentation, the map is a weak equivalence since  $\epsilon:C_*^\Delta(\alpha)\to \F$ is a quasi isomorphism.

  There is also a weak equivalence $\mathrm{cst}\otimes \id:\bl{\bu} C\fwcs \bl{\bu} A\to \bl{\bu} PX_\bu \und{S}\otimes \bl{\bu} A$ sending $\sigma\otimes a\in \bl{\alpha}{C}\otimes \bl{\alpha}{A}$ to $\cst_\sigma\otimes a\in \blbr{\alpha}{PX}{\alpha}\otimes \bl{\alpha}{A}$. To see that $\mathrm{cst}\otimes \id$ is a weak equivalence it is enough to show that is it a weak equivalence over a single fiber $\alpha_0$, i.e. the map
  \begin{align*}
    \bl{\alpha_0}{A}\to \blbr{\alpha_0}{PX}{\bullet}\und{S}\otimes \bl{\bu}{A} \cong \left[\bigoplus_{\beta\in S}\blbr{\alpha_0}{PX}{\beta} \otimes \bl{\beta}{A} \right] / \sim
  \end{align*}
  should be a weak equivalence, but this follows from Lemma~\ref{lem:article:5} above.
\end{proof}

\begin{proposition}\label{prop:pi_0_equivalence}
  The functors in Equation~(\ref{eq:article:2}) and Equation~(\ref{eq:article:2b}) are inverse $\pi_0$-equivalences of categories.
\end{proposition}

\begin{proof}
  This follows directly from Corollary~\ref{cor:article:1} and Corollary~\ref{cor:article:2}.
\end{proof}

\begin{corollary} \label{cor:article:3}
  Assuming that $X=|S|$ and $Y=|T|$ the functor $\tltr{X}{\GG}{Y} \to \blbr{S}{\GG}{T}$ given by
  \begin{align*}
    \tltr{\bu}{D}{\bu} \mapsto \bltr{\bu}{PX}{\bu} \und{PX}\otimes \tltr{\bu}{D}{\bu} \und{PX}\otimes \tlbr{\bu}{PX}{\bu} 
  \end{align*}
  is a $\pi_0$-equivalence.
\end{corollary}

\subsection{The equivalence preserves tensors}

\begin{lemma} \label{lem:article:18}
  The functor $F=\tlbr{\bu}{PX}\bu \und{S}\otimes - :\bl{S}{\GG} \to \tl{X}{\GG}$ is weakly symmetric monoidal with respect to fiber-wise tensor of path local systems. That is, we have a natural zig-zag of weak equivalences
  \begin{align*}
   F(\bl{\bu} (A\fwcs B)) \simeq F(\bl{\bu} A) \fwlp F(\bl{\bu} B).
  \end{align*}
  It follows that so is its $\pi_0$-inverse from Equation~(\ref{eq:article:2}).
\end{lemma}

\begin{proof}
  Rewriting both sides we want to construct this map as
  \begin{align} \label{eq:article:4}
    \tot(\tl{\bu} PX_\bu \fwcs (\bl{\bu}A \fwcs \!\bl{\bu} B)) \to  \tot(\tl{\bu} PX_\bu \fwcs \! \bl{\bu} A) \fwlp     \tot(\tl{\bu} PX_\bu \fwcs \! \bl{\bu} B).
  \end{align}
  For each $\alpha \in S$ we have a map
  \begin{align*}
    \tl{\bu} PX_\alpha \otimes (\bl{\alpha}{A} \otimes \bl{\alpha}{B}) \to (\tl{\bu} PX_\alpha \otimes \bl{\alpha}{A}) \otimes (\tl{\bu} PX_\alpha \otimes \bl{\alpha}{B})
   \end{align*}
   (in $\tl{X}{\GG}$) given by using Alexander-Whitney on the path factor and rearranging the factors a bit. As each factor of the target maps to the colimits (total complexes) in Equation~(\ref{eq:article:4}) we get for each $\alpha$ a map
  \begin{align*}
    \tl{\bu} PX_\alpha \otimes (\bl{\alpha}{A} \otimes \bl{\alpha}{B}) \to \tot(\tl{\bu} PX_\bu \fwcs \bl{\bu} A) \fwlp     \tot(\tl{\bu} PX_\bu \fwcs \bl{\bu} B).
  \end{align*}
  This map is compatible with face inclusions - hence we define
  \begin{align*}
    \eta_{\sbl{\bu} A,\sbl{\bu} B}:F(\bl{\bu} (A\fwcs B)) \to F(\bl{\bu} A) \fwlp F(\bl{\bu} B)
  \end{align*}
  as the colimit of these maps. It remains to show that it is a weak equivalence. Consider any zero cell $\alpha_0$ and the map
  \begin{align*}
    \bl{\alpha_0}{A} \otimes \bl{\alpha_0}{B} \to \tot(\tlbr{|\alpha_0|}{PX}{\bu} \fwcs (\bl{\bu} A \fwcs \bl{\bu} B))
  \end{align*}
  defined by sending $a \otimes b$ to $\cst_{|\alpha_0|}\otimes a \otimes b$. By Lemma~\ref{lem:article:5} this is a weak equivalence, and composing it with $\eta_{\sbl{\bu} A,\sbl{\bu} B}$ the lemma also shows that it is a tensor product of two weak equivalences - hence another weak equivalence and the first part of the lemma follows.

  To see that the $\pi_0$-inverse $G=\bltr{\bu}{PX}\bu \und{PX}\otimes - : \tl{X}{\GG}\to \bl{S}{\GG}$ is weakly symmetric monoidal we apply $F$ to $G(\tl \bu D \fwlp \tl \bu E)$ and $G(\tl \bu D)\fwcs G(\tl \bu E)$ and use that the fiber-wise tensor of path local systems is homotopical.
\end{proof}

The following corollary follows by applying the above lemma followed by its right path version.

\begin{corollary} \label{cor:article:3a}
  The functor in Corollary~\ref{cor:article:3} is weakly symmetric monoidal with respect to fiber wise tensor product.
\end{corollary}

\begin{lemma} \label{lem:article:9}
  The functor $F'\times F: \tr{\GG}{X} \times \tl{X}{\GG} \to \br{\GG}{S} \times \bl{S}{\GG}$ (where each factor is given by tensoring with the appropriate bi-path local system) preserves the derived global tensor up to weak equivalence. That is, we have a natural zig zag of quasi isomorphisms between
  \begin{align*}
    \tr{D}{\bu} \undL{PX}\otimes \tl{\bu}{E} \qquad \textrm{and} \qquad F'(\tr{D}{\bu}) \und{S}\otimes F(\tl{\bu}{E}).
  \end{align*}
\end{lemma}

Here ``natural'' emphasizes that if we e.g. have a left structure on $\tr{D}{\bu}$ then the zig-zag will be compatible with this.

\begin{proof}
  As the path global tensor is derived we may assume that $\tr D\bu$ is semi-free. We get
  \begin{align*}
    F'(\tr{D}{\bu}) \und{S}\otimes F(\tl{\bu}{E}) = \tr{D}{\bu} \und{PX}\otimes (\tlbr{\bu}{PX}{\bu} \und{S}\otimes \bltr{\bu}{PX}{\bu} \und{PX}\otimes \tl{\bu}{E}) \simeq \tr{D}{\bu} \und{PX}\otimes \tl{\bu}{E}.
  \end{align*}
  where the quasi isomorphism is by Corollary~\ref{cor:article:1} and Lemma~\ref{lem:tensor_lemma}.
\end{proof}

\begin{lemma}
  The functor $G'\times G:\br{\GG}{S} \times \bl{S}{\GG}\to \tr{\GG}{X} \times \tl{X}{\GG}$ preserves global tensor, i.e. $\br{A}{\bu} \und{S}\otimes \bl{\bu}{B}\simeq G'(\br{A}{\bu}) \und{PX}\otimes G(\bl{\bu}{B})$.
\end{lemma}

\begin{proof}
  We have
  \begin{align*}
    G'(\br{A}{\bu}) \und{PX}\otimes G(\bl{\bu}{B})=\br A\bu\und{S}\otimes (\bltr \bu{PX}\bu \und{PX}\otimes \tlbr \bu{PX}\bu \und S\otimes \bl \bu B)\simeq \br A\bu\und S\otimes \bl \bu B
  \end{align*}
  by Corollary~\ref{cor:article:2} and Corollary~\ref{lem:cor:tensor_lemma_cosimp}.
\end{proof}

\subsection{Fibrant replacement}
The constructions so far in this section involve cosimplicial local systems. However, it can also be used to turn pre-cosimplicial local systems into cosimplicial local systems. In fact, this is completely analogous to taking any map $X\to Y$ of spaces and replacing it with a Serre fibration $X' \to Y$ with $X \simeq X'$.

\begin{lemma} \label{lem:article:16}
  We have a well-defined functor $\bl{S}{\pGG} \to \bl{S}{\GG}$ given by
  \begin{align*}
    \bl{\bu}A \mapsto \bl{\bu} PX_\bu \und{S}\otimes \bl{\bu}A
  \end{align*}
  that restricts to the identity on weak equivalence classes of $\bl{S}\GG$.
\end{lemma}

\begin{proof}
  The fact that this functor is the identity on $\bl{S}\GG$ follows from Lemma~\ref{lem:article:3}, Corollary~\ref{cor:article:2} and Corollary~\ref{lem:cor:tensor_lemma_cosimp}. That $\bl{\bu} PX_\bu \und{S}\otimes \bl{\bu}A$ is a pre-cosimplicial local system follows by definition, and local homology condition making it into a cosimplicial local system follows from Lemma~\ref{cor:lem:lemmas:1b} and the fact that the inclusion
  \begin{align*}
    \bl{\alpha_0} PX_\beta \otimes \bl \beta A \subset \bl{\alpha} PX_\beta\otimes \bl \beta A
  \end{align*}
  is a quasi isomorphism.
\end{proof}

\subsection{Diagonals}
The following lemma is a reflection of $\tltr{\bu}{PX}{\bu}$ being a representation of the diagonal $X \subset X\times X$, which is equivalent to it representing the identity under global tensor.

\begin{lemma}\label{lem:skyscraper}
  For any $\tl{\bu}{D} \in \tl{X}{\GG}$ we have
  \begin{align*}
    \tltr{\bu}{PX}{\bu}\fwlp \tl{\bu}{D}\simeq \tltr{\bu}{PX}{\bu} \fwrp \tr{\overline{D}}{\bu}.
  \end{align*}
\end{lemma}

The intuitive idea is that parallel transport along the path connects the two local systems. Indeed, the two maps to $X$ in the Serre fibration $PX \to X\times X$ are homotopic - so pullback of a local system with either of these should give equivalent results.

\begin{proof}
  By Corollary~\ref{cor:article:3a} we may prove this using cosimplicial local systems. That is we want to prove that
  \begin{align*}
    \blbr{\bu}{PX}{\bu} \fwcs \bl{\bu}{A} \simeq  \blbr{\bu}{PX}{\bu} \fwcs \br{\overline A}{\bu}.
  \end{align*}
  For cosimplicial local systems these mean the functors that sends $(\alpha,\beta)$ to
  \begin{align*}
    \blbr{\alpha}{PX}{\beta} \otimes \bl{\alpha}{A} \cong \bl{\alpha}{A} \otimes \blbr{\alpha}{PX}{\beta} \quad \textrm{and} \quad  \blbr{\alpha}{PX}{\beta} \otimes \bl{\beta}{A}
  \end{align*}
  respectively are weakly equivalent. We prove this using that global tensoring with $\bl{\bu} PX_\bu$ from left or right is the identity (Lemma~\ref{lem:article:16}). Indeed, tensoring the first from the left and the second from the right with $\bl{\bu} PX_\bu$ yields functors both isomorphic to
  \begin{align*}
    (\alpha,\beta) \mapsto \tot \blbr{\alpha}{PX}{\bu} \fwcs \bl{\bu}{A} \fwcs \blbr{\bu}{PX}{\beta}.
  \end{align*}
\end{proof}

\section{Projection formula}
In this section we prove the following proposition and corollary. For any $\tl{\bu}{E} \in \tl{W}{\GG}$ where $W$ is a compact Liouville domain, we will by abuse of notation let $\tl{\bu}{E}$ denote the restriction to any subspace, in all cases a Lagrangian. 

\begin{proposition}\label{prop:projection_formula}
  For exact Lagrangians $K,L \subset W$ together with $\tr{D}{\bu} \in \tr{\GG}{K}$, $\tr E{\bu} \in \tr{\GG}W$ and $\tl{\bu}{F} \in \tl{L}{\GG}$ we have a quasi isomorphism
  \begin{align*}
     f:CF_*(K^{\str{D}{\bu}},L^{\stl\bu{\overline E}\fwlp\stl{\bu}{F}})\to CF_*(K^{\str{D}{\bu}\fwrp \str{E}{\bu}},L^{\stl{\bu}{F}}).
  \end{align*}
\end{proposition}

The idea of the proof is simply put that parallel transport along one side of the strips is homotopic to parallel transport along the other side, but this requires the local system - in this case $\tr E{\bu}$ to be defined over the entire target of the discs.

\begin{proof}
  Similar to the construction in Section~\ref{section:continuation} consider
  \begin{align*}
    &\widetilde{K}=K\times \{(t,t^2-t)\}\subset W\times \R^2,\\
    &\widetilde L=L\times \R\times \{0\}\subset W\times \R^2
  \end{align*}
  and the local systems $\tr{\widetilde D}\bu\in \tr \GG{\widetilde K}$, $\tr{\widetilde E}\bu\in \tr \GG{W\times \R^2}$ and $\tl\bu{\widetilde F}\in \tl {\widetilde L}\GG$ defined by pullback. Consider the Floer complex $CF_*(\widetilde{K}^{\str{\widetilde{D}}{\bu}\fwrp \str{\widetilde{E}}{\bu}},\widetilde L^{\stl{\bu}{\widetilde F}})$ but slightly modify the external differentials. Instead of always parallel transporting $\tr{\widetilde E}\bu$ along $\widetilde K$ we do the following. By monotonicity all discs $u:D^2\to W\times \R^2$ either have image in $W\times \{0\}\times\{0\}$, have image in $W\times\{1\}\times\{0\}$, or maps $-1$ to $W\times \{0\}\times \{0\}$ and maps $1$ to $W\times \{1\}\times \{0\}$. For discs whose image is in $W\times \{0\}\times \{0\}$ we parallel transport $\tr{\widetilde E}\bu$ along $\widetilde K$ and for discs whose image is in $W\times \{1\}\times \{0\}$ we parallel transport $\tr{\widetilde E}\bu$ along $\widetilde L$. In both cases we use the arc length parametrization of the paths in $W\times \R^2$. For discs such that $u(-1)\in W\times\{0\}\times\{0\}$ and $u(1)=W\times\{1\}\times\{0\}$ we choose the representative of $u$ such that $u(0)\in W\times\{\frac 12\}\times \R$ and parallel transport $\tr{\widetilde E}\bu$ along the arc length parametrization of the image under $u$ of the path $e^{it}$, $t\in [\pi,3\pi/2]$ followed by $it$, $t\in [-1,1]$ followed by $e^{-it}$, $t\in [-\pi/2,0]$. This way the differential still square to zero and takes the form
  \begin{align*}
    \begin{pmatrix}
      (\mu_{1})_{CF_*(K^{D^\bu\fwrp E^\bu},L^{\lef{\bu}F})}&f\\
      0&-(\mu_{1})_{CF_*(K^{D^\bu},L^{\lef\bu{\overline E}\fwlp\lef{\bu}F})}
    \end{pmatrix}
  \end{align*}
  for some chain map $f:CF_*(K^{\str{D}{\bu}},L^{\stl\bu{\overline E}\fwlp\stl{\bu}{F}})\to CF_*(K^{\str{D}{\bu}\fwrp \str{E}{\bu}},L^{\stl{\bu}{F}})$ which is a quasi isomorphism since it is plus or minus the identity on each generator on page one of the spectral sequence for the action filtration (argument as in the proof of Lemma~\ref{lem:article:constant_isotopy}).
\end{proof}

\begin{corollary}\label{cor:projection_formula}
  Let $K,L\subset W$ be exact Lagrangians and $\tl\bu{D}\in \tl L{\GG}$ and $\tl{\bu}{E}\in \tl{W}{\GG}$. Let $F=\tltr{\bu}{\FF(K,L)}{\bu}\und{PL}\otimes (-):\tl{L}{\GG}\to \tl{K}{\GG}$ then
  \begin{align*}
    F(\tl{\bu}{D}\fwlp \tl{\bu}{E})\simeq F(\tl{\bu}{D})\fwlp \tl{\bu}{E}.
  \end{align*}
\end{corollary}
\begin{proof} By using Proposition~\ref{prop:projection_formula} and Lemma~\ref{lem:skyscraper} we get
  \begin{align*}
    F(\tl{\bu}{D}\fwlp \tl{\bu}{E})&=CF_*(K^{\stltr{\bu}{PK}{\bu}},L^{\stl{\bu}{D}\fwlp \stl{\bu}{E}})\\
    &\simeq CF_*(K^{\stltr{\bu}{PK}{\bu}\fwrp \str{\overline E}{\bu} },L^{\stl{\bu}{D}})\\
    &\simeq CF_*(K^{\stltr{\bu}{PK}{\bu}\fwlp\stl{\bu}{E}},L^{\stl{\bu}{D}})\\
    &\cong CF_*(K^{\stltr{\bu}{PK}{\bu}},L^{\stl{\bu}{D}})\fwlp \tl{\bu}{E}\\
    &= F(\tl{\bu}{D})\fwlp \tl{\bu}{E}
  \end{align*}
  where we in the isomorphism on the next to last row can move out $\tl \bu E$ from the direct sum since the external differentials are not acting on it.
\end{proof}

\section{Simplicial Morse functions} \label{sec:simpl-morse-smale}

In this section we describe how to construct a Morse function $f$ and a Riemannian structure $g$ on $X$ so that we get the needed control over pseudo holomorphic strips to argue later that we can define small cosimplicial models for local systems coming from Floer theory.

\begin{definition}
  We call a smooth co-dimension 0 submanifold $X' \subset X$ an \emph{insulator} with respect to a Riemannian structure $g$ if there is a tubular neighborhood of its boundary $(\partial X')\times (-1,1) \subset X$ in which $g$ is on product form. We further say that $X'$ is an \emph{insulator} for $f : X\to \R$ if $f$ is a smooth function so that $\nabla f$ is transversely outwards pointing at $\partial X'$.
\end{definition}

The main idea here is that of course this makes it impossible for negative gradient trajectories to exit $X'$. However, due to how pseudo holomorphic curves behave in relation to the Riemannian structure we will see that neither can pseudo holomorphic curves with boundary on the zero section and $df$. We prove some of these and more general statements in appendix~\ref{sec:monotonicity}.

By a smooth triangulation $|S| \cong X$ of a smooth closed manifold $X$ we will mean a semi-simplicial set together with a homeomorphism $|S|\cong X$ such that each simplex inclusion $|\alpha| \to X$ is smooth.
\begin{definition}
  By an \emph{insulated simplicial Morse function} with respect to $S$ we will mean tuple $(f,g,\{X(\alpha)\}_{\alpha \in S})$ such that
  \begin{itemize}
  \item $f:X \to \R$ is a self indexing Morse function with a critical point of index $\deg(\alpha)$ at each barycenter of $\alpha\in S$ and no other critical points.
  \item $g$ is a Riemannian structure on $X$.
  \item Each $X(\alpha)$ is a contractible insulator for $f$ with respect to $g$ containing $|\alpha|$ in its interior and the boundary of $X(\alpha)$ is diffeomorphic to a sphere.
  \item Each $X(\alpha)$ contains no other critical points of $f$ than those at the barycenters of $|\alpha|$.
  \item For $\beta \subset \alpha \in S$ we have that $X(\beta) \subset \mathring X(\alpha)$.
  \item There is a $\delta_0 \in (0,\tfrac12)$ so that 
    \begin{align*}
      f(X(\alpha_0)) < \delta_0 
    \end{align*}
    for each 0-simplex $\alpha_0$ and
    \begin{align*}
      f(\partial X(\alpha)) > \delta_0 
    \end{align*}
    for all higher dimensional simplices $\alpha$.
  \item There is a top simplex $\alpha \in S$ with a point $x \in |\alpha|$ so that
    \begin{align*}
      1 < f(x) < n
    \end{align*}
   and $x$ is not contained in $X(\beta)$ for any $\beta\neq \alpha$.
  \end{itemize}  
\end{definition}
The two last bullet points are technical conditions which can seem a bit arbitrary but will be convenient later when we have to glue cells onto $D^*X$. The remaining goal of this section is to prove the following.

\begin{proposition} \label{prop:article:1}
  Any smooth triangulation $S$ of a closed smooth manifold $X$ admits an insulated simplicial Morse function.
\end{proposition}

For a continuous function $\tilde f: X \to \R$ and a continuous vector field $\tilde Y$ on $X$ we say that $\tilde Y$ is a pseudo gradient for $\tilde f$ if $\tilde f$ is strictly increasing along the non-constant flow lines of $\tilde Y$. We call the zeros of $\tilde Y$ the critical points of $\tilde f$ with respect to $\tilde Y$. As a first approximation to a simplicial Morse function we start by constructing such $C^0$ function and pseudo gradient so that the critical points are isolated at the barycenters of $S$.

Let $bS$ denote the barycentric subdivision of $S$, which satisfies $|bS|\cong |S| \cong X$. Note that this is also a smooth triangulation. Its set of $k$-simplices is defined as the set of sequences
\begin{align*}
  \alpha_{i_0} \subsetneq \alpha_{i_1} \subsetneq \cdots \subsetneq \alpha_{i_k}  
\end{align*}
with each $\alpha_{i_j}$ a simplex in $S$. Here the barycenters of each $\alpha_{i_j}$ are the corners of the new simplex in $bS$. On such a $k$-simplex in $bS$ we will use the canonical decreasing coordinates $1=t_0 \geq t_1 \geq \cdots \geq t_k \geq 0=t_{k+1}$, where $t_i=1,t_{i+1}=0$ corresponds to the corner at the barycenter at $\alpha_{i_k}$. Let $h : [0,1] \to [0,1]$ be a smooth homeomorphism such that $h(t)=t^2$ for $t$ close to $0$ and $h(t)=1-(1-t)^2$ for $t$ close to 1 and with nonzero derivative in the interior. We define
\begin{align*}
  \tilde f(t) = \sum_{i=1}^k h(t_i)
\end{align*}
using the standard coordinates $1\geq t_1 \geq \cdots \geq t_n\geq 0$ on each top dimensional simplex in $bS$. As $h$ has vanishing derivatives close to $0$ and $1$ it follows that the gradient (with respect to these standard coordinates) of $\tilde f$ on each top simplex is parallel to each of its subfaces.

Using standard coordinates on \emph{any} simplex $\beta \in bS$ represented as a sequence above this formula restricts to
\begin{align*}
  \tilde f(t) = \deg(\alpha_{i_0}) + \sum_{j=1}^{k} (\deg(\alpha_{i_j})-\deg(\alpha_{i_{j-1}}))h(t_j).
\end{align*}
It follows that the gradient of $\nabla \tilde f$ defined on each top simplex using the standard coordinates restricts to the similarly defined gradient on each subface. This thus defines a global continuous pseudo-gradient $\tilde Y$ for $\tilde f$ on all of $X$.

To construct the smooth functions from these (and later the insulators) we will need to be able to smoothen topological co-dimension 1 submanifolds. So for any co dimension 1 \emph{topological} submanifold $N \subset X$ we call a vector field $Y$ transverse to $N$ if there is an $\epsilon>0$ so that the flow map $N\times (-\epsilon,\epsilon) \to X$ for $Y$ is an embedding of a normal bundle, whose image we denote $U$. This identifies $N$ with the flow lines in $U$ and in the case $Y$ is smooth this means that we can give $N$ a canonical smooth structure. This makes $U\to N$ a smooth fiber bundle. The embedding of $N \subset U$ then corresponds to a section which is only smooth if $N$ was actually a \emph{smooth} submanifold to begin with. Replacing $N$ with the image of a $C^0$ close smooth section is what we will call a \emph{smoothing} of $N$ using $Y$. If $N$ is the boundary of a co-dimension 0 topological submanifold $X'\subset X$ we will by a smoothing of $X'$ using $Y$ mean the obvious smooth manifold with boundary a smoothing of $N$.  

\begin{lemma} \label{lem:article:6}
  There exists a self indexing smooth Morse function $f:X\to \R$ and a smooth pseudo gradient $Y$ on $X$ such that:
  \begin{itemize}
  \item $f$ has its critical points at the barycenters (the zeros of $\tilde Y$).
  \item $(f,Y)$ is arbitrarily $C^0$ close to $(\tilde f,\tilde Y)$.
  \end{itemize}
\end{lemma}

\begin{proof}
  Pick any Riemannian structure $g$ on $X$. Given $\epsilon_0>0$ we will find such $f$ and $Y$ so that
  \begin{align*}
    |f-\tilde f| \leq 10\epsilon_0 \quad \textrm{and} \quad \norm{Y-\tilde Y}_{C^0} \leq 10\epsilon_0.
  \end{align*}
  By construction of $\tilde f$ and $\tilde Y$ it follows by considering each simplex separately that given any neighborhood $U$ around all the critical points there is an $\epsilon>0$ such that any vector field $Y'$ with $\norm{Y'-\tilde Y}_{C^0} < \epsilon$ we have that $Y'$ is a pseudo gradient for $\tilde f$ away from $U$. Note that this is not true for general continuous functions and pseudo gradients.
  
  Given a small $r>0$ we let $B_r(x)$ denote the open ball of radius $r$ with center $x$. Let $B_r$ denote the union of these balls for all barycenters $x$. Pick in order a small $\epsilon>0$, an even smaller $\delta >0$ and a smooth vector field $Y'$ such that
  \begin{itemize}
  \item $\delta < \min(\epsilon,\epsilon_0)$.
  \item $\norm{Y'-\tilde Y} \leq \epsilon_0$.
  \item For each barycenter $x$ we have $|\tilde f-\tilde f(x)| \leq \epsilon_0$ and $\norm{\tilde Y} \leq \epsilon_0$ on $B_\epsilon(x)$.
  \item The zeros of $Y'$ are within $B_\delta$.
  \item $Y'$ is transverse to all the level-sets of the type $\{\tilde f = i \pm \delta\}$.
  \item Given any point in $y\in \{i-\delta < \tilde f < i+\delta\}\setminus B_\epsilon$ the flow of $\pm Y'$ takes $y$ to $\{\tilde f=i\pm \delta\}$ (respectively).
  \end{itemize}
  The second last condition is possible by the remark in the first paragraph above. To get the last condition we make $\delta$ so small compared to $\epsilon$ that the flow of $\pm \tilde Y$ from any $B_\delta(x)$ does not leave $B_\epsilon(x)$ before it reaches the value $i \pm \delta/2$. We then also choose $Y'$ so close to $\tilde Y$ that this same condition is still satisfied with the function value replaced by $i\pm \delta$.

  We now construct $f$ and $Y$ by induction in $i=0,\dots,n$ where the induction assumption is that these have been defined on a smoothing of $\{\tilde f\leq i-1+\delta\}$ using $Y'$ and that $Y$ and $Y'$ agree in the tubular neighborhood of the boundary in which the smoothing happens. The base case of a zero simplex uses that $\{\tilde f\leq \delta\}$ is a disjoint union of balls (up to homeomorphism). Hence the smoothing of these are disjoint smooth balls. We may assume that this smoothing happens within the set $\{\tilde f < \epsilon\}$. On each of these smoothed balls we can find a function with a single critical point (with critical value $0$) which has the outwards pointing (and smooth) $Y'$ as pseudo gradient near the smoothed boundary. In fact we can assume that this $f=\delta$ on this smoothed boundary. Interpolating $Y'$ inside the balls to become a small rescaling of the gradient of this $f$ defines $Y$. Note that by the bounds above we have $|f-\tilde f| \leq \delta + \epsilon_0 < 2\epsilon_0$ on each ball.  Similarly we have $\norm{Y' - \tilde Y} < \epsilon_0$ and $\norm{\tilde Y} \leq \epsilon_0$ which implies the wanted bound on $Y$ if we scale the gradient term to be small enough.

  For the induction step we first extend the field $Y$ to a smoothing of $\{\tilde f \leq i-\delta\}$ (using $Y'$) by simply using $Y=Y'$ in the (smoothed) region $\{i-1+\delta \leq \tilde f \leq i-\delta\}$. In this region we can also extend $f$ by simply using the flow of $Y=Y'$ normalized to increase the function from $i-1+\delta$ to $i-\delta$. This way $Y$ automatically becomes a pseudo gradient. We can even make sure that this $f$ is close to $\tilde f$. Indeed, as any continuous increasing function can be approximated by a smooth increasing function with positive derivative we can find a \emph{continuous} function $s$ on the region so that using the flow of $sY'$ we define a function which $C^ 0$ approximates $\tilde f$. replacing $s$ by a $C^ 0$ close smooth function $s'$ we use the flow of $s'Y$ to define $f$ in the region. 

  We cannot do the exact same in the (smoothed) region $\{i-\delta \leq \tilde f \leq i+\delta\}$. Indeed, in that region $Y'$ has zeros. However, by the third bullet point above we can pick a smooth function $f'$ on the smoothed region such that: whenever a flow line (in this region) of $Y'$ is not fully contained in one of the $\epsilon$ balls we have that $f$ increases along this flow line from what value it had on the smoothing of $\{\tilde f = i-\delta\}$ (which can be assumed arbitrarily close to $i-\delta$) to the value $i+\delta$ on the smoothed $\{\tilde f= i+\delta\}$. In fact we pick $f'$ such that $f'=i+\delta$ on the entire smoothed $\{\tilde f=i+\delta\}$. This $f'$ may initially have many critical points inside each $\epsilon$-ball. However, we can cancel them all but one as the local change of homotopy type is a standard handle attachment. Applying an ambient isotopy with support inside the ball we may move this critical point to the barycenter. The fact that the values of $\tilde f$ and this $f$ are within $[i-\delta,i+\delta]$ makes them $C^0$ close enough to get the wanted $C^0$ bound.

  We can again define $Y=Y'$ outside of the balls but then interpolate to a rescaling of the gradient of $f$ inside the balls. The rescaling is again so that we can control the $C^0$ distance of this $Y$ to $\tilde Y$. Indeed, $\tilde Y$ is again sufficiently small inside the $\epsilon$-balls for this to work.
\end{proof}

\begin{proof}[Proof of Proposition~\ref{prop:article:1}]
  
  Let $\beta \in (bS)_n$ again be a top face represented by a sequence $\alpha_0 \subsetneq \cdots \subsetneq \alpha_n$. The $d$th front face $F_d(\beta) \subset \beta$ is the simplex represented by the sequence
  \begin{align*}
    \alpha_0 \subsetneq \cdots \subsetneq \alpha_d.
  \end{align*}
  Let $T_d^{\epsilon}(\beta) \subset |\beta|$ be the open subset defined by those sequences $1\geq t_1\geq \cdots \geq t_n\geq 0$ where $t_d< \epsilon$. This is a neighborhood in $|\beta|$ around $|F_d(\beta)|$.

  For each $\alpha\in S$ and $\epsilon>0$ we define the open neighborhood $N_\epsilon(|\alpha|)$ in $X$ around $|\alpha|$ by the formula
  \begin{align*}
    N_\epsilon(\alpha)=\cup_{|F_d(\beta)| \subset |\alpha|} T_d^{\epsilon}(\beta).
  \end{align*}
  Such a neighborhood could be defined purely using the simplices in $S$. However, with this definition it is clear that when any $t_d=\epsilon$ on the boundary the flow of $\tilde Y$ takes it directly out of the set and the negative flow of $\tilde Y$ takes it into the set. Indeed, the flow of $\tilde Y$ on a top simplex in $bS$ strictly increases all $t_i$ which are not $0$ or $1$. This means that $\tilde Y$ is transverse to the boundary of this neighborhood in the sense described above. This also proves that the boundary is a topological manifold of codimension 1 in $X$. It is also not difficult to verify that $N_\epsilon(\alpha)$ is homeomorphic to $D^n$. 

  We thus construct a $C^0$ approximation to $X(\alpha)$ by fixing small numbers $0<\epsilon_0<\epsilon_1<\cdots <\epsilon_n$ and defining
  \begin{align*}
    \tilde X(\alpha) = N_{\epsilon_{\deg(\alpha)}}(\alpha).
  \end{align*}
  The ``radius'' depends on the dimension of the simplex so that we have $\tilde X(\beta)$ is inside the interior of $\tilde X(\alpha)$ when $\beta \subsetneq \alpha$. Note, in particular that having all $\epsilon_i, i<n$ small enough we can assume for any top simplex $\alpha$ the existence of $x\in |\alpha|$ with $x\notin X(\beta)$ for $\beta\neq \alpha$ and such that $1 < \tilde f(x) <n$. Also, by then making $\epsilon_0$ even smaller compared to the other $\epsilon_i$ we may assume the existence of $\delta_0$ so that
  \begin{align*}
    \tilde f(\tilde X(\alpha_0)) < \delta_0
  \end{align*}
  for $\alpha_0$ a 0-simplex and 
  \begin{align*}
    \tilde f(\partial \tilde X(\alpha)) > \delta_0
  \end{align*}
  for other simplices.

  By Lemma~\ref{lem:article:6} we may replace $\tilde f$ and $\tilde Y$ by $C^0$ close $f$ and $Y$, where $f$ is Morse. We may make these close enough to not change the above bounds and the existence of $x$ in some top simplex $\alpha$. Again by the specific definition of $\tilde Y$ and the above neighborhood we can make sure that $Y$ is also transverse to the boundaries of $\tilde X(\alpha)$. Indeed, this follows by considering each top simplex in $bS$ at a time. It follows that we can replace each $\tilde X(\alpha)$ by domains $X(\alpha)$ whose boundaries are smoothings of $\partial \tilde X(\alpha)$ using $Y$. By constructions of smoothings we have that $Y$ is still transverse to their boundaries, and by picking these smoothings sufficiently close to the old boundaries we can retain the property that $X(\beta) \subset \mathring X(\alpha)$ when $\beta \subsetneq \alpha$, the bound using $\delta_0$ above and the existence of $x$.

  We are left with producing a Riemannian structure $g$ satisfying the requirements. To this end, we perturb each $X(\alpha)$ so that their boundaries are all transverse to each other (in the sense that the intersection of the boundaries of $k$ of them is a codimension $k$ transverse intersection) and still have the needed conditions satisfied. This provides a smooth stratification of $X$. Indeed, the $n-k$ dimensional strata (which are open manifolds) we denote $N_{n-k}$ and we define it to be the intersection locus of precisely $k$ of these boundaries not including the intersection of $k+1$ of them. In particular the top strata of dimension $n$ is that which is not part of any boundary.

  The first goal is to construct $g$ so that each $X(\alpha)$ is an insulator (after that we will modify $f$ to make sure they insulate for $f$). We start by picking any Riemannian structure $g$ on $X$ and then we will modify it inductively to make all $X(\alpha)$ insulators. Pick smooth functions $f_\alpha : X \to \R$ so that $0$ is regular and $\partial X(\alpha)=\{f_\alpha=0\}$ with $df_\alpha$ positive on outwards pointing vectors. Given distinct $\alpha_1,\dots,\alpha_n \in S$ (each of arbitrary simplex degree) and $x\in \cap_i \partial X(\alpha_i) \subset N_0$ we get that
  \begin{align*}
    (f_{\alpha_1},\dots,f_{\alpha_n}) : X \to \R^n
  \end{align*}
  is locally around $x$ a chart. We thus modify $g$ to be the standard flat Riemannian structure in a small neighborhood around $x$ using this chart. This structure does not depend on the order of the simplices $\alpha_i$. We induce such local flat structures around any $x\in N_0$ using these charts. Let $V_0$ be a small neighborhood around $N_0$ contained well within the flat area. Now consider distinct $\alpha_1,\dots,\alpha_{n-1}\in S$ and the compact 1 manifold $N_1'=\cap_i \partial X(\alpha_i)\setminus V_0 \subset N_1$. Let $\nu \to N_1'$ be its normal bundle (using the modified $g$). Pick an $\epsilon >0$ so that the exponential function is a diffeomorphism of $D_\epsilon \nu$ onto its image. Let $U_\epsilon$ be this image. Consider the map
  \begin{align*}
    \Psi : U_\epsilon \to N_1'\times \R^{n-1}
  \end{align*}
  given by taking the closest point in $N_1'$ (the source point of the exponential function) and the functions $f_{\alpha_i}$. By possibly making $\epsilon$ smaller, this is a smooth embedding and we may assume that close to the points in $\partial N_1' \in V_0$ the tubular neighborhood is contained within the neighborhood where we already made the Riemannian structure flat. It follows that in such a neighborhood of $\partial N_1'$ the restriction of $g$ is under $\Psi$ identified with a Riemannian structure on $N_1'$ times the flat structure on $\R^{n-1}$. We then extend this to all of $N_1'$ by modifying $g$ to be such a product structure in a neighborhood of $N_1'$. We may do this for all the components of $N_1$. We now pick a small neighborhood $V_1$ around $N_1$ where the structure is on product form.

  Generally in the $k$th inductive step we consider $\alpha_1,\dots,\alpha_{n-k}$ and $N_k'=\cap_i \partial X(\alpha_i)\setminus (\cup_{i=0}^{k-1}V_i) \subset N_k$. We need not care whether $\partial N_k'$ is a manifold. Again we pick a small $\epsilon>0$ and tubular neighborhood $U_\epsilon$ over $N_k'$ (using the image of the exponential function) and consider the map
  \begin{align*}
    \Psi : U_\epsilon \to N_k'\times \R^{n-k}
  \end{align*}
  given by the closest point in $N_k'$ and the functions $f_{\alpha_i}$. Again, we have by strong induction that close to $\partial N_{k-1}'$ we have that using $\Psi$ the structure $g$ is a product structure of some structure on $N_k'$ times the flat structure on $\R^{n-k}$. It follows that we again can extend this product structure over $U_\epsilon$ and modify $g$ to be this structure close to $N_k'$ and pick a small neighborhood $V_k$ as before.

  The only remaining problem is that even though $Y$ points to the right sides of each $\partial X(\alpha)$ we do not know that the gradient in $g$ does so. We fix this for each $X(\alpha)$ inductively by adding a function $h(f_\alpha)$ to $f$, where $h$ is a function as graphed in Figure~\ref{Fig:h}
  \begin{figure}[ht]
    \centering
    \begin{tikzpicture}
      \draw[->,dotted] (-4,0) -- (4,0);
      \draw[->,dotted] (0,-1) -- (0,1);
      \draw (-3,0) -- (-2,0) to[out=0,in=180] (-0.3,-0.3) to[out=0,in=180] (0.3,0.3) to[out=0,in=180] (2,0) -- (3,0);
      \draw (-2,0) node[below] {$-\delta$};
      \draw (2,0) node[below] {$\delta$};
    \end{tikzpicture}
    \caption{$h$} \label{Fig:h}
  \end{figure}
  such that
  \begin{itemize}
  \item $h'>-\delta$
  \item $h'(0)>C$ for arbitrary given $C$.
  \end{itemize}
  with this we see that
  \begin{align*}
    Y(f + h(f_\alpha)) = Y(f) + h'(f_\alpha)Y(f_\alpha)
  \end{align*}
  as $Y(f_\alpha)>0$ and $Y(f)>0$ close to $\partial X(\alpha)$ we can by picking $\delta$ small enough make sure that this is positive (and for the inductive step we make sure that $Y$ remains a pseudo gradient for the new function). We also see that (using how $g$ was constructed)
  \begin{align*}
    \nabla (f + h(f_\alpha)) = \nabla f +  h'(f_\alpha)\tfrac{\partial}{\partial f_\alpha}
  \end{align*}
  which for large $C$ is outwards pointing (positive in $\tfrac{\partial}{\partial f_\alpha}$) at $f_\alpha=0$ if we picked $C$ large enough.
\end{proof}

\section{\texorpdfstring{Construction of $W_X$ and skeletal push off}{Construction of W\_X and skeletal push off}} \label{sec:constr-w_x-skel}

In this section we give a bit more details on $W_X$ and describe how to push the skeleta of $W_X$ of itself except at critical points for a given function $f:X\to \R$. We analyze this specifically for an insulated simplicial Morse function $(f,g,\{X(\alpha)\}_{\alpha \in S})$. However, to define the ``small'' co-simplicial local systems in the next section we will need to consider some deformations of $f$ that cancel all critical points except one in a given simplex.

We first pick and fix an insulated simplicial Morse function $(f,g,\{X(\alpha)\}_{\alpha \in S})$ for a smooth triangulation $S$ such that $|S|=X$. In the introduction we described $W_X$ as given by inductively attaching sub-critical Weinstein handles to $D^*X$ over a small contractible ball $D^n \cong A \subset X$. Let $x \in|\alpha|$ satisfy the last condition in the definition of a simplicial Morse function. We can by applying an isotopy to $A$ (and simultaneously isotope the attaching maps) assume that $A$ is contained in a very small neighborhood around this $x$. In particular it is outside of all $X(\beta)$ for all simplices $\beta\neq \alpha$. We can even assume (by enlarging $A$ a bit and using an isotopy) that it is on the form $[a,b]\times D^{n-1} \xrightarrow{\varphi \cong} A \subset X$ with $f(\varphi(t,x))=t$ and $1<a<b<n$.

As we are attaching standard sub-critical Weinstein handles we may also assume that $W_X$ has a Liouville form $\lambda$ such that:
\begin{itemize}
\item The limit of the Liouville flow (called the skeleta) $S_X \subset W_X$ is given by attaching CW-handles of dimension strictly less than $n$ inductively to the zero section $X$ and previously attached cells.
\item The form $\lambda$ is the standard Liouville form $-pdq$ on $D^*X$.
\end{itemize}
Note that the last condition implies that the new handles in $S_X$ attaches to $X$ in the interior of $A$ (and inductively on each other). Hence the homotopy type of $S_X$ is $X$ wedge the space build by these new cells.

By construction $f$ has no critical points on $A$. Fix a smooth function $b : W_X \to [0,1]$ with compact support in the interior of $W_X$ and which equals $1$ on $D_{1/2}^*X$. Define
\begin{align} \label{eq:article:10}
  H_{f}(z) = b(z) f(r(z))
\end{align}
where $r : W_X \to X$ is a fixed smooth retraction equal to the projection $\pi : D^*X \to X$ on $D^*X$ and such that everything outside of $D^*X$ maps to $A$. Let $\psi^f_t : W_X \to W_X$ for $t\in \R$ be the compactly supported Hamiltonian isotopy given by the time $t$ flow associated to $X_{H_{f}}$ defined by $\omega(X_{H_f}, -)=-dH_f$. We denote
\begin{align*}
  A_W = W\setminus D^*(X\backslash A).
\end{align*}
Note that the flow on $A_W$ only depends on $f_{\mid A}$ (and our now fixed choice of function $b$ and retraction $r$).

\begin{lemma} \label{lem:article:17}
  There is an $\epsilon>0$ and arbitrary small Hamiltonian perturbations $\psi_{\pm \epsilon}$ of $\psi^f_{\pm \epsilon}$ such that
  \begin{itemize}
  \item $\psi_{\pm \epsilon}(q) = q \pm \epsilon df \in D^*X$ for $q\in X$.
  \item The intersection of any two of $\psi_{-\epsilon}(S_X),X,\psi_{\epsilon}(S_X)$ is exactly the set of critical points for $f$.
  \end{itemize}
  Furthermore, for fixed functions $b$ and $r$ as above, we may assume that this $\epsilon$ and the perturbation only depends on $f_{\mid A}$ and that the perturbation is supported in $A_W$. 
\end{lemma}

\begin{proof}
  We first pick so small $\epsilon>0$ so that $\epsilon df$ lies in $D_{1/2}^*X$ where $b=1$ and so small that the flow of the new cells $S_X-X$ stays within the interior of $A_W$ for the time $[-\epsilon,\epsilon]$. This implies that the flow of $S_X-X$ for these times only depend on $f_{\mid A}$. The two flows $\psi_{\pm \epsilon}^f$ now satisfy the first condition, which implies that any two of $\psi^f_{-\epsilon}(X),X,\psi^f_{\epsilon}(X)$ precisely intersect in the critical points of $f$. Thus any ``unwanted'' intersection points in the second condition must involve the new cells in $S_X$. This means that these intersections happen within $A_W$. Using an inductive argument (induction in the cells) we may perturb these less than half dimensional isotropic subspaces away from the isotropic skeleta. This can be done purely within the interior of $A_W$.
\end{proof}

\begin{lemma} \label{lem:article:7}
  For an insulated simplicial Morse function $(f,g,\{X(\alpha)\}_{\alpha \in S})$ and any $\alpha \in S$ with a chosen vertex $\alpha_0 \in \alpha$ there exists a smooth homotopy $f_s : X \to \R_{\geq0} , s\in I$ such that:
  \begin{itemize}
  \item $f_0=f$.
  \item $f_s$ equals $f$ outside a compact set in the interior of $X(\alpha)$ and on a neighborhood of $X(\alpha_0)$. In particular, $X(\alpha)$ and $X(\alpha_0)$ are insulators for each $f_s$.
  \item $f_1$ has no other critical points than $|\alpha_0|$ inside $X(\alpha)$·
  \end{itemize}
  If $\alpha$ is the top simplex containing $A$ then we can also assume that $f_s$ is constant on $A$.
\end{lemma}

\begin{proof}
  As both $X(\alpha_0) \subset X(\alpha)$ are contractible with sphere boundaries it follows that $\overline{X(\alpha)\setminus X(\alpha_0)}$ is an $h$-cobordism. As we also have $f(\partial X(\alpha_0)) < \delta_0 < f(\partial X(\alpha))$ and the gradient is transverse at both ends, we can by the $h$-cobordism theorem modify $f_0$ in the interior of $X(\alpha)\setminus X(\alpha_0)$ to an $f_1$ that has no critical points there.

  In the case where $|\alpha| \supset A$ we have as 
  \begin{align*}
    f(\partial X(\alpha_0)) < 1
  \end{align*}
  that this $f_1$ takes all values between $1$ and $n$. In particular we can find an embedding $\varphi' : [a,b] \times D^{n-1} \to X(\alpha)\setminus X(\alpha_0)$ on which $f_1(\varphi'(t,x))=t$. By applying an ambient isotopy in the interior of $X(\alpha)\setminus X(\alpha_0)$ we may assume that this coincides with the $\varphi$ above describing $A$. Hence this $f_1$ equals $f_0$ on $A$.
\end{proof}

\section{Cosimplicial Floer complexes}\label{sec:cosimplicial-floer-complexes}

In this section we assume that we have exact Lagrangians $K,L \subset W_X$ and we want to replace the Floer objects $\tltr{\bu}{\FF(K,X)}{\bu}$, $\tltr{\bu}{\FF(X,L)}{\bu}$ and $\tltr{\bu}{\FF(K,L)}{\bu}$ defined in Section~\ref{sec:floer-theory-with} by certain \emph{small} equivalent cosimplicial models. They will depend on large deformations of $L$ and $K$ pushed of the skeleta using the previous section. We also give an explicit and \emph{local} representative for $\mu_2$ from Section~\ref{sec:floer-theory-with} using these models.

In Section~\ref{sec:local-syst-homol} we use this to prove that $\tltr{\bu}{\FF(L,L)}{\bu}$ is equivalent to $\tltr{\bu}{PL}{\bu}$. This could have been proven earlier, but as we need the following anyway, it was convenient to wait.

\begin{definition}
  For \emph{any} Liouville domain $W$ and exact Lagrangians $K,L \subset W$ and simplicial structures $|S|\cong K$ and $|T|\cong L$ we use the equivalences from Equation~(\ref{eq:article:2}) to define
  \begin{align*}
    \bltr{\bu}{\FF(K,L)}{\bu}=\bltr{\bu}{PK}{\bu} \und{PK}{\otimes } \tltr{\bu}{\mathcal F(K,L)}{\bu}\cong CF_*(K^{\sbltr{\bu}{PK}{\bu}}, L^{\stltr{\bu}{PL}{\bu}})  \in \bltr{S}{\GG}{L}.
  \end{align*}
  and similar for $\tlbr{\bu}{\FF(K,L)}{\bu} \in \tlbr{K}{\GG}{T}$ and $\blbr{\bu}{\FF(K,L)}{\bu} \in\blbr{S}{\GG}{T}$. 
\end{definition}

\subsection{Small local models}

To define the large deformations that will replace the above defined local systems with smaller models we let $r_t : W_X \to W_X$ denote the negative Liouville flow which pushes any subspace of $W_X$ closer to the skeleta $S$ as $t\to \infty$. We define for large $t \gg 0$ the Lagrangians
\begin{align*}
  L_t^+ = \psi_{\epsilon}(r_t(L)) \quad \textrm{and} \quad K^-_t = \psi_{-\epsilon}(r_t(K))
\end{align*}
where $\psi_{\pm \epsilon}$ are the perturbations from Lemma~\ref{lem:article:17} used with an $f$ coming from an insulated simplicial Morse function $(f,g,\{X(\alpha)\}_{\alpha \in S})$ such that the underlying semi-simplicial set $S$ satisfies S1 from Section~\ref{sec:cosimplicial-local}. Note that inside $D_{1/2}^*(X\setminus A)$ (away from the attaching area $A$) we have that  $K^-_t = e^{-t}K -\epsilon df$ and $L_t^+ = e^{-t}L+\epsilon df$. When we apply Lemma~\ref{lem:article:17} to construct $K_t^-$ and $L_t^+$ we may assume that $t$ is so large that the intersection points of any pair of $K_t^+,X,L_t^-$ are arbitrarily close to the critical points of $f$. I.e. they are inside $D^*(X \setminus A)$ and close to the zero section.

As $A$ is away from the boundary of the insulators $\partial X(\alpha)$ we may pick $t$ large enough to make sure that all the lemmas in Appendix~\ref{sec:monotonicity} hold. To use the lemmas we also define the almost complex structure on $W_X$ by extending the one on $D^*X$ induced by the metric $g$ on $X$, but as usual we are implicitly taken arbitrarily small perturbations of the Lagrangians and $J$ to define the complexes. However, we may still assume that the results of the Lemmas from Appendix~\ref{sec:monotonicity} hold by taking these perturbations small enough.

For $\alpha \in S$ let
\begin{align*}
  \bltr{\alpha}{\FF'(X,L_t^+)}{\bu} = \quad\smashoperator{\bigoplus_{x \in X(\alpha)\cap L_t^+}} \Sigma^{|x|}\ \tltr{x}{P(L_t^+)}{\bu} \quad \subset \quad CF_*(X^{\str{\F}{\bu}},(L_t^+)^{\stltr{\bu}{P(L_t^+)}{\bu}}).
\end{align*}
Note that its total complex precisely gives all of $CF_*(X^{\str{\F}{\bu}},(L_t^+)^{\stltr{\bu}{P(L_t^+)}{\bu}})$. We also define the following in between version
\begin{align*}
  \bltr{\alpha}{\tilde \FF(X,L_t^+)}{\bu} = \,\,\smashoperator{\bigoplus_{x \in X(\alpha)\cap L_t^+}}  \Sigma^{|x|} \bltr{\alpha}{PX(\alpha)}{x} \otimes \tltr{x}{P(L_t^+)}{\bu} \subset \bltr{\alpha}{\FF(X,L_t^+)}{\bu}.
\end{align*}
By Lemma~\ref{lem:article:monotonicity1} the differential in both $\bltr{\alpha}{\FF'(X,L_t^+)}{\bu}$ and $\bltr{\alpha}{\tilde \FF(X,L_t^+)}{\bu}$ stay within the defined subspace. By construction they are both pre-cosimplicial local systems.

\begin{lemma}\label{lem:article:F_are_cosimplicial}
  Both $\bltr{\bu}{\FF'(X,L_t^+)}{\bu}$ and $\bltr{\bu}{\tilde \FF(X,L_t^+)}{\bu}$ are cosimplicial local systems in $\bltr{S}{\GG}{L_t^+}$.
\end{lemma}

\begin{proof}
  For a simplex $\alpha \in S$ and a 0-simplex $\alpha_0 \subset \alpha$ we have the inclusion
  \begin{align*}
    \bltr{\alpha_0}{\FF'(X,L_t^+)}{\bu} \subset \bltr{\alpha}{\FF'(X,L_t^+)}{\bu}
  \end{align*}
  and we want to prove that this is a weak equivalence. We may use Lemma~\ref{lem:article:7} to get a homotopy $f_s$. Using $f_s$ we may use this to push off $L$ defining $L_{t,s}^+$, which is an isotopy of Lagrangians. This isotopy is constant over $X(\alpha_0)$, over the complement of $X(\alpha)$, and inside $A_W$. The isotopy removes all the intersection points of $X$ and $L_t^+$ inside $X(\alpha)$ except those in $X(\alpha_0)$. As $f_s$ is constant in a neighborhood of $\partial X(\alpha)$ and $\partial X(\alpha_0)$ it follows that the continuation maps satisfy the same monotonicity lemmas with respect to $X(\alpha)$ and $X(\alpha_0)$ as in Appendix~\ref{sec:monotonicity}. This implies by the results in Section~\ref{section:continuation} that the sub complexes defined using the intersection points inside either of these are preserved up to quasi isomorphism.

  Since all the intersections $X\cap L_{t,1}^+$ inside $X(\alpha)$ are also inside $X(\alpha_0)$ we see that these complexes are the same for $s=1$, and hence the original inclusion was a weak equivalence. Indeed, the continuation map on the sub-complex is the restricted continuation map so we have a commuting diagram
  \begin{center}
    \begin{tikzcd}[column sep=40]
      \bltr{\alpha}{\FF'(X,L_t^+)}{\bu}&\bltr{\alpha}{\FF'(X,L_{t,1}^+)}{\bu}\ar[swap]{l}{C_{X,L_{t,s}^+}}\\
      \bltr{\alpha_0}{\FF'(X,L_t^+)}{\bu}\ar[hook]{u}&\bltr{\alpha_0}{\FF'(X,L_{t,1}^+)}{\bu}\ar[equal]{u}\ar[swap]{l}{C_{X,L_{t,s}^+}}
    \end{tikzcd}
  \end{center}
  where both continuation maps are weak equivalences.
  
  The argument for $\bltr{\bu}{\tilde \FF(X,L_t^+)}{\bu}$ is similar.
\end{proof}

\begin{lemma} \label{lem:article:2-old-prop}
  We have canonical weak equivalences
  \begin{align*}
    \bltr{\bu}{\FF'(X,L_t^+)}{\bu}  \leftarrow \bltr{\bu}{\tilde \FF(X,L_t^+)}{\bu} \to \bltr{\bu}{\FF(X,L_t^+)}{\bu}.
  \end{align*}
\end{lemma}

\begin{proof}
  
 The map in the lemma going left is induced by the augmentation $\epsilon : \bltr{\alpha}{PX(\alpha)}{x} \to \tr{\F}{x}$. We see that this is a weak equivalence by using the action filtration (perturbed to make each intersection action unique). Indeed, as $\bltr{\alpha}{PX(\alpha)}{x}$ is contractible this induces a quasi isomorphism on the filtered quotients (one summand for each intersection point $x$).

  The map going right is simply the inclusion from the above definition of $\bltr{\bu}{\tilde \FF(X,L_t^+)}{\bu}$. To see that this is an equivalence one has to consider it more ``globally''. To this end it is convenient to show that we have a weak equivalence after applying the functor $\tl{\bu} PX_{\bu} \und{S}\otimes (-)$ (using Proposition~\ref{prop:pi_0_equivalence}). The left hand side of this equation becomes the total complex of
  \begin{align*}
    \alpha \mapsto \bigoplus_{x \in X(\alpha)\cap L_t^+} \Sigma^{|x|}\ \tl{\bu} PX_{\alpha} \otimes \bltr{\alpha}{PX(\alpha)}{x} \otimes \tltr{x}{P(L_t^+)}{\bu}
  \end{align*}
  while the right hand side becomes the total complex of 
  \begin{align*}
    \alpha \mapsto \bigoplus_{x \in X\cap L_t^+} \Sigma^{|x|}\ \tl{\bu} PX_{\alpha} \otimes \bltr{\alpha}{PX}{x} \otimes \tltr{x}{P(L_t^+)}{\bu}.
  \end{align*}
  We now again use the action filtration and prove that for fixed $x$ the inclusion is a weak equivalence. Note, however, that for fixed $x$ the first is only a pre-cosimplicial local system. The statement follows from that the inclusion
  \begin{align*}
    \tot \pare*{ \tl{\bu} PX_{\alpha} \otimes \bltr{\alpha}{PX(\alpha)}{x} } \subset \tot (\tl{\bu} PX_{\alpha} \otimes \bltr{\alpha}{PX}{x}) \simeq \tl{\bu} PX^x
  \end{align*}
  is a weak equivalence. Indeed, the first only has contributions from the contractible semi simplicial set $s(x)$ (star of $x$) defined by those simplices $\alpha$ where $x\in X(\alpha)$. Note that $s(x)$ is contractible by S1.
\end{proof}

We may similarly define
\begin{align*}
  \tl{\bu} \FF'(K_t^-,X)_\alpha = \quad\smashoperator{\bigoplus_{x \in X(\alpha)\cap K_t^-}}  \Sigma^{|x|}\ \tl{\bu} P(K_t^-)^x \quad \subset \quad CF_*((K_t^-)^{\stl{\bu} P(K_t^-)^\bu},X^{\stl \bu\F}).
\end{align*}
and
\begin{align*}
  \tl{\bu}{\tilde \FF}(K_t^-,X)_\alpha = \quad\smashoperator{\bigoplus_{x \in X(\alpha)\cap K_t^-}}  \Sigma^{|x|}\ \tl{\bu} P(K_t^-)^x \otimes \tlbr{x}{PX(\alpha)}{\alpha} .
\end{align*}
We also have similar weak equivalences:
\begin{align*}
  \tl{\bu} \FF'(K_t^-,X)_\bu \leftarrow \tl{\bu}{\tilde \FF}(K_t^-,X)_\bu \to \tl{\bu}{\FF(K_t^-,X)}_\bu.
\end{align*}
We will define an explicit representative of the important co-product
\begin{align*}
  \mu_2 : \tl{\bu} \FF(K_t^-,L_t^+)^\bu \to \tl{\bu} \FF(K_t^-,X)^\bu \und{PX}\otimes \tl{\bu} \FF(X,L_t^+)^\bu
\end{align*}
using these small models (and the functors from Section~\ref{sec:equiv-glps-gcss} relating the two types of local systems). We will consider the following important cosimplicial local system taking values in bi-path local systems.
\begin{definition} \label{def:article:1}
  We define $\tlar{\bu}{\FF(K_t^-,L_t^+)}{\bu}{\bu} \in \tlar{K_t^-}{\GG}{L_t^+}{S}$ for large $t$ over $\alpha \in S$ as the sub complex 
  \begin{align*}
    \quad\smashoperator{\bigoplus_{x \in K_t^-\cap L_t^+  \cap r^{-1}X(\alpha)}}  \Sigma^{|x|}\, \tl{\bu} P(K_t^-)^x \otimes \tltr{x}{P(L_t^+)}{\bu} \quad \subset \quad \tltr{\bu}{\FF(K_t^-,L_t^+)}{\bu}
  \end{align*}
  where $r : W_X \to X$ is any retraction whose restriction to $D^*X$ equals the projection $D^*X \to X$. This is a well defined cosimplicial local system by Lemma~\ref{lem:article:monotonicity2} and the fact that all intersection points are inside $D^*X$, all inclusion maps are quasi isomorphisms in the same way as in the proof of Lemma~\ref{lem:article:F_are_cosimplicial}.
\end{definition}
Note that by definition we have that
\begin{align*}
  \tot \tlar{\bu}{\FF(K_t^-,L_t^+)}{\bu}{\bu} = \tltr{\bu}{\FF(K_t^-,L_t^+)}{\bu},
\end{align*}
but more importantly we have (using Lemma~\ref{lem:article:monotonicity3}) a map
\begin{align*}
  (\mu_2')_\alpha : \tl{\bu} \FF(K_t^-,L_t^+)^\bu_\alpha \to \tlbr{\bu}{\FF'(K_t^-,X)}{\alpha} \otimes \bltr{\alpha}{\FF'(X,L_t^+)}{\bu}
\end{align*}
defined over each $\alpha$ as the restriction of the map $\mu_2^{\F}$ defined in Section~\ref{section:mu_variant}. This map commutes with inclusion of simplices and is hence a map of cosimplicial local systems
\begin{align}\label{eq:article:6}
  (\mu_2')_\bullet:\tl{\bu} \FF(K_t^-,L_t^+)^\bu_\bullet \to \tlbr{\bu}{\FF'(K_t^-,X)}{\bullet} \fwcs \bltr{\bullet}{\FF'(X,L_t^+)}{\bu}.
\end{align}

\begin{lemma}\label{lem:article:1}
  The map $\mu_2' = \tot (\mu_2')_\bullet$ represents $\mu_2$ in the sense that we define a commuting diagram:
  \begin{center}
    \begin{tikzcd}[row sep=10]
      &\tlbr{\bu}{\FF'(K_t^-,X)}{\bu}\und{S}\otimes \bltr{\bu}{\FF'(X,L_t^+)}{\bu} \\ 
      \tltr{\bu}{\FF(K_t^-,L_t^+)}{\bu} \ar{ur}{\mu_2'} \ar{dr}{\mu_2} \ar{r}{\tilde \mu_2} & \tltr{\bu}{\tilde \FF(K_t^-,L_t^+)}{\bu} \ar[->,swap]{u}{\simeq} \ar[->]{d}{\simeq}\\
      &\tltr{\bu}{\FF(K_t^-,X)}{\bu}\und{PX}\otimes \tltr{\bu}{\FF(X,L_t^+)}{\bu} 
    \end{tikzcd}
  \end{center}
  with the vertical arrows are weak equivalences in $\tltr{K_t^-}{\GG}{L_t^+}$.
\end{lemma}

\begin{proof}
  The proof is similar to the proof of Lemma~\ref{lem:article:2-old-prop}. We first define the intermediate system using
  \begin{align*}
    \tlar{\bu}{\tilde \FF(K_t^-,L_t^+)}{\bu}{\alpha} =  \qquad \smashoperator{\bigoplus_{y,z \in K_t^-\cap L_t^+\cap r^{-1}(X(\alpha))}} \Sigma^{|x|}\tltr{\bu}{P(K_t^-)}{y} \otimes \tltr{y}{PX(\alpha)}{z} \otimes \tltr{z}{P(L_t^+)}{\bu}. 
  \end{align*}
  This and its total complex, denoted $\tltr{\bu}{\tilde \FF(K_t^-,L_t^+)}{\bu}$, is contained in
  \begin{align*}
    \tltr{\bu}{\FF(K_t^-,X)}{\bu}\und{PX}\otimes \tltr{\bu}{\FF(X,L_t^+)}{\bu} \cong \quad \smashoperator{\bigoplus_{y,z \in K_t^-\cap L_t^+}}\Sigma^{|x|} \tltr{\bu}{P(K_t^-)}{y} \otimes \tltr{y}{PX}{z} \otimes \tltr{z}{P(L_t^+)}{\bu} 
  \end{align*}
  and Lemma~\ref{lem:article:monotonicity3} implies that $\mu_2$ over $\alpha$ lands in this subcomplex and defines a local map:
  \begin{align*}
    (\tilde \mu_2)_\alpha : \tl{\bu}{\FF(K_t^-,L_t^+)}^\bu_\alpha \to \tlar{\bu}{\tilde \FF(K_t^-,L_t^+)}{\bu}{\alpha}
  \end{align*}
  whose total map defines $\tilde \mu_2$ in the diagram. The map going up in the diagram is again augmentation, and an equivalence (over each $\alpha$). We finish the proof by proving that the inclusion going down in the diagram is a weak equivalence.

  We prove this by considering the diagrams
  \begin{center}
    \begin{tikzcd}
      & \tlbr{\bu}{\tilde \FF(K_t^-,X)}{\alpha} \otimes \bltr{\alpha}{\tilde \FF(X,L_t^+)}{\bu} \ar[->]{r}{\simeq} \ar[->]{d}{\simeq} & \tlbr{\bu}{\FF(K_t^-,X)}{\alpha}\otimes \bltr{\alpha}{\FF(X,L_t^+)}{\bu}  \ar[->]{d} \\
      & \tlar{\bu}{\tilde \FF(K_t^-,L_t^+)}{\bu}{\alpha} \ar[->]{r}  & \tltr{\bu}{\FF(K_t^-,X)}{\bu}\und{PX}\otimes \tltr{\bu}{\FF(X,L_t^+)}{\bu} 
    \end{tikzcd}
  \end{center}
  where the horizontal maps are the inclusions and the vertical maps are concatenations inserting an affine piece similar to the one used in Equation~(\ref{eq:article:8}). The top horizontal map is the tensor of the equivalences (one of which we) established in Lemma~\ref{lem:article:2-old-prop}. The left vertical map is an equivalence for each summand (indexed by $(y,z)$) as the path spaces $P(X(\alpha))$ are contractible.

  Taking the induced maps of total complexes we get the diagram
  \begin{center}
    \begin{tikzcd}
      & \tlbr{\bu}{\tilde \FF(K_t^-,X)}{\bu} \und{S} \otimes \bltr{\bu}{\tilde \FF(X,L_t^+)}{\bu} \ar[->]{r}{\simeq} \ar[->]{d}{\simeq} & \tlbr{\bu}{\FF(K_t^-,X)}{\bu}\und{S}\otimes \bltr{\bu}{\FF(X,L_t^+)}{\bu}  \ar[->]{d}{\simeq} \\
      & \tltr{\bu}{\tilde \FF(K_t^-,L_t^+)}{\bu} \ar[->]{r}  & \tltr{\bu}{\FF(K_t^-,X)}{\bu}\und{PX}\otimes \tltr{\bu}{\FF(X,L_t^+)}{\bu} 
    \end{tikzcd}
  \end{center}
  where now the right vertical map is also checked to be a weak equivalence, indeed, for each pair $(y,z)$ in the summand this is the case of Lemma~\ref{lem:article:zig-zag}.
\end{proof}

\section{The zero section generates}
In this section we prove the following proposition and corollary.

\begin{proposition} \label{prop:article:2}
  For exact Lagrangians $K,L \subset W_X$ the map
  \begin{align*}
    \mu_2 : \tltr{\bu}{\FF(K,L)}{\bu} \to \tltr{\bu}{\FF(K,X)}{\bu} \und{PX}\otimes \tltr{\bu}{\FF(X,L)}{\bu}
  \end{align*}  
  is a weak equivalence. In fact, $\mu_2'$ from Equation~(\ref{eq:article:6}) is an equivalence of cosimplicial local systems taking values in bi path local systems.
\end{proposition}

\begin{proof}[Proof of Proposition~\ref{prop:article:2}]
  Using continuation maps we may prove this after applying any Hamiltonian isotopies to $K$ and $L$. So, pick an insulated simplicial Morse function $(f,g,\{X(\alpha)\}_{\alpha \in S})$ satisfying the same conditions as in Section~\ref{sec:cosimplicial-floer-complexes}. Pick a zero-cell (point) $\alpha_0 \in S$. Using Lemma~\ref{lem:article:11} we may assume that we have a small chart $\varphi:D_\delta^n \to X(\alpha_0) \subset X$ such that
  \begin{itemize}
  \item $\varphi(0)=|\alpha_0|$.
  \item The image of $\varphi$ does not meet $A$ (the attaching domain).
  \item $f(\varphi(q))=\tfrac12 \norm{q}^2$ (Morse chart).
  \item Using the induced symplectic chart $\varphi^* : T^*D_\delta^n \to T^*X$ there are some distinct $a_i\in \R$ and $b_i\in \R$ so that
    \begin{align*}
      \varphi^*K = \bigsqcup_{i=1}^k D_\delta^n\times\{(a_i,0,\dots,0)\} \quad \textrm{and} \quad
      \varphi^*L = \bigsqcup_{i=1}^l D_\delta^n\times\{(b_i,0,\dots,0)\}.
    \end{align*}
  \item By shrinking the chart and modifying the Riemannian structure near $\alpha_0$ (which is far from the boundary of any insulator) we may assume that $g$ pulls back to the standard metric on $D_\delta^n$.
  \end{itemize}
  We will refer to each component of $\varphi^*K$ and $\varphi^*L$ as \emph{layers}. Now fix global primitives for the Liouville form on $K$ and $L$. We denote the primitives of the layers for $\varphi^*K$ by $f_i$ and the primitives of the layers for $\varphi^*L$ by $g_i$, they are on the form:
  \begin{align*}
    f_i(q) = -a_iq_1+c_i \qquad \textrm{and} \qquad g_i(q)=-b_iq_1+d_i
  \end{align*}
  for some $c_i,d_i\in \R$ (with $q=(q_1,\dots,q_n)\in D_\delta^n$).

  As $g$ is standard in the chart it means that the almost complex structure is standard in the chart $\varphi^*$ (before any potential small perturbations).

  We then use $K$ and $L$ as in Section~\ref{sec:cosimplicial-floer-complexes} to construct $K_t^-$ and $L_t^+$ for each $t$. In Lemma~\ref{lem:article:1} we saw that $\mu_2$ is a weak equivalence if $\mu_2'$ is a weak equivalence of cosimplicial local systems (in bi-path local systems and for large $t$). It is enough to check this on a single fiber. So we check this over $\alpha_0$ and consider the definition of $\mu_2'$ in Equation~(\ref{eq:article:6}) over this zero cell:
  \begin{align*}
    \tlar{\bu}{\mathcal F(K_t^-,L_t^+)}{\bu}{\alpha_0}\xrightarrow{(\mu_2')_{\alpha_0}} \tl{\bu}\FF'(K_t^-,X)_{\alpha_0} \otimes \bltr{\alpha_0}{\FF'(X,L_t^+)}{\bu}.
  \end{align*}
  Unpacking the definitions this is a map
  \begin{align*}
    \quad\smashoperator{ \bigoplus_{\substack{x\in K_t^- \cap L_t^+\\x\in  r^{-1}X(\alpha_0)}}} \Sigma^{|x|}\,\tltr{\bu}{(PK_t^-)}{x} \otimes \tltr{x}{(PL_t^+)}{\bu}
    \ \to\ 
    \smashoperator{\bigoplus_{y\in K_t^- \cap X(\alpha_0)}} \Sigma^{|y|}\,\tltr{\bu}{(PK_t^-)}{y} \ \otimes\ 
     \smashoperator{\bigoplus_{z\in \cap L_t^+ \cap X(\alpha_0)}} \Sigma^{|z|}\,\tltr{z}{(PL_t^+)}{\bu}.
  \end{align*}
  We may assume for large $t$ that all these intersections happen inside the chart $\varphi^*$.

  The Lagrangians $K_t^-$ and $L_t^+$ have induced primitives by scaling those on $K$ and $L$ and adding $\pm \epsilon df$ ($\epsilon$ from Lemma~\ref{lem:article:17}). They also have similar layers inside the chart. The primitives for the layers of $K_t^-$ are:
  \begin{align*}
    f_i^t(q)=-e^{-t}a_iq_1 + e^{-t}c_i + \tfrac12\epsilon \norm{q}^2
  \end{align*}
  while the primitives for the layers of $L_t^+$ are:
  \begin{align*}
    g_i^t(q)=-e^{-t}b_iq_1 +e^{-t}d_i - \tfrac12\epsilon \norm{q}^2.
  \end{align*}
  Examples of such Lagrangians in $T^*D_\delta^1$
  \begin{figure}[ht]
    \centering
    \begin{tikzpicture}
      \fill[yellow] (0.7,0) -- (0,0) -- (0.35,-0.35);
      \draw[thick] (-3,0) -- (3,0);
      \draw[thick, red] (-2,0) -- (2,0);
      \draw[red] (-2,0) node {$($};
      \draw[red] (2,0) node {$)$};
      \draw[brown] (-2,2) -- (2,-2);
      \draw[brown] (-2.4,2) -- (1.6,-2);
      \draw[brown] (-1.7,2) -- (2.3,-2);
      \draw[green!70!black] (2.5,2) -- (-1.5,-2);
      \draw[green!70!black] (2.7,2) -- (-1.3,-2);
      \draw[green!70!black] (1.9,2) -- (-2.1,-2);
      \fill[gray!!20] (-2.0,1.7) -- (-2.0,1.9) -- (-0.5,0.4) -- (-0.5,0.2) -- cycle;
    \end{tikzpicture}
    \caption{Intersection of {\color{brown} $K_t^-$}, {\color{green!70!black} $L_t^+$} and {\color{red} $D_\delta^n$}} \label{fig:triangle} 
  \end{figure}
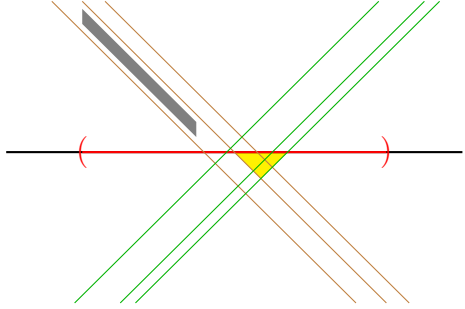
  are illustrated in Figure~\ref{fig:triangle}. It follows that we have a $C_1>0$ such that
  \begin{align} \label{eq:article:12}
    \norm{x} \leq C_1e^{-t}
  \end{align}
  for any intersection point $x$ of a pair of $K_t^-,X,L_t^+$ in this chart. In turn it follows that there is a bijection between $x\in K_t^- \cap L_t^+ \cap T^*X(\alpha)$ and pairs $(y,z)$ with $y \in K_t^- \cap X(\alpha)$ and $z\in L_t^+ \cap X(\alpha)$. In fact for each such $x$ there is a unique pseudo holomorphic triangle (one of which is illustrated in yellow in Figure~\ref{fig:triangle}) with corners in $x,y$ and $z$ and sides on $K_t^-,X,L_t^+$.

  Each such triangle has area $O(e^{-2t})$, which implies that we have a constant $C_2>0$ such that
  \begin{align*}
    \mathcal A(x)-\mathcal A(y)-\mathcal A(z) \leq C_2e^{-2t}
  \end{align*}
  for $(y,z)$ the pair associated to $x$. We also see that the gray area in Figure~\ref{fig:triangle} between two strands is bounded from below by some $C_3e^{-t}$ with $C_3>0$, and similar for the other gaps between layers. This means that any pseudo holomorphic strip from some $x$ to another pair $(y,z)$ which it is not the pair associated to $x$ must have an area larger than $C_4e^{-t}$ for some $C_4>0$. 

  It follows that if we order the generators $x$ by action and give the pairs $(y,z)$ the induced order (which may not be the same as their action order) then for large $t$ the differential is upper triangular with $\pm 1$ on the diagonals (the 1 being the count of the triangle). It follows that it is a quasi isomorphism.
\end{proof}

\section{Local systems of homology} \label{sec:local-syst-homol}

In this section we consider the generalized local systems from the two previous sections when they define standard Floer homology (i.e we take coefficients in $\F$). In particular, we prove some lemmas for the classical (possibly graded) local systems defined by their homologies. We also use this to prove the following lemma, which could have been proved directly after the definitions in Section~\ref{sec:floer-theory-with}, but the signs are easier to handle using the work we have done until this point.

\begin{lemma} \label{lem:cor:article:7}
  For any exact Lagrangian $L \subset W$ there is a weak equivalence
  \begin{align*}
    \tltr{\bu}{\FF(L,L)}{\bu} \simeq \tltr{\bu}{PL}{\bu}.
  \end{align*}
\end{lemma}

We assume the same setup as in the previous two sections (i.e. we have an insulated simplicial Morse function). For any exact $K,L \subset W_X$ we consider the cosimplicial local system on $X$ which over $\alpha$ are, using Definition~\ref{def:article:1}, given by
\begin{align*}
  \tr{\F}{\bu} \und{PK_t^-} \otimes \tlar{\bu}{\FF(K_t^-, L_t^+)}{\bu}{\alpha} \und{PL_t^+} \otimes \tl{\bu}{\F} \subset CF_*((K_t^-)^{\str{\F}{\bu}},(L_t^+)^{\stl{\bu}{\F}}).
\end{align*}
Here the last inclusion uses Equation~(\ref{eq:article:14}), which also shows that its total complex gives usual Floer homology of $K$ with $L$. We define
\begin{align} \label{eq:article:9}
  A(K,L) = H_*(\tr{\F}{\bu} \und{PK_t^-} \otimes \tlar{\bu}{\FF(K_t^-, L_t^+)}{\bu}{\bu} \und{PL_t^+} \otimes \tl{\bu}{\F})
\end{align}
for large $t$ as the classical local system of (possibly graded) homology on $X$ associated to this cosimplicial local system.

Note that by definition $A(K,X)$ and $A(X,L)$ are the local systems of homology of
\begin{align*}
  \tr{\F}{\bu} \und{PK_t^-} \otimes \tl{\bu} \FF'(K_t^-,X)_\bu \qquad\textrm{and}\qquad  \bltr{\bu}{\FF'(X,L_t^+)}{\bu} \und{PL_t^+} \otimes \tl{\bu}{\F}
\end{align*}
respectively. In particular the last sentence in Proposition~\ref{prop:article:2} and Künneth theorem implies that $\mu_2'$ induce an isomorphism of local systems:
\begin{align} \label{eq:article:13}
  A(K,L) \cong A(K,X) \otimes_f A(X,L),
\end{align}
where $\otimes_f$ denotes the fiber-wise tensor of classical (possibly graded) local systems. For a (possibly graded) vector space $V$ we denote its dual by $V^\dagger$, when $V$ is graded we use the convention that $V_k^\dagger$ is the dual of $V_{-k}$. For a classical local system $A$ of (possible graded) vector spaces we let $A^\dagger$ denote the fiber-wise dual local system.

\begin{lemma} \label{lem:article:10}
  Let $\alpha_0 \in S$ be a zero simplex. There is a Poincare duality isomorphism $A(K,L)_{\alpha_0} \cong A(L,K)_{\alpha_0}^\dagger$. In particular $A(L,L)_{\alpha_0}$ is self dual. 
\end{lemma}

\begin{remark} \label{rem:article:3}
  Note that we will later in Section~\ref{sec:proof-main-theorem} use this to show that $A(L,L)$ is in fact trivial. Combining this with Equation~(\ref{eq:article:13}) above we see that $A(L,X)$ is in fact the dual of $A(X,L)$ as local systems (not just over a single fiber).
\end{remark}

\begin{proof}[Proof of Lemma~\ref{lem:article:10}]
  Because of Equation~(\ref{eq:article:13}) it is enough to prove that $A(X,L)_{\alpha_0}$ is the dual of $A(L,X)_{\alpha_0}$ for any exact $L\subset W_X$. 
  
  Let $U = \partial X(\alpha_0) \times [-1,1]$ be a tubular neighborhood around $\partial X(\alpha_0)$ in which $g$ is on product form and $\nabla f$ is outwards pointing. Let $h : U\cup X(\alpha_0) \to \R$ be such that $h=0$ on the part inside $X(\alpha_0)$ and such that $h\to -\infty$ at the boundary of $U$ outside of $X(\alpha_0)$ and $\nabla h$ is inwards pointing. The Lagrangian given by $dh$ is a small perturbation of the fiber $F=T_{|\alpha_0|}^*X$. Lemma~\ref{lem:article:monotonicity1} shows that
  \begin{align*}
    \br{A(X,L)}{\alpha_0} \cong HF_*(dh, L_t^+)\cong HF_*(F, L).
  \end{align*}
  Indeed, as the gradients of $f$ and $h$ point in different directions in $U$ and $h=0$ on the part inside $X(\alpha_0)$ the generators (intersection points) $X\cap L_t^+$ and $dh\cap L_t^+$ are the same. A similar argument shows that $\br{A(L,X)}{\alpha_0}\cong HF_*(L_t^-, -dh)\cong HF_*(L, F)$.

  In both cases we see that we get the Floer homology of $L$ with the fiber. We, however, get them in the opposite order so it follows that they are dual with a possible grading shift. However the grading on $F$ in $HF_*(F,L)$ is $-n/2$ and the grading on $F$ in $HF_*(L,F)$ is $n/2$, that is since the grading of $F$ is induced from the isotopy from $dh$ respectively $-dh$ to $F$ in the continuation maps above. This implies that the grading of a generator in $HF_*(F,L)$ is minus the grading of the same generator in $HF_*(L,F)$.
\end{proof}

\begin{proof}[Proof of Lemma~\ref{lem:cor:article:7}]
  We will prove the equivalent fact that $\bltr{\bu}{PL}{\bu}$ is weakly equivalent to $\bltr{\bu}{\FF(L,L)}{\bu}$. We only need to consider a small Weinstein neighborhood $W_L$ of $L$ so we can use the above and the previous sections with $X=L$ and where $W_X=W_L=D^*L$ has no handles attached. We thus pick an insulated simplicial Morse function $(f,g,\{L(\alpha)\}_{\alpha \in S})$ with $|S|\cong L$. We have $L_t^\pm = \pm \epsilon df$ and identify canonically $PL_t^\pm \cong PL$ in the following. Recall from Lemma~\ref{lem:article:2-old-prop} that $\bltr{\bu}{\FF(L,L)}{\bu}$ is equivalent to
  \begin{align*}
    \alpha \mapsto \bltr{\alpha}{\tilde \FF(L,L_t^+)}{\bu} = \smashoperator{\bigoplus_{x \in L(\alpha)\cap \epsilon df}} \Sigma^{|x|} \bltr{\alpha}{PL(\alpha)}{x} \otimes \tltr{x}{PL}{\bu}.
  \end{align*}
  We then use $\tlbr{x}{PL(\alpha)}{\alpha} \otimes \bltr{\alpha}{PL}{\bu} \xrightarrow{\simeq} \tltr{x}{PL}{\bu}$ when $x\in L(\alpha)$ (defined using concatenation as in Equation~(\ref{eq:article:8})) to expand the last factor in each summand to see that:
  \begin{align*}
    \bltr{\alpha}{\FF(L,L)}{\bu} \simeq \smashoperator{\bigoplus_{x \in L(\alpha)\cap \epsilon df}} \Sigma^{|x|}  \bltr{\alpha}{PL(\alpha)}{x} \otimes \tlbr{x}{PL(\alpha)}{\alpha} \otimes \bltr{\alpha}{PL}{\bu}.
  \end{align*}
  This is now written as a fiber wise tensor product of two cosimplicial local systems, one of which we may augment
  \begin{align*}
    \smashoperator{\bigoplus_{x \in L(\alpha)\cap \epsilon df}} \Sigma^{|x|} \bltr{\alpha}{PL(\alpha)}{x} \otimes \tlbr{x}{PL(\alpha)}{\alpha} \to
    \smashoperator{\bigoplus_{x \in L(\alpha)\cap \epsilon df}} \Sigma^{|x|} \F \subset CF_*(L^{\str{\F}{\bu}},df^{\stl{\bu}{\F}})
  \end{align*}
  which proves that this is in fact the cosimplicial local system whose homology is $A(L,L)$ (now defined on $L$). In this special case the rank of $A(L,L)$ is 1. The other tensor factor is $\bltr{\bu}{PL}{\bu}$.

  Let $\tl{\bu}{D}$ denote a path local system with homology of rank 1 equivalent to $A(L,L)$. It follows that $ \tltr{\bu}{\FF(L,L)}{\bu} \simeq \tl{\bu}{D} \fwlp \tltr{\bu}{PL}{\bu}$, where the fiber homology of the left path local system $\tl{\bu}{D}$ is rank 1. Note that, by Lemma~\ref{lem:skyscraper} we can equivalently write this as $\tltr{\bu}{PL}{\bu}$ tensor with the opposite rank 1 right path local system.

  It remains to show that $\tl{\bu}{D}$ is equivalent to the rank 1 trivial local system. Using Proposition \ref{prop:article:2} in our current special case we see that we must have
  \begin{align*}
    \tltr{\bu}{\FF(L,L)}{\bu} \simeq \tltr{\bu}{\FF(L,L)}{\bu} \und{PL} \otimes \tltr{\bu}{\FF(L,L)}{\bu}
  \end{align*}
  which is satisfied by $\tltr{\bu}{PL}{\bu}$, but not if twisted by a non-trivial rank 1 local system. Hence $\tl{\bu}{D}$ is trivial.
\end{proof}

The triviality of $\tl{\bu}{D}$ really proves that self Floer homology of $L$ is the usual \emph{homology} of $L$ (Floer's lemma).

\begin{remark}
  For general $L\subset W_X$ the above argument does not show that the homology local system $A(L,L)$ on $X$ is trivial. This we prove later.
\end{remark}

\begin{lemma} \label{lem:cor:article:6}
  For any exact Lagrangian $L \subset T^*X$ there is a spectral sequence converging to $H_*(L)$ with page two equal to
  \begin{align*}
    H_*(X; A(L,L)).
  \end{align*}
\end{lemma}

\begin{proof}
  From Lemma~\ref{lem:lemmas:spectral_sequence_for_total_space}, Lemma~\ref{lem:article:14} and Floers lemma the spectral sequence converges to
  \begin{align*}
    H_*(\tot(\tr{\F}{\bu} \und{PL_t^-} \otimes \tlar{\bu}{\FF(L_t^-, L_t^+)}{\bu}{\bu} \und{PL_t^+} \otimes \tl{\bu}{\F}))\cong HF_*(L,L)\cong H_*(L).
  \end{align*}
\end{proof}

\section{Proof of main theorem} \label{sec:proof-main-theorem}

For this section we assume that $L\subset W_X$ is closed, connected and exact. For any Lagrangian $K \subset W_X$ we let $i_K : K \to W_X$ denote the inclusion. Also recall that we have a retraction $r:W_X \to X$. We denote the composition
\begin{align*}
  p = r\circ i_L : L \to X.
\end{align*}

\begin{proposition} \label{prop:article:3}
  Assume that the Floer complexes $\tltr{\bu}{\FF(L,X)}{\bu}$ and $\tltr{\bu}{\FF(X,L)}{\bu}$ are defined. Then $p^* :\tl{X}{\GG} \to \tl{L}{\GG}$ is $\pi_0$-surjective.
\end{proposition}

\begin{proof}
   Consider the functor $\tr{\GG}{L} \to \tr{\GG}{L}$ given by
  \begin{align*}
    \tl{\bu}{D} \mapsto \tltr{\bu}{\FF(L,X)}{\bu}\und{PX}\otimes \tltr{\bu}{\FF(X,L)}{\bu}\und{PL}\otimes \tl{\bu}{D},
  \end{align*}
  this is a $\pi_0$-equivalence since it is the identity on weak equivalence classes. Indeed
  \begin{align*}
    \tltr{\bu}{\FF(L,X)}{\bu}\und{PX}\otimes \tltr{\bu}{\FF(X,L)}{\bu}\und{PL}\otimes \tl{\bu}{D}
    &\simeq \tltr{\bu}{\FF(L,X)}{\bu}\und{PX}\otimes \tltr{\bu}{\FF(X,L)}{\bu}\und{PL}\otimes \tl \bu{\FF D}\\
    &\simeq\tltr{\bu}{\FF(L,L)}{\bu}\und{PL}\otimes \tl \bu{\FF D}\simeq \tl \bu{\FF D}\simeq \tl \bu D
  \end{align*}
  where we in the first equivalence use that the functor above is homotopic by Lemma~\ref{lemma:article:floer-bimodule-is-semifree} and in the second equivalence use Proposition~\ref{prop:article:2} and Lemma~\ref{lem:tensor_lemma} and in the third equivalence use Lemma~\ref{lem:cor:article:7} and that global tensor with $\tltr{\bu}{PL}{\bu}$ is the identity functor.

 It follows that the functor
  \begin{align*}
    F=\tltr{\bu}{\FF(L,X)}{\bu}\und{PX}\otimes -
  \end{align*}
  is $\pi_0$-surjective. The retraction condition implies that $i_X^* \circ r^*$ is the identity functor on $\tl X\GG$. So, by Corollary~\ref{cor:projection_formula} we can write
  \begin{align*}
    F(\tl{\bu}{E}) \simeq & F(\tl\bu \F \fwlp i_X^*(r^*(\tl{\bu}{E}))) \simeq F(\tl{\bu}{\F}) \fwlp  i_L^*(r^*(\tl{\bu}{E})) 
  \end{align*}
  with $\tl{\bu}{E}$ in $\tl{X}{\GG}$. It follows that the functor given by the pullback $p^* = i_L^* \circ r^*$ composed with fiber wise tensor with $F(\tl{\bu}{\F})$ is also $\pi_0$-surjective.

  In turn it follows that this last tensoring with $F(\tl{\bu}{\F})$ must also be $\pi_0$-surjective, which means that there is a $\tl{\bu}{D} \in \tl{L}{\GG}$ such that $F(\tl{\bu}{\F}) \fwlp \tl{\bu}{D} \simeq \tl{\bu}{\F}$. In particular tensoring with $F(\tl{\bu}{\F})$ is in fact a $\pi_0$-equivalence and the claim follows.
\end{proof}

\begin{corollary}
  The map $p_*:\pi_1(L)\to  \pi_1(X)$ is injective.
\end{corollary}

\begin{proof}
  Using the above proposition with $\F=\F_2$ in the ungraded cases (where all Floer complexes are defined) we get that pullback using $p=r \circ i_L$ is $\pi_0$-surjective.

  Assume we have a non-trivial kernel $\ker(p_*)\subset \pi_1(L,x)\xrightarrow{p_*}\pi_1(X,p(x))$. Any element of this kernel would act trivially on the homology of the fiber over $x$ of any pullback $p^*(\tl{\bu}{D})$. However the path local system corresponding to the module $\F_2[\pi_1(L)]$ shows that not all local systems have this property, which yields a contradiction with the $\pi_0$-surjectivity from the proposition above.
\end{proof}

\begin{theorem}
  The Maslov class of $L$ vanishes.
\end{theorem}

\begin{proof}
  Consider the covering space $\tilde W_X \to W_X$ associated to the kernel of the retraction $\pi_1(W_X)\xrightarrow{r_*}\pi_1(X)$. This is given by considering $D^*\tilde X$ where $\tilde X$ is the universal cover of $X$ with identical (translated) handle attachments along each connected component of the pullback $\tilde A \subset D^*\tilde X$ of $A \subset D^*X$. The retraction $r : W_X \to X$ lifts to a retraction $\tilde r : \tilde W_X \to \tilde X$ and $L\subset W_X$ gives by pullback $\tilde L\subset\tilde W_X$. By the corollary above $\tilde L$ has simply connected components. Indeed, if there was a non-trivial loop in $\tilde L$ this would be a lift of a non-trivial loop in $L$, which would map by $p_*$ to a non-trivial loop in $X$. However, that means that $\tilde r$ maps it to a path with different endpoints in $\tilde X$ - a contradiction.
  
  Let $(f,g,\{X(\alpha)\}_{\alpha \in S})$ be an insulated simplicial Morse function satisfying everything assumed in Section~\ref{sec:cosimplicial-floer-complexes}. Let $\tilde S$ denote the deck transformation invariant semi simplicial structure on $\tilde X$ lifting $S$. The insulated Morse function pulls back to such $(\tilde f,\tilde g,\{\tilde X(\tilde \alpha)\}_{\tilde \alpha \in \tilde S})$ which is an obvious generalization of an insulated simplicial Morse function on $\tilde X$ (which is a possibly open manifold).

  We use this to define a cosimplicial local system on the cover $\tilde X$. Indeed, we define $\tlar{\bu}{\FF(\tilde L_t^-,\tilde L_t^+)}{\bu}{\bu}$ for large $t$ similar to Definition~\ref{def:article:1}. That is, for $\tilde \alpha \in \tilde S$ we define
  \begin{align*}
    \tlar{\bu}{\FF(\tilde L_t^-,\tilde L_t^+)}{\bu}{\alpha}\ =\ \smashoperator{\bigoplus_{\tilde x \in \tilde L_t^-\cap \tilde L_t^+  \cap \tilde r^{-1}\tilde X(\tilde \alpha)}}  \Sigma^{|x|}\,\tl{\bu} P(\tilde L_t^-)^x \otimes \tltr{x}{P(\tilde L_t^+)}{\bu}
  \end{align*}
  which is again well-defined as a cosimplicial local system on $\tilde S$ by Lemma~\ref{lem:article:monotonicity2}. Note in particular that these Floer complexes are well-defined as any pseudo holomorphic curve in $\tilde W_X$ is a lift of such in $W_X$.

  Since $\tilde L$ is component wise simply connected these covering Floer complexes can be defined in a graded category. We may tensor this with $\tr{\F}{\bu}$ and $\tl{\bu}{\F}$ from the sides to get cosimplicial local versions of the standard Floer complex (on the cover). We define $A(\tilde L,\tilde L)$ as the classical graded local system on $\tilde X$ defined by the local homology of this cosimplicial local system representing the standard Floer complex. This is similar to the definition of $A(L,L)$ as a local system on $X$ from Equation~(\ref{eq:article:9}). In fact, we claim that $A(\tilde L,\tilde L)$ is the pullback of $A(L,L)$ to the covering space (where we forget the grading of $A(\tilde L, \tilde L)$ if $A(L,L)$ is not graded). To see this we again appeal to Lemma~\ref{lem:article:monotonicity2}. Indeed, this shows (used up on $\tilde X$) that there are no pseudo holomorphic strips that go from one lift $\tilde X(\tilde \alpha)$ to another lift $\tilde X(\tilde \alpha')$ of an insulator $X(\alpha)$ in $X$, which implies that all the strips in the differential over $\alpha$ lifts bijectively to strips in the differential over $\tilde \alpha$. The claim follows, as the generators are in bijection as well.

  By Lemma~\ref{lem:cor:article:6} the local system $A(L,L)$ is non-zero, hence $A(\tilde L,\tilde L)$ is non-zero. Assuming then that there is a loop in $L$ with non-trivial Maslov class we lift such a loop based in a cotangent fiber over a zero simplex $|\alpha_0|$ to a path in $\tilde L$ which by the above starts and ends in two different cotangent fibers (over two different lifts $|\tilde \alpha_0|$ and $|\tilde \alpha_o'|$). As the Maslov index of this path is non-zero it follows that the two graded vector spaces $A(\tilde L,\tilde L)_{\tilde \alpha_0}$ and $A(\tilde L,\tilde L)_{\tilde \alpha_0'}$ are isomorphic after the shift by the Maslov index of the loop. However, as they are of course unshifted isomorphic, finite dimensional and non-zero that is a contradiction.
\end{proof}

\begin{theorem}
  The composition $p : L \subset W_X \xrightarrow{r} X$ is a homology equivalence.
\end{theorem}

\begin{proof}
  By the above, everything is now assumed to be graded, but initially defined over $\F_2$. The following argument is first done using the ground field $\F=\F_2$. However, it then shows that $H_*(L) \to H_*(X)$ is an $\F_2$ homology equivalence, which implies that they preserve the Wu classes, and hence preserve the Stiefel-Whitney class (as they are determined by the Wu classes by Wu's formulas). This implies that we can then give $L$ a relative pin structure and run the argument again for any $\F$.

  The vanishing Maslov class implies that $L \to X$ is relatively oriented. Indeed, the mod 2 reduction of the Maslov class is the relative first Stiefel-Whitney class. Fixing an orientation on $\tilde X$ and such a relative orientation we may define the degree of $L\to X$, which we denote $d\in \Z$. Indeed, the relative orientation (pulled back from $L \to X$) provides $\tilde L$ with a preferred orientation so this can be defined on the cover (in fact on the lift to any oriented cover of $X$).

  The argument in the proof above on the covering $\tilde W_X$ provides a cosimplicial local system on $\tilde X$ whose total complex is the standard Floer complex of $\tilde L$ with itself. It follows from Lemma~\ref{lem:article:14} (as in Corollary~\ref{cor:article:6}) that there is a spectral sequence converging to $H_*(\tilde L)$ with page 2 given by $H_*(\tilde X; A(\tilde L,\tilde L))$. Since $\tilde X$ is simply connected it follows that $A(\tilde L,\tilde L)$ is non-negatively graded (as $H_*(\tilde L)$ cannot have homology in negative degree). As we saw in the proof above it is the pullback of $A(L,L)$ which by Lemma~\ref{lem:article:10} has self-dual fibers, so it follows that both $A(L,L)$ and $A(\tilde L,\tilde L)$ are supported in degree $0$. As their mutual fibers are finite dimensional these fibers are isomorphic to $\F^k$ for some $k\geq 1$.

  The same spectral sequence on the cover then shows that the $H_0(\tilde L) \cong \F^k$. As each pair of connected components of $\tilde L$ are identified under some deck transformation of $\tilde W_X \to W_X$, they have the same degree and $k$ therefore divides $d$. On the other hand, considering how the grading in the generators of the Floer complex are defined we see that the fiber of $A(L,X)$ (and its dual, the fiber of $A(X,L)$) have Euler characteristic equal to $\pm d$, which means by Equation~(\ref{eq:article:13}) used in a single fiber that $k=d^2$. We conclude that $d^2$ divides $d$ and thus $d=\pm1$ and $k=1$ (as $k$ was not $0$). It follows directly that $\tilde L$ is connected and that $\pi_1(L) \to \pi_1(X)$ is surjective hence an isomorphism. We may change the chosen relative orientation to make $d=1$.

  As the degree of $L \to X$ is 1 it follows from the projection formula for cohomology ($p_!(p^*a \cup b) = a \cup p_!b$) that $H^*(X) \to H^*(L)$ is injective and thus $H_*(L) \to H_*(N)$ is surjective.
  
  In the case $\F=\F_2$ the fact that $A(L,L)$ has rank 1 makes it the trivial rank 1 local system. In the general case the spectral sequence from Lemma~\ref{lem:cor:article:6} shows that the local system $A(L,L)$ is trivial (if not $H_0(L)$ would vanish). In any case the spectral sequence collapses on page 2 as it is supported on the horizontal axis. However, the collapsing of the spectral sequence implies that the total dimension of $H_*(L)$ is the same as $H_*(X)$ hence the surjective map between this is in fact an isomorphism.
\end{proof}

\appendix

\section{Monotonicity lemmas}\label{sec:monotonicity}
Let $K_t^-$ and $L_t^+$ be the large deformations of $K$ and $L$ considered in Section~\ref{sec:cosimplicial-floer-complexes}. We assume that $t$ is sufficiently large.
\begin{lemma}\label{lem:article:monotonicity1}
  All pseudo holomorphic discs used to define the differential in $CF_*(X, L_t^+)$ having the input inside an insulator $X(\alpha)$ have its output inside $X(\alpha)$, furthermore its boundary on $X$ is completely contained in $X(\alpha)$.
\end{lemma}
\begin{proof}
  Make $t$ so large that each $x\in L_t^+$ which lies in the cotangent bundle over the boundary of each isolator $\partial X(\alpha)$ corresponds under $g^*:T^*X \cong TX$ to outwards pointing vectors. This is possible as $\nabla f$ is outward pointing.

  In a tubular neighborhood $T^*(\partial X(\alpha)\times [-1,1])$ in which $g$ is on product form we may project to $T^*[-1,1] \subset \R^2$ preserving the almost complex structure and such that $X(\alpha)$ is mapped to $[-1,0]$. The condition on $L_t^+$ thus means that it projects to be strictly above the first axis (red area in Figure~\ref{Fig:appendix:ProjL} for some $\epsilon>0$).
  \begin{figure}[ht]
    \begin{center} 
      \begin{tikzpicture}[scale=1.3]
        \fill[red!30!white] (2,0.5) -- (2,2.0)  -- (-2,2.0) -- (-2,0.5) -- cycle;
        \draw[red] (0,1.25) node {$L_t^+$};
        \fill (0,0) circle (2pt) node[below] {$\partial X(\alpha)$};
        \draw[thick, dotted] (2,0) -- (0,0);
        \draw[very thick] (-2,0) -- node[below] {$X(\alpha)\quad$} (0,0);
        \draw[dotted] (-2,2) -- (-2,-0.3);
        \draw[dotted] (2,2) -- (2,-0.3);
        \draw (-2,0) node[below] {$-1$};
        \draw (2,0) node[below] {$1$};
        \draw[->] (-2.4,0) node[left] {Zero section} -- (-2.1,0);
        \draw (-2,0.5) node[left] {$\epsilon$};
      \end{tikzpicture}
    \end{center}
  \caption{Projection of $L_t^+$} \label{Fig:appendix:ProjL}
\end{figure}
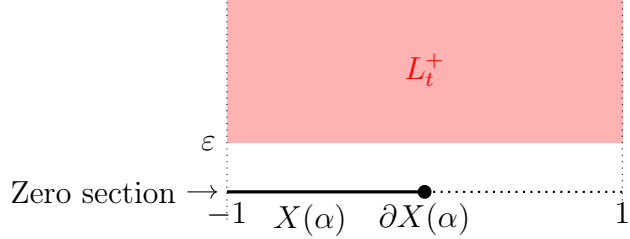
  A pseudo holomorphic curve must project to a holomorphic curve, this implies that the side of the disc which maps to $X$ (which thus projects to the zero section) cannot have a positive derivative as that would put part of the curve below the zero section, which is not possible (see Figure~\ref{Fig:Stripdefn} with $K=X$). It follows that pseudo holomorphic curves can only have their side on $X$ contained in $X(\alpha)$.
\end{proof}
\begin{remark}
  The above lemma also holds if we replace $CF_*(X,L_t^+)$ with $CF_*(K_t^-, X)$, the argument is similar and using that $K_t^-$ projects below the zero section in Figure~\ref{Fig:appendix:ProjL}.
\end{remark}
\begin{lemma}\label{lem:article:monotonicity2}
  All pseudo holomorphic discs used to define the differential in $CF_*(K_t^-, L_t^+)$ having the input inside an insulator $X(\alpha)$ has its output inside $X(\alpha)$.
\end{lemma}
\begin{proof}
  Consider the same projection to $T^*[-1,1]$ as in the proof of Lemma~\ref{lem:article:monotonicity1} above. For large $t$ the two Lagrangians project to different parts of $T^*[-1,1]$.
  \begin{figure}[ht]
    \begin{center} 
      \begin{tikzpicture}[scale=1.3]
        \fill[red!30!white] (2,0.5) -- (2,2.0)  -- (-2,2.0) -- (-2,0.5) -- cycle;
        \draw[red] (0,1.25) node {$L_t^+$};
        \fill[brown!30!white] (2,-0.5) -- (2,-2.0)  -- (-2,-2.0) -- (-2,-0.5) -- cycle;
        \draw[brown] (0,-1.25) node {$K_t^-$};
        \fill (0,0) circle (2pt);
        \draw[thick, dotted] (2,0) -- (0,0);
        \draw[very thick] (-2,0) -- node[below] {$X(\alpha)$} (0,0);
        \draw[dotted] (-2,2) -- (-2,-2);
        \draw[dotted] (2,2) -- (2,-2);
        \draw (-2,0) node[below] {$-1$};
        \draw (2,0) node[below] {$1$};
        \draw (-2,0.5) node[left] {$\epsilon$};
        \draw (-2,-0.5) node[left] {$-\epsilon$};
        \fill[green!70!black] (0,0.3) circle (1pt) node[right] {$A$};
        \fill[red] (0,-0.3) circle (1pt) node[right] {$B$};
      \end{tikzpicture}
    \end{center}
    \caption{Projections of $K_t^-$ and $L_t^+$ }\label{Fig:appendix:projLK}
  \end{figure}
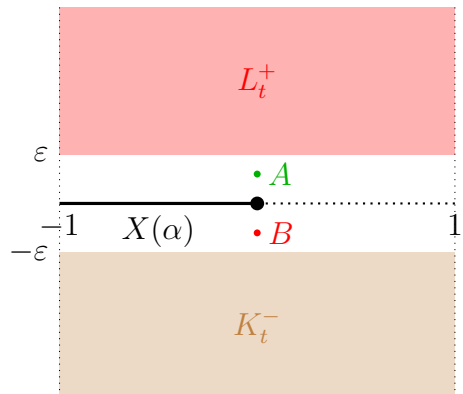
  Indeed, for large $t$ we can assume that $K_t^-$ and $L_t^+$ projects to $[-1,1] \times (-\infty,-\epsilon)$ and $[-1,1] \times (\epsilon,\infty)$ respectively (See Figure~\ref{Fig:appendix:projLK}). This in turn rules out any pseudo holomorphic strip which exits an $X(\alpha)$. Indeed, such a strip would go \emph{over} $(0,0)$ along $L_t^+$ algebraically a total of 1 time (and thus back \emph{under} $(0,0)$ along $K_t^-$ totally $1$ time), which would imply that the part of the strip landing in $T^*[-1,1]$ has degree $-1$ over $(0,0)$, which is impossible for a holomorphic map.
\end{proof}
\begin{lemma}\label{lem:article:monotonicity3}
  All pseudo holomorphic discs used to define the co-product $\mu_2:CF_*(K_t^-, L_t^+)\to CF_*(K_t^-,X)\otimes CF_*(X, L_t^+)$ from Section~\ref{sec:floer-theory-with} having the input inside an insulator $X(\alpha)$ have both its outputs inside $X(\alpha)$, furthermore its boundary on $X$ is completely contained in $X(\alpha)$.
\end{lemma}
\begin{proof}
  Again we use the projection to $T^*[-1,1]$ from the proof of Lemma~\ref{lem:article:monotonicity1}. A disc used to define $\mu_2$ with an input inside $X(\alpha)$ and any output outside $X(\alpha)$ have three parts of the boundary and degree argument similar to the proof of Lemma~\ref{lem:article:monotonicity2} using the points $A$ and $B$ in Figure~\ref{Fig:appendix:projLK} shows that the degree at either $A$ or $B$ is $-1$ which is impossible. If both outputs are inside $X(\alpha)$ it follows that the boundary on $X$ stays within $X(\alpha)$.
\end{proof}

\section{Boundary paths have finite length}

In this appendix we prove a few lemmas about the moduli spaces of discs used when defining Floer complexes. We only consider mapping spaces and moduli spaces which occur in the paper.

\begin{lemma} \label{lem:article:13}
  Let $\bM$ be one of the solution spaces of pseudo holomorphic discs $u : D^2 \to W$ with marked points and Lagrangian boundary conditions used in the paper. Then the map $\bM \to C^\infty(D^2,W)$ is well defined and continuous.
\end{lemma}

\begin{proof}
  This is standard elliptical boot-strapping.
\end{proof}

For any smooth map $u:D^2 \to W$ and two points $x,y\in S^1$ we may consider the path from $u(x)$ to $u(y)$ defined by $u$ restricted to the arc on the boundary that goes positively from $x$ to $y$. This we may reparametrize to an arc length parametrized path and this is continuous in the $C^\infty$ topology. This thus defines a map $\bM \to \tltr{u(x)}{\PP W}{u(y)}$.

For any planer tree configurations of discs there is an associated moduli space of pseudo holomorphic maps $\bM$. If we pick two of these marked points and denote them $x$ and $y$ there is again a unique way of tracing concatenated boundary arcs (that jump between the discs at gluing pairs) positively around the tree. Given any tree compatible pseudo holomorphic map $u=(u_1,\dots,u_k)$, we may arc length parametrizing the images of these arcs and concatenating them to define a path from $u(x)$ to $u(y)$. As we parametrize by arc length this is well defined on the equivalence class $[u] \in \bM$ and thus defines a continuous map $\bM \to \tltr{u(x)}{\PP W}{u(y)}$.

\begin{lemma} \label{lem:article:12}
  Let $\MM$ be a compactified moduli space of pseudo holomorphic discs with Lagrangian boundary conditions as in the paper. Let $x$ and $y$ be two of the marked points on the disc tree configurations defining this moduli space. Then the map $\MM \to \tltr{u(x)}{\PP W}{u(y)}$ defined above on each strata is continuous.
\end{lemma}

\begin{proof}
  The standard proof of Gromov compactness has sequences $(u_1,\dots,u_k)$ of tree configuration maps converging to a possible larger collection of limit tree configuration maps $(u_{d_1^1},\dots,u_{d_1^{j_1}},\dots,u_{d_k^1},\dots,u_{d_k^{j_k}})$, such that each $u_i$ Gromov converges to the sub-tree $(u_{d_i^1},\dots,u_{d_i^{j_i}})$. The way these converge easily implies that each arc between two marked points converges into the associated broken paths on the limit tree.
\end{proof}

\section{Isotopies}

\begin{lemma}\label{lem:article:11}
  For any $L \subset T^*X$ and a chart $\varphi :D_\delta^n \to X$ there is a Hamiltonian isotopy of $L$ to an $L'$ and a small $\delta'\in (0,\delta)$ so that $L'$ in the chart $\varphi^*: T^*D_{\delta'}^n \to T^*X$ is on the form $\bigsqcup_i  D_{\delta'}^n\times \{(a_i,0,\dots,0)\} $.
\end{lemma}

\begin{proof}
  Let $q=\varphi(0)$. We may first isotope $L$ to make it transverse to the fiber over $q$. We may then pick a compactly supported Hamiltonian isotopy close to this fiber that moves these intersecting points to points in the chart all of the form $\{(ia_i,0,\cdots,0)\} \in i\R^n \subset \C^n = \R^{2n}$, indeed, any isotopy $\psi_t$ of $i\R^n$ with time dependent generating vector field $X_t:\R^n \to \R^n$ can be extended to a Hamiltonian isotopy by the formula
  \begin{align*}
    H_t(x,y) = \inner{X_t(y),x}
  \end{align*}
  for $(x,y) \in \R^n \times i \R^n =\C^n$. Now we may shrink to smaller $\delta$ so that near each intersection point with the fiber over $q$ the new Lagrangian is a graph over $D_\delta^n$. This implies that each small piece equals $df$ for some $f: D_\delta^n \to \R$. It is now easy to Hamiltonian isotope $L$ by convexly changing these functions so that each $df$ is constant in a small neighborhood of $q$. We then pick $\delta'\in (0,\delta)$ so that $D_{\delta'}^n$ is inside this neighborhood for all the pieces of $L$.
\end{proof}

\bibliographystyle{plainurl}
\bibliography{Mybib}

\def\cprime{$'$} \def\cprime{$'$}
\begin{thebibliography}{10}

\bibitem{Abou2}
Mohammed Abouzaid.
\newblock A cotangent fibre generates the {F}ukaya category.
\newblock {\em Adv. Math.}, 228(2):894--939, 2011.
\newblock URL: \url{http://dx.doi.org/10.1016/j.aim.2011.06.007}, \href
  {https://doi.org/10.1016/j.aim.2011.06.007}
  {\path{doi:10.1016/j.aim.2011.06.007}}.

\bibitem{Abou1}
Mohammed Abouzaid.
\newblock Nearby {L}agrangians with vanishing {M}aslov class are homotopy
  equivalent.
\newblock {\em Inventiones mathematicae}, 189:251--313, 2012.
\newblock URL: \url{http://dx.doi.org/10.1007/s00222-011-0365-0}, \href
  {https://doi.org/10.1007/s00222-011-0365-0}
  {\path{doi:10.1007/s00222-011-0365-0}}.

\bibitem{MR2904032}
Mohammed Abouzaid.
\newblock On the wrapped {F}ukaya category and based loops.
\newblock {\em J. Symplectic Geom.}, 10(1):27--79, 2012.
\newblock URL: \url{http://projecteuclid.org/euclid.jsg/1332853049}.

\bibitem{MR2373371}
Jean-Fran\c{c}ois Barraud and Octav Cornea.
\newblock Lagrangian intersections and the {S}erre spectral sequence.
\newblock {\em Ann. of Math. (2)}, 166(3):657--722, 2007.
\newblock URL: \url{http://dx.doi.org/10.4007/annals.2007.166.657}, \href
  {https://doi.org/10.4007/annals.2007.166.657}
  {\path{doi:10.4007/annals.2007.166.657}}.

\bibitem{MR1718076}
J.~Michael Boardman.
\newblock Conditionally convergent spectral sequences.
\newblock In {\em Homotopy invariant algebraic structures ({B}altimore, {MD},
  1998)}, volume 239 of {\em Contemp. Math.}, pages 49--84. Amer. Math. Soc.,
  Providence, RI, 1999.
\newblock \href {https://doi.org/10.1090/conm/239/03597}
  {\path{doi:10.1090/conm/239/03597}}.

\bibitem{FSS}
Kenji Fukaya, Paul Seidel, and Ivan Smith.
\newblock Exact {L}agrangian submanifolds in simply-connected cotangent
  bundles.
\newblock {\em Invent. Math.}, 172(1):1--27, 2008.
\newblock URL: \url{http://dx.doi.org/10.1007/s00222-007-0092-8}, \href
  {https://doi.org/10.1007/s00222-007-0092-8}
  {\path{doi:10.1007/s00222-007-0092-8}}.

\bibitem{Guillermou1}
St\'ephane Guillermou.
\newblock {Quantization of conic {L}agrangian submanifolds of cotangent
  bundles}.
\newblock {\em arXiv:math/1212.5818}, 2012.

\bibitem{Husin_2023}
Axel Husin.
\newblock Licentiate thesis: Local systems and vanishing of {M}aslov class,
  2023.
\newblock URL: \url{https://urn.kb.se/resolve?urn=urn:nbn:se:uu:diva-497748}.

\bibitem{MySympfib}
Thomas Kragh.
\newblock Parametrized ring-spectra and the nearby {L}agrangian conjecture.
\newblock {\em Geom. Topol.}, 17:639--731, 2013.
\newblock Appendix by Mohammed Abouzaid.

\bibitem{Kragh2016}
Thomas Kragh.
\newblock Homotopy equivalence of nearby {L}agrangians and the {S}erre spectral
  sequence.
\newblock {\em Mathematische Annalen}, pages 1--26, 2016.
\newblock URL: \url{http://dx.doi.org/10.1007/s00208-016-1447-5}, \href
  {https://doi.org/10.1007/s00208-016-1447-5}
  {\path{doi:10.1007/s00208-016-1447-5}}.

\bibitem{MR2441780}
Paul Seidel.
\newblock {\em Fukaya categories and {P}icard-{L}efschetz theory}.
\newblock Zurich Lectures in Advanced Mathematics. European Mathematical
  Society (EMS), Z\"urich, 2008.
\newblock URL: \url{http://dx.doi.org/10.4171/063}, \href
  {https://doi.org/10.4171/063} {\path{doi:10.4171/063}}.

\end{thebibliography}

\end{document}